%% file: Vlasov_pointcharge.tex
\documentclass[abstracton, paper=a4, fontsize=11pt,bibliography=totoc, final]{scrartcl}
\usepackage{amsthm, amssymb, amsmath, amsfonts, mathrsfs, dsfont}

\usepackage[utf8]{inputenc}
\usepackage[T1]{fontenc}
\usepackage{lmodern}
\usepackage{graphicx,graphics}
\usepackage{mathtools}
\usepackage[english]{babel}
\usepackage{csquotes}
 \usepackage{epsfig,url}  
 \usepackage{bbm}
\usepackage{a4wide}
\usepackage{enumerate}
\usepackage{verbatim} 
\usepackage{color}
\usepackage{esint}
\usepackage{microtype}
\usepackage{cancel}

\usepackage[T1]{fontenc} 

\usepackage[colorlinks=true, pdfstartview=FitV, linkcolor=blue, citecolor=blue, urlcolor=blue,pagebackref=false]{hyperref}
\usepackage[notcite,notref,color]{showkeys}
\definecolor{labelkey}{gray}{.8}
\definecolor{refkey}{gray}{.8}

\definecolor{darkblue}{rgb}{0,0,0.7} 
\definecolor{darkgreen}{rgb}{0,0.5,0}

\DeclareFontFamily{U}{mathx}{}
\DeclareFontShape{U}{mathx}{m}{n}{<-> mathx10}{}
\DeclareSymbolFont{mathx}{U}{mathx}{m}{n}
\DeclareMathAccent{\widehat}{0}{mathx}{"70}
\DeclareMathAccent{\widecheck}{0}{mathx}{"71}
\renewcommand{\check}{\widecheck}

\newcommand{\ud}{\;\mathrm{d}}

\providecommand{\s}{\mathcal{S}}
\providecommand{\Reac}{\mathcal{R}}

\providecommand{\Id}{\operatorname{Id}}
\providecommand{\ol}{\overline}

\providecommand{\eps}{\varepsilon}

\providecommand{\Vnl}{V_{\operatorname{nl}}}
\providecommand{\nnl}{n_{\operatorname{nl}}}

\providecommand{\supp}{\operatorname{supp}}

\providecommand{\Vmin}{V_{\operatorname{min}}}
\providecommand{\SI}{S_{\operatorname{P}}}
\providecommand{\bS}{\overline{S}_{\operatorname{P}}}

\providecommand{\F}{\mathcal{F}}
\newtheorem{theorem}{Theorem}[section]
\newtheorem{corollary}[theorem]{Corollary}
\newtheorem{proposition}[theorem]{Proposition}
\newtheorem{definition}[theorem]{Definition}
\newtheorem{lemma}[theorem]{Lemma}

\newtheorem{assumption}[theorem]{Assumption}

\newtheorem{remark}[theorem]{Remark}

\numberwithin{equation}{section}
\numberwithin{theorem}{section}

\newcommand{\Reals}{{\mathbb R}}
\newcommand{\R}{{\mathbb R}}
\newcommand{\Complex}{{\mathbb C\hspace{0.05 ex}}}
\newcommand{\Naturals}{{\mathbb N}}

\renewcommand{\varrho}{{\rho}}

\newcommand{\dd}{\, \mathrm{d}}

\newcommand{\1}{{\mathbbm 1}}

\newcommand{\N}{\mathbb N}

\DeclareMathOperator{\dist}{dist}
\newcommand{\loc}{\mathrm{loc}}
\DeclareMathOperator{\dv}{\mathrm{div}}
\newcommand{\step}[1]{\noindent \textit{Step} #1.}

\usepackage[margin=2cm,includefoot]{geometry}

\mathtoolsset{showonlyrefs}
 
\usepackage{authblk}

\setkomafont{date}{\large}

\title{A fast point charge interacting with the screened Vlasov-Poisson system}{}

\author[1]{Richard M. H\"ofer\thanks{richard.hoefer@ur.de}}
\author[2]{ Raphael Winter\thanks{raphael.elias.winter@univie.ac.at}}
\affil[1]{Faculty of Mathematics, University of Regensburg, Germany}
\affil[2]{University of Vienna, Austria}

\begin{document}

\maketitle

 \begin{abstract}
 We consider the long-time behavior of a fast, charged particle interacting with an initially spatially homogeneous background plasma. The background is modeled by the screened Vlasov-Poisson equations, whereas the interaction potential of the point charge is assumed to be smooth. We rigorously prove the validity of the \emph{stopping power theory} in physics, which predicts a decrease of the velocity $V(t)$ of the point charge given by $\dot{V} \sim  -|V|^{-3} V$, a formula that goes back to Bohr (1915). Our result holds for all initial velocities  larger than a threshold value that is larger than the velocity of all background particles and remains valid until (i) the particle slows down to the threshold velocity, or (ii) the time is exponentially long compared to the velocity of the point charge.
 
The long-time behavior of this coupled system is related to the question of Landau damping which has remained open in this setting so far.
 Contrary to other results in nonlinear Landau damping, the long-time behavior of the system is driven by the non-trivial electric field  of the plasma, and the damping only occurs in regions that the point charge has already passed. 
\end{abstract}

\tableofcontents

 \section{Introduction}
\input{1.Introduction}

\section{Outline of the proof of the main result} 
\label{sec:outline}

\input{2.Outline}

\section{Estimates for the Green's function} 
\label{sec:Green}
\input{3.GreensFunction}

\section{Estimates on the characteristics} 
\label{sec:char}
  \input{4.Characteristics}


\section{Straightening the characteristics} 
\label{sec:straighening}
 \input{5.Straightening2}

\section{Estimate of the direct contribution of the reaction term and the point charge} 
\label{sec:source}

\input{6.Source} 


\section{Error estimates for the friction force}
\label{sec:error.force}

\input{8.ErrorForce}

\section{The linearized friction force}
\label{sec:linear.force}

\input{7.LinearFriction}

\appendix 

\input{appendix.tex}

\section*{Acknowledgements}
R.H. is supported  by the German National Academy of Science Leopoldina, grant LPDS 2020-10.
Moreover, R.H. acknowledges support by the Agence Nationale de la Recherche,  Project BORDS, grant ANR-16-CE40-0027-01
and by the Deutsche Forschungsgemeinschaft (DFG, German Research Foundation) 
through the collaborative research center ``The Mathematics of Emerging Effects'' (CRC 1060, Projekt-ID 211504053) 
and the Hausdorff Center for Mathematics (GZ 2047/1, Projekt-ID 390685813).

R.W. acknowledges support of Université de Lyon through the IDEXLYON Scientific Breakthrough
Project “Particles drifting and propelling in turbulent flows”, and the hospitality of the UMPA
ENS Lyon. Furthermore, R.W. would like to thank
the Isaac Newton Institute for Mathematical Sciences for support and hospitality during the programme
‘‘Frontiers in kinetic theory: connecting microscopic to macroscopic scales - KineCon 2022’’ when work
on this paper was undertaken. This work was supported by EPSRC Grant Number EP/R014604/1.
R.W. acknowledges financial support from the Austrian Science Fund (FWF) project F65.

\bibliographystyle{plain}
\bibliography{StoppingPower}

\end{document}

%% file: 1.Introduction.tex
We consider the screened Vlasov-Poisson equation coupled to the motion of a point charge.  Let $F(t,x,v)$ be a phase space density of the plasma on $\Reals^3 \times \Reals^3$ and $X(t), V(t)\in \Reals^3$ be the position and velocity of the point charge. We are interested in the coupled system
    \begin{gather} 
    \begin{aligned}
  \partial_t F + v \cdot \nabla_x F  + E\cdot  \nabla_v F &= -e_0 \nabla \Phi(x-X(t))\cdot  \nabla_v F, \\
  F(0,x,v) &= \mu(v), 
    \end{aligned}   \\
  \begin{aligned}[t] \label{eq:main}
  \rho[F]&=\int_{\Reals^3} F(x,v)\ud{v},\quad  &E(t,x) &= -\nabla \phi *_x \rho[F], \\
  \dot{X}(t)&=V(t),&    X(0)&=0 , \\
  \dot{V}(t)&=  -\alpha e_0 E(t,X(t)),&  V(0)&= V_0 e_1. 
  \end{aligned} 
  \end{gather} 
Here $\mu(v)$ is a probability density, determining the spatially homogeneous initial datum of the density~$F$. Moreover, the initial velocity of the point charge is $V_0>0$, and oriented in direction of the first coordinate vector $e_1$, without loss of generality. The parameter $\alpha>0$ is related to the coupling strength, and  $e_0=\pm 1$ distinguishes whether the interaction of the point charge with the background is attractive or repulsive.

We consider the screened Vlasov-Poisson equation, i.e. $\phi(x)$ is the screened Coulomb potential. Moreover, $\Phi$ is a smooth decaying potential. We refer to Assumption~\ref{ass:phi} for details. The screened potential $\phi$ takes into account the shielding of interactions beyond the Debye length. We refer to \cite{bardos_2018,bouchut_1991,boyd_2003} for details on this mechanism. The assumptions on $\Phi$ are made for technical reasons. 
Note that by considering the screened Coulomb potential, we have $\nabla \phi \in L^1(\R^3)$ such that $E$ is well defined for homogeneous $\rho$ and there is no need to subtract a constant as for the unscreened potential

In this paper we rigorously prove that the large-time behavior of the
system~\eqref{eq:main} are governed by a deceleration of the point charge. More precisely, after some initial layer where the self-consistent field approaches a travelling wave solution, we show that for $|V(t)|$ sufficiently large, the friction force experienced by the point charge is given by
\begin{align}\label{eq:FrictionIntro}
    -e_0 E(t,X(t))  \sim -\frac{V(t)}{|V(t)|^3}.
\end{align}
This means that for large initial velocity of the point charge, i.e. $V_0\gg 1$, the particle decelerates on a slow time scale $V_0^3\tau =t$.

The friction force of order $|V(t)|^{-2}$ can be heuristically understood as follows: the swiftly moving point charge induces a perturbation in the spatial density $\rho[F]$ of the plasma. The perturbation will be asymmetric with respect to the direction of motion, since the particle has affected the region behind it for longer than the region ahead of it. For $e_0=1$, i.e. if the charge attracts plasma particles, $\rho[F]$ will be larger behind the moving charge than in front of it, so that $-e_0 E(t,X(t))$ is a friction force. For $e_0=-1$, the argument is analogous. 

The typical size of the perturbation is proportional to the time spent in a region of order one, i.e. of order $|V(t)|^{-1}$.
On the other hand, the force~\eqref{eq:FrictionIntro} acting on the point charge is of order $|V(t)|^{-2}$, and therefore much smaller. This is due to the fact that $E(t,x)$ can be expressed through $\nabla_{e_1} \rho[F(t)]$. As a result of the swift motion of the charged particle, the characteristic length scale along the  direction of motion is stretched by $|V(t)|$, hence  $|\nabla_{e_1} \rho[F(t)]|\sim |V(t)|^{-2}$. Consequently,  very detailed estimates in the vicinity of the point charge are required in order to make \eqref{eq:FrictionIntro} rigorous.

For a more precise description of~\eqref{eq:FrictionIntro}, we proceed as follows: For $t_\ast > 1$ and $V_\ast:=V(t_\ast) \gg 1$, we show that $F(t_\ast,\cdot)$ is close to a travelling wave solution.  More precisely, we write  $F = \mu + f$ and show that for $|x| \ll V(t_*)$, $\rho[f](t_\ast,X(t_\ast) +x) \approx \lim_{t \to \infty} \rho[h_{V_\ast}](t,x)$, where $h_{V_\ast}$ is the solution to the linearized equation in the inertial frame of the point charge, namely
   \begin{align}\label{eq:TravelingWave}
       \partial_s h_{V_\ast} + (v-V_\ast) \cdot \nabla_x h_{V_\ast} - \nabla (\phi*_x \rho[h_{V_\ast}]) \cdot \nabla_v \mu &= - e_0\nabla \Phi(x) \cdot \nabla_v \mu, \qquad h_{V_\ast}(0,\cdot) = 0.
   \end{align}
This traveling wave solution $h_{V_\ast}$ is explicitly computable in Fourier variables and satisfies the friction relation~\eqref{eq:FrictionIntro}.

The linearization~\eqref{eq:TravelingWave} is only valid over time intervals where $V(t)$ can be approximated by a fixed value~$V_*$. Hence~\eqref{eq:TravelingWave} is only valid as a short-time linearization on a timescale much shorter than the timescale on which we observe deceleration of the point charge. This allows us to get precise information on the response of the plasma to the presence of the point charge.

In order to obtain estimates for the equation on the long timescale, we first perform a  long-time linearization.  Here we cannot approximate the velocity of the point charge by a constant and pass to an inertial frame. This is then only done in the short-time linearization that yields \eqref{eq:TravelingWave}.

We show that the perturbation on the background induced by the point charge is (roughly) of order $|V(t)|^{-1}$ near the point charge and decays algebraically in the distance to the point charge in regions that have not (yet) been penetrated by it. 
In order to bootstrap this argument, we show that in regions the point charge has already passed, Landau damping occurs as a result of dispersion. A precise description of Landau damping is necessary already for the long-time well-posedness of~\eqref{eq:main}, which is a byproduct of our result. 

\subsection{Previous results}

The model~\eqref{eq:main}
and the resulting friction force~\eqref{eq:FrictionIntro} are widely studied in plasma physics to describe the stopping of a fast ion passing through plasma, see for instance \cite{boinefrankenheim_nonlinear_1996,grabowski_molecular_2013,peter_energy_1991}. The formula \eqref{eq:FrictionIntro} (with additional logarithmic corrections accounting for Coulomb interactions) goes back to Bohr \cite{bohr_decrease_1915}.

\medskip

The  Vlasov(-Poisson) equation and its large-time behavior (Landau damping) is the subject of numerous important mathematical works over the last decades. The celebrated paper \cite{mouhot_landau_2011} gave a first proof for Landau damping on the torus, while the analysis on the full space goes back to ~\cite{bardos_global_1985,glassey_time_1994,glassey_time_1995}. 
The analysis has since been significantly extended and refined. For small (absolutely continuous) perturbations of the spatially homogeneous plasma described by the screened Vlasov-Poisson equation, this has been first achieved in~\cite{bedrossian_landau_2018, han-kwan_asymptotic_2021}. Recently, sharp estimates for this problem have been proved in~\cite{huang_sharp_2022,huang_sharp_2022a}.
Moreover, in \cite{ionescu_nonlinear_2022}, the results in \cite{bedrossian_landau_2018, han-kwan_asymptotic_2021} have been extended to the Coulomb case for slowly decaying velocity profile $\mu$.

The presence of a point charge gives rise to additional problems  for the qualitative and  quantitative behavior. In particular, the coupled system enjoys much weaker dispersive properties, since the point charge does not disperse at all. Due to these difficulties and its physical relevance, Vlasov-point charge models have been extensively studied in recent years. Most results concern the coupled system~\eqref{eq:main} with $\phi=\Phi$ given by the Coulomb potential. For existence and growth bounds for plasmas with density decaying for $|x|\rightarrow \infty$, we refer to~\cite{caprino_time_2015,caprino_plasma-charge_2010,chen_asymptotic_2015}. Let us point out that the result in~\cite{caprino_time_2015} does not require the spatial density to be integrable. 
For the case of a plasma with finite mass, existence and growth-bounds for solutions can be found in~\cite{crippa_lagrangian_2018,desvillettes_polynomial_2015,marchioro_cauchy_2011}. Global existence of weak solutions has been shown in~\cite{chen_global_2015} for a finite plasma and attractive Coulomb interaction.

The existing results assume some decay of the initial data  $f_0(x,v)$ for $|x|\rightarrow \infty$ in order to handle the problem explained above. To our knowledge, the long-time existence of~\eqref{eq:main} for homogeneous plasmas remained an open problem so far.

Even less is known on the asymptotic behavior of solutions. The publications~\cite{arroyo-rabasa_debye_2021} and~\cite{pausader_stability_2021} investigate the properties of radially symmetric Vlasov-Poisson systems in interaction with a point charge at rest. For the spatially homogeneous plasma with infinite mass and energy, existence and Debye screening for stationary solutions is shown in~\cite{arroyo-rabasa_debye_2021}.
For small initial data with finite mass and finite energy of the plasma density, the result in~ \cite{pausader_stability_2021} gives a precise characterization of the asymptotic scattering. 
A common feature of the asymptotic results in~\cite{bedrossian_landau_2018,han-kwan_asymptotic_2021,pausader_stability_2021} is the decay of the plasma's electric field for~$t\rightarrow \infty$. 

\medskip

The key novelty and difficulty of the present paper is the analysis of the non-trivial long-time behavior of the self-consistent electric field. This poses major difficulties, both for the long-time well-posedness and the long-time behavior  of the system~\eqref{eq:main}.
The system~\eqref{eq:main} combines the difficulties of lack of dispersion of the point charge, and a plasma of infinite mass and energy. This results in the persistence of the electric field
\begin{align}\label{eq:Eremains}
    \|E_f(t,\cdot)\|_{L^\infty(\Reals^3)} = O(1), \quad \text{for $t \gg 1$},
\end{align}
and a linear growth of the mass of the perturbation
\begin{align}\label{eq:rhogrowth}
    \|\rho[f(t)](\cdot)\|_{L^1(\Reals^3)} = O(t), \quad \text{for $t \gg 1$}.
\end{align}

Due to~\eqref{eq:Eremains} and~\eqref{eq:rhogrowth}, the characteristics of the system do not return to free transport or an explicitly computable ODE for $t \gg 1$. Instead, we derive stronger pointwise estimates (cf.~\eqref{apriori:E}) for the perturbation, which are strongly related to the scattering-geometry of plasma particles by the point charge (cf. Definition~\ref{defi:passage.times}). This allows us to separate characteristics which are close to free transport from those which are non-explicit, see Corollary~\ref{cor:psi}.

\subsection{Statement of the main result} 
\begin{assumption}[Potentials] \label{ass:phi}
   In the following, let $\phi$ be the screened Coulomb potential. More precisely, with the convention \eqref{eq:FourierTF} for Fourier transforms,
    \begin{align} \label{def:phi}
	\hat{\phi}(\xi) = \frac{1}{1+|\xi|^2}.
    \end{align}  
    We assume $\Phi$ satisfies $\hat \Phi > 0$ and, for some constants $C_\Phi, c_\Phi > 0$,
    \begin{align}\label{ass:Phi}
        (|\Phi|+|\nabla \Phi| +|\nabla^2 \Phi|)(x)\leq C_\Phi e^{-c_\Phi|x|}.
    \end{align}
\end{assumption}

\begin{assumption}[Radial symmetry and regularity of $\mu$] \label{Ass:Radial}
	Let $\mu \in C^\infty(\Reals^3)$ be a radially symmetric probability density which satisfies 
	\begin{align} \label{est:nabla.mu}
	    |\nabla^k \mu(v)| \leq C_k e^{-c_k |v|}, 
	\end{align}
	for some $c_k>0$, $C_k>0$.
\end{assumption}
We also assume that the initial distribution $\mu$ is monotone.
\begin{assumption}[Monotonicity of $\mu$] \label{Ass:monotone} 
    We assume that $\mu(v)$ satisfies the monotonicity assumption
    \begin{align} \label{eq:monotone}
        \nabla_v \mu(v) = -v \psi(v),
    \end{align}
    for some nonnegative function $\psi\in C^\infty(\Reals^3)$. 
\end{assumption}
\begin{assumption} [Penrose stability] \label{Ass:Penrose}
    We assume $\mu$ satisfies the Penrose stability criterion. More precisely, let $a(z)$ for  $z\in \Complex$, $\Im(z)\leq 0$
    be defined by
    \begin{align} \label{eq:a}
	    a(z) 				&= -\int_0^\infty e^{-ip z} p \widehat{ \mu}(p e_1) \ud{p}.
	\end{align}
	We then assume that $\mu$ is Penrose stable in the sense that there exists a constant $\kappa>0$ such that
	\begin{align} \label{eq:Penrose}	
	    \inf_{\Im(z)\leq 0, \xi \in \Reals^3}| 1 - \hat{\phi}(\xi) a(z)|\geq \kappa.
	\end{align}
\end{assumption}
Sufficient conditions for Penrose stability for screened Coulomb interactions can be found in~\cite{bedrossian_landau_2018}. Since we consider compactly supported densities $\mu(v)$ in this paper, we include a sufficient criterion for this case, which is an adaptation of Proposition~2.7 in~\cite{bedrossian_landau_2018}. The proof is postponed to Appendix~\ref{App:B}.
\begin{proposition}[Penrose criterion, compactly supported functions] \label{prop:penrose}
    Let $\mu$ satisfy Assumption~\ref{Ass:Radial}. Then there exists a constant $\overline{C}>0$, depending only on the constants $C_k,c_k$, such that 
    \begin{align}
        \mu(v)>0, \quad \text{for all } |v|\leq \overline{C},
    \end{align}
    implies the Penrose stability criterion~\eqref{eq:Penrose} for some $\kappa>0$. 
\end{proposition}

We will work with strong solutions $F$ to \eqref{eq:main} in the following function space. 
\begin{definition}
For $k>0$, let $C^1_k(\R^3 \times \R^3)$ be the space given by the norm
    \begin{align} \label{def:C1k}
       \|F\|_{C^1_k(\Reals^3 \times \Reals^3)}:=\| \langle v\rangle^k F \|_{L^\infty}+ \| \langle v\rangle^k \nabla _{x,v} F \|_{L^\infty}.
\end{align}
\end{definition}

Our main result is the following theorem.
	\begin{theorem}  \label{thm:main}
	 Let $\phi, \Phi, \mu$ satisfy Assumptions \ref{ass:phi}--\ref{Ass:Penrose}. Then, there exist $n,A_{\min},A_{\max}>0, \bar V$ depending only on the constants in  Assumptions \ref{ass:phi}--\ref{Ass:Penrose} such that for all $V_0> \bar V$ and all $\alpha > 0$, the following holds true: 
	 
	 \medskip 
	    There exists $T > 0$ and a function $F\in C([0,T);C^1_k(\R^3 \times \R^3)) \cap C^1([0,T);C(\R^3 \times \R^3)) $ for all $k>3$ and $X,V\in C^1([0,T))$ uniquely solving the system \eqref{eq:main} on $(0,T)$. Moreover, if $\mu$ has compact support, then for all $8 V_0^{-\frac 3 5} < t < T$
		\begin{align} \label{eq:estimate.force}
			- \frac{\alpha A_{\max}}{|V(t)|} \leq \dot V(t) \cdot V(t) \leq - \frac{\alpha A_{\min}}{|V(t)|}.
		\end{align}
	and on $(0,T)$
		\begin{align} \label{eq:V(t)}
			(V_0^3 -1 - 3 \alpha A_{\max} t)^{1/3} \leq  |V(t)| \leq (V_0^3 +1 - 3 \alpha A_{\min} t)^{1/3}.
		\end{align}

		Furthermore, at time $T$ at least one of the following conditions holds:
		\begin{enumerate}
		    \item $V(T) = \bar V$,
		    \item $V(T) = \log^n V_0$ ,
		    \item $\supp \mu \cap B^c_{V(T)/5} \neq \emptyset$.
		\end{enumerate}
	\end{theorem}

A few comments are in order on the conditions at time $T$. 
\begin{enumerate}
    \item The threshold velocity $\bar V$ is related to the critical velocity of the point charge which is necessary to study the system perturbatively. 

\item  We are only able to bootstrap the estimates as long as the velocity of the point charge still satisfies a lower bound in terms of its initial velocity. This leads to the second condition, $V(T) = \log^n V_0$. The constant $n$ arising from our proof could be made explicit, but we do not pursue to optimize this constant.

\item The third condition,  $\supp \mu \cap B^c_{V(T)/5} \neq \emptyset$, means that the theorem only makes a statement about the deceleration of the point charge as long as the point charge remains faster than all the background particles at time zero.  We remark that the ball $B_{V(T)/5}$  could be replaced by $B_{\theta V(T)}$ for any fixed $\theta < 1$ and we only choose $\theta = 1/5$ for definiteness.

\end{enumerate}

We also remark that the condition $t > 8 V_0^{-3/5}$ for the validity of \eqref{eq:estimate.force} could be improved but we do not pursue this question either. In the initial layer  the velocity of the point charge does not significantly change anyway. In \eqref{eq:V(t)}, this is expressed by the term $\pm 1 \ll V_0^3$ which could be further improved without difficulty.

\medskip
	
We believe that Theorem~\ref{thm:main} remains valid under more general assumptions. First of all, in Assumption \ref{Ass:Radial}, it suffices to assume, for some $n \in \N$ sufficiently large,  $\mu \in C^n(\R^3)$ and the bound \eqref{est:nabla.mu} for all $k \leq n$. All proofs directly apply in this case. 

Weakening the assumption of compact support of $\mu$ should be possible with the methods of this paper, at least to super-exponential decay of $\mu$. The assumption ensures that the collision time between the point particle and background particles is bounded above. For $\mu$ with unbounded support, analogous estimates on the characteristics as in Section~\ref{sec:char} would  grow exponentially in time for  particles with similar velocity to the point charge. Hence our argument can only work if the fraction of such particles is super-exponentially small.

Due to the corresponding Grönwall estimates, it seems difficult to apply the current method for profiles $\mu$ with only exponential or slower decay. 

We assume $\Phi$ to be non-singular at the origin. An appealing, and likely very challenging problem would be the extension of Theorem~\ref{thm:main} to the case where $\phi=\Phi$ are both given by the (screened) Coulomb potential. The main difficulty for treating the Coulomb singularity for $\Phi$ is the lack of an a priori bound for the exchanged momentum between plasma particles and the point charge in a collision. In particular, the deviation of background particles by the point charge cannot be bounded by $|V|^{-1}$. Moreover, the presence of collisions with very small impact parameter formally leads to a logarithmic correction of the timescale of deceleration of the point charge, known as \emph{Coulomb logarithm} (cf. \cite{peter_energy_1991}). An additional difficulty in treating the full Coulomb potential is due to its slow  decay (for both $\phi$ and $\Phi$), and the slow, logarithmic damping of perturbations. We refer to~\cite{bedrossian_linearized_2022,han-kwan_linearized_2021} for results on the linearized problem, and to~\cite{ionescu_nonlinear_2022} for proof of nonlinear Landau damping around equilibria with very slow decay in velocities. 

For the Vlasov-Poisson equation without a point charge, this has recently been treated in \cite{ionescu_nonlinear_2022}. However, this work crucially assumes slow polynomial decay for $\mu$, which is in conflict with our assumptions on $\mu$.

In the case of a radially symmetric plasma with finite mass and energy and a point charge at rest, a stability analysis has been achieved in~\cite{pausader_stability_2021} through a delicate geometric argument.

The fact that we are only able to treat velocities $V(T) \geq \log^2(V_0)$  is related to errors that grow logarithmically in time.
These problem is also present in \cite{han-kwan_asymptotic_2021} which precludes the treatment of the $2$-$d$ case in \cite{han-kwan_asymptotic_2021}. The removal of these logarithmic errors by using suitable H\"older-type norms has been achieved in \cite{huang_sharp_2022} which appeared after the first version of the present  paper. 
In \cite{huang_sharp_2022a}, the authors also deal with the $2$-$d$ case, and these papers thus open a perspective on removing the constraint $V(T) \geq \log^2(V_0)$ in our result.

More precisely, we make use of the fact that the perturbation induced by the moving point charge disperses in the two-dimensional orthogonal complement to its direction of motion. However, the current techniques fail to show global-in-time well-posedness of the screened Vlasov-Poisson equation in two dimensions due to a logarithmic divergence (see~\cite{han-kwan_asymptotic_2021}).

Another challenge consists in the behavior of system \eqref{eq:main} when $V(t)$ becomes of order one. This seems a very hard problem because of the lack of any small parameter that allows for a linearization. For a large range of physically relevant problems, it seems that there is a small parameter in front of the right-hand side in the first line of~\eqref{eq:main}, which corresponds to the ratio of the so-called effective charge of the ion to the Debye number. If this parameter is small, a linearization is again formally possible (see e.g. \cite{boinefrankenheim_nonlinear_1996,peter_energy_1991}), but we are currently not able to treat this case rigorously.

\subsection{Outline of the rest of the paper}

As indicated above, the main challenge of the analysis of the coupled system \eqref{eq:main}
is to rigorously prove nonlinear Landau damping in this setting. 
Our basic strategy is inspired by \cite{han-kwan_asymptotic_2021} where Landau damping is shown using a  Lagrangian approach for the screened Vlasov-Poisson system in the whole space.
The argument in \cite{han-kwan_asymptotic_2021} roughly proceeds as follows: First, the screened Vlasov-Poisson equation is reformulated as a linear system with a solution-dependent source term. In a second step, estimates for the linear system are shown via Fourier analysis. Finally, the solution-dependent source term is estimated by means of a bootstrap argument and a representation of the solution through characteristics. This last step is accomplished by a careful analysis of the characteristics. More precisely, it is shown that the characteristics can be well-approximated by rectilinear trajectories (`straightening') under the bootstrap assumption.

Such a Lagrangian approach seems particularly suitable  for the system \eqref{eq:main} in order to quantify dispersion, which only occurs after the point charge has passed a region and only acts in the directions orthogonal to the trajectory of the point charge. However, our analysis is much more delicate than the one in~\cite{han-kwan_asymptotic_2021} in several ways. For instance, the point charge induces a perturbation which is large in the $L^1$- and $L^\infty$-norms considered in the bootstrap argument of~\cite{han-kwan_asymptotic_2021}~(cf. \eqref{eq:rhogrowth}). Instead we need to consider a solution-dependent weighted norm adapted to the expected dispersive effects.

Moreover, it is not possible to
globally straighten the characteristics as in \cite{han-kwan_asymptotic_2021}: two background particles with the same initial position but different initial velocities might attain the same position at later time due to the influence of the point charge. The straightening argument therefore only applies to background particles that are not scattered too much  by the point charge.

\medskip

The rest of the paper is organized as follows.

In Section \ref{sec:outline}, we collect some key ingredients for the proof of Theorem \ref{thm:main}. The proof itself is given in Section \ref{subsec:proof}. 

In Section \ref{sec:Green}, we provide additional pointwise estimates for the linear equation already studied in~\cite{han-kwan_asymptotic_2021}.

Section \ref{sec:char} is devoted to estimates for the characteristics of the nonlinear equation,
which leads to their straightening in suitable regions in Section \ref{sec:straighening}.

We gather the results of the preceding sections to estimate the source term in the linear formulation (in Section~\ref{sec:source}), as well as the difference of the forces on the point charge corresponding to the system~\eqref{eq:main} and its linearization \eqref{eq:TravelingWave} (in Section~\ref{sec:error.force}).

Finally, in Section  \ref{sec:linear.force}, we show that the force corresponding to the linearized equation, \eqref{eq:TravelingWave} satisfies~\eqref{eq:FrictionIntro}.

In Appendix~\ref{App:B}, we prove Proposition~\ref{prop:penrose}, a Penrose stability criterion for compactly supported velocity distributions. Appendix \ref{App:A} gathers two standard auxiliary Lemmas.

\subsection{Some notation}

To lighten the notation, we will set the constants from Assumptions \ref{ass:phi} and \ref{Ass:Radial} to $1$,
as well as the coupling strength $\alpha$ in  \eqref{eq:main} i.e.,
\begin{align}
    \alpha=C_k = c_k = C_\Phi = c_\Phi = 1.
\end{align}
The value of these constants does not affect any of the proofs.

Throughout the paper we will use the Japanese brackets defined for any $a \in \R^d$, $d > 0$ by
\begin{align}\label{eq:japanese}
	\langle a \rangle &:= \sqrt{1 + |a|^2}.
\end{align}

For $x \in \R^3$, we introduce the orthogonal part $x^\perp \in \R^2$ such that 
\begin{align}
    x=(x_1,x^\perp).
\end{align}

We use the following conventions for the Fourier transform in space and space-time respectively
\begin{align} \label{eq:FourierTF}
	\hat{g}(\xi) = \int_{\Reals^3} e^{-i x\cdot \xi } g(x) \ud{x}, \quad \tilde{h}(\tau,\xi) = \int_{\Reals} 
	 \int_{\Reals^3} e^{-i\tau t} e^{-i x\cdot \xi } h(t,x) \ud{x} \ud{t}	.
\end{align}
For radial functions we will use the convention
\begin{align}\label{eq:radialfunctions}
    g(k) = g(|k|) ,
\end{align}
whenever there is no risk of confusion.

We use $C$ for a constant that may change from line to line and use $A \lesssim B$ for $A \leq C B$.

%% file: 2.Outline.tex
This section contains the proof of Theorem~\ref{thm:main} and sets the structure of the remainder of the paper. We start by giving estimates on the linearized friction force in Subsection~\ref{subsec:linearized}. We then reformulate~\eqref{eq:main} in terms of the Green's function of the linearized problem in Subsection~\ref{subsec:Greensformulation}. In Subsection~\ref{subsec:bootstrapest}, we introduce scattering variables for the interaction of the plasma with the moving charged particle, as well as associated norms. At this point we also state the estimates which are used for the bootstrap argument and are proved in the remaining sections. Finally, in Subsection~\ref{subsec:proof} we give the proof of Theorem~\ref{thm:main}.

\subsection{The force on the point charge for the linearized equation} \label{subsec:linearized}

As outlined in the introduction, the proof of the main result is based on a rigorous linearization of the equation.

The solution $h_{V_\ast}$ to the linearized equation \eqref{eq:TravelingWave}, 
has an integral representation through the space-time Fourier transformation, which gives access to the the force on the point charge corresponding to $h_{V_\ast}$. 
More precisely, we prove the following Proposition.
\begin{proposition} \label{pro:force.linear}
    Recall the function $h_{V_\ast}$ defined in \eqref{eq:TravelingWave}.
    For any $0\neq V_*\in \Reals^3$, the following limit exists and is negative:
    \begin{align}\label{linearForce}
        \lim_{s\rightarrow \infty}  e_0\nabla \phi *_x \rho[h_{V_\ast}(s,\cdot)](0) \cdot V_*< 0.
    \end{align}
    More precisely, there exists a constant $A>0$ and $c>0$ such that
    \begin{align}\label{eq:linearbounds}
       \lim_{s \to \infty} \left|A +|V_*| e_0\nabla \phi *_x \rho[h_{V_\ast}(s,\cdot)](0) \cdot V_* \right| \lesssim e^{-c|V_*|}.
    \end{align}
\end{proposition}
The proof of Proposition \ref{pro:force.linear} will be given in Section~\ref{sec:linear.force}.

Although the short-time linearized equation \eqref{eq:TravelingWave} in the inertial frame of the point charge is very practical for computing this force, we rewrite it in the original coordinate frame to compare with the nonlinear equation \eqref{eq:main}.
It is then convenient to introduce the parameter $R>0$ that will play the role of a large time and consider the 
short-time linearized equation where the charged particle starts at position $X_\ast - R V_\ast$ at time zero and moves with constant velocity $V_\ast$. More precisely, we define 
for $R>0$, $X_\ast \in \R^3$
\begin{align}
    g_{R,X_\ast,V_\ast}(t,x,v) = h_{V_\ast}(t,x-X_\ast + (R-t) V_\ast,v).
\end{align}
Observe that $g_{R,X_\ast,V_\ast}$  solves
\begin{align} \label{eq:g_ast}
\begin{aligned} 
       \partial_s g_{R,X_\ast,V_\ast} + v\cdot  \nabla_x g_{R,X_\ast,V_\ast}- \nabla (\phi*_x \rho[g_{R,X_\ast,V_\ast}]) \cdot \nabla_v \mu &= -e_0\nabla \Phi((x-X_\ast+(R-s)V_\ast) \cdot \nabla_v \mu, \qquad \\
       g_{R,X_\ast,V_\ast}(0,\cdot) &= 0.
\end{aligned}  
\end{align}
Then, on the one hand,
 we have the following relation of the forces
\begin{align}
    \lim_{s\rightarrow \infty} \nabla \phi *_x \rho[h_{V_\ast}(s,\cdot)](0) = \lim_{R\rightarrow \infty} \nabla \phi *_x \rho[g_{R,X_\ast,V_\ast}(R,\cdot)](X_\ast).
\end{align}
On the other hand, for $t_\ast \gg V_{\min}^{-1}$ and $t \approx t_\ast$ such that  $V(t) \approx V_\ast$,  we expect $g_{t_\ast,X(t_\ast),V(t_\ast)}(t,\cdot)$ to be close to the solution $f = F - \mu$ of \eqref{eq:main}.


\subsection{Representation of the solution through a Green's function} \label{subsec:Greensformulation}

Let $F$ be a solution to \eqref{eq:main}. We decompose $F$ as
\begin{align} \label{def:f}
    F(t,x,v) = \mu(v) + f(t,x,v).
\end{align}
Then $f$ solves the following equation
  \begin{equation} \label{eq:coupledVlasov}
  \begin{aligned}
  \partial_t f + v \cdot \nabla_x f+\nabla_x (e_0  \Phi(\cdot-X(t)) -& (\phi*_x \rho[f])) \cdot \nabla_v f   = \nabla_x( (\phi*_x \rho[f])- e_0  \Phi(\cdot-X(t))) \cdot \nabla_v \mu, \\
  \quad  f(0,\cdot) &=f_0, \\
  \dot{X}(t)=V(t), \quad \dot{V}(t)&= e_0  \nabla (\phi *_x \rho[f])(X(t)), \quad X(0)=0, \, V(0)= V_0,
  \end{aligned}
  \end{equation}
  with $f_0 = 0$.

Since $\mu(v)$ is spatially homogeneous, the self consistent force field $E$ in~\eqref{eq:main} can be expressed as
\begin{align}\label{eq:E}
    E(t,x)= - (\nabla \phi *_x \rho[f(t,\cdot)])(x).
\end{align}
We introduce the total force $\ol{E}$, defined by 
\begin{align}
    \ol{E}(t,x) = E(t,x) + e_0 \nabla \Phi(x-X(t)).
\end{align}
Let $X_{s,t}$, $V_{s,t}$ be the characteristics associated to $\ol{E}$. More precisely, for $x,v\in \R^3$, $0\leq s\leq t$,
 \begin{align}
	\frac{d}{ds} X_{s,t}(x,v) &= V_{s,t}(x,v), && X_{t,t}(x,v) = x, \label{def:X} \\
	\frac{d}{ds} V_{s,t}(x,v) &= \ol{E}(s,X_{s,t}(x,v)) , && V_{t,t}(x,v) = v. \label{def:V} 
 \end{align}
 Then, if $f$ is sufficiently regular
\begin{align}
    f(t,x,v) = -\int_0^t E(s,X_{s,t}(x,v)) \cdot \nabla_v \mu( V_{s,t}(x,v))-\int_0^t e_0 \nabla \Phi(X_{s,t}(x,v)-X(s))\cdot  \nabla_v \mu( V_{s,t}(x,v)).
\end{align}
This we can rewrite as 
\begin{align}
    f(t,x,v) = &\int_0^t  (\nabla \phi \ast \rho[f])(s,x-(t-s)v) \cdot \nabla \mu (v)  \dd s  \\
    +&\int_0^t  E(s,x-(t-s)v) \cdot \nabla \mu (v)  \dd s-\int_0^t E(s,X_{s,t}(x,v)) \cdot \nabla_v \mu( V_{s,t}(x,v))\\
    -&\int_0^t e_0 \nabla \Phi(X_{s,t}(x,v)-X(s)) \cdot \nabla_v \mu( V_{s,t}(x,v)),
\end{align}
and therefore the density $\rho[f]$ solves the following integral equation for $t \geq 0$
\begin{align} \label{eq:solution.linear}
 \rho[f](t,x) = \int_0^t \int_{\R^3} (\nabla \phi \ast \rho[f])(s,x-(t-s)v) \cdot \nabla \mu (v) \dd v \dd s + \s (t,x),
\end{align}
where $\s$ for $t \geq 0$ is given by
\begin{align} \label{def:S}
	\s (t,x) &=  \Reac(t,x) + \SI(t,x), \\
	\Reac(t,x)&=\int_0^t \int_{\R^3} \left(E(s,x - (t-s) v) \cdot \nabla_v \mu (v) -  E(s,X_{s,t}(x,v)) \cdot \nabla_v \mu(V_{s,t}(x,v)) \right) \dd v \dd s, \label{def:Reaction} \\
	\SI&=- \int_0^t \int_{\R^3} e_0 \nabla \Phi(X_{s,t}(x,v)-X(s)) \cdot \nabla_v \mu( V_{s,t}(x,v)) \ud{v} \ud{s} \label{def:SCharge}.
\end{align}
We will call $\s$ the source term, $\Reac$ the reaction term, and $\SI$ the contribution of the point charge. 

We can write the solution to the short-time linearization of the equation in \eqref{eq:g_ast} analogously.
Then, for $t \geq 0$, the density $\rho[g_{R,X_\ast,V_\ast}]$ satisfies the equation
\begin{align}
     \rho[g_{R,X_\ast,V_\ast}](t,x) &= \int_0^t \int_{\R^3} (\nabla \phi \ast  \rho[g_{R,X_\ast,V_\ast}])(s,x-(t-s)v) \cdot \nabla \mu (v) \dd v \dd s + S_{R,X_\ast,V_\ast}(t,x), \label{eq:rho.g.S}\\
     S_{R,X_\ast,V_\ast}(t,x) &= -\int_{0}^t \int_{\Reals^3} \nabla \Phi(x-(t-s)v-(X_\ast-(R-s)V_\ast))\cdot  \nabla_v \mu(v) \ud{v} \ud{s}. \label{eq:S_R}
\end{align}
We extend both $\s$ and $S_{R,X_\ast,V_\ast}$ by $0$ for negative times.

Following \cite{han-kwan_asymptotic_2021}, we obtain a representation of $\rho[f]$ and $\rho[g_{R,X_\ast,V_\ast}]$ through a Green`s function $G$ of the form
\begin{align} \label{eq:rhoRep}
    \rho(t,x) &= G *_{t,x} S + S
\end{align}
with $\rho = \rho[f]$ and $S = \mathcal S$, respectively, $\rho = \rho[g_{R,X_\ast,V_\ast}]$ and $S = S_{R,X_\ast,V_\ast}$.
More precisely, corresponding to  \cite[Theorem 2.1]{han-kwan_asymptotic_2021}, we have the following Proposition. In addition to \cite[Theorem 2.1]{han-kwan_asymptotic_2021}, we show also pointwise estimates for $G$ that will be needed later on.
\begin{proposition} \label{pro:G}
	Let $\mu$ satisfy Assumptions \ref{Ass:Radial} and \ref{Ass:Penrose} and let $\phi$ be given as in Assumption  \ref{ass:phi}. Then, for all $S \in L^1_{loc}(\R ;L^2(\R^3) \cap L^\infty(\R^3))$, there exists a unique solution $\rho \in L^1_{\loc}(\R;L^2(\R^3))$ to \eqref{eq:solution.linear}  that can be expressed through \eqref{eq:rhoRep} with a kernel $G:\R \times \R^3 \to \R$ that satisfies $G(t,\cdot) = 0$ for $t < 0$ and, for $t \geq 0$
	\begin{align} \label{eq:G_L^1}
	\|G(t,\cdot)\|_{L^1} \leq \frac{C}{1+t}.
\end{align}

	Moreover, for all $t \geq 0$ and $x \in \R^3$, $G$ satisfies the pointwise estimates
	\begin{align} \label{eq:Gpoint}
			|G(t,x)| &\lesssim  
			 				 \frac{1}{t^4+ |x|^4} , \\
			|\nabla G(t,x)|&\lesssim 
			\frac{1}{t^5+ |x|^5}.\label{eq:nablaGpoint}
	\end{align}
\end{proposition}
This proposition is proved in Section~\ref{subsec:proG}.

\subsection{Bootstrap estimates} \label{subsec:bootstrapest}

The proof of long-time well-posedness of the solution to \eqref{eq:coupledVlasov} relies on local well-posedness and a bootstrap argument.
We start by stating the local well-posedness result, the proof of which is a standard fixed point argument and will be omitted for the sake of conciseness. 
 \begin{theorem}[Local well-posedness] \label{thm:local.wellposedness}
    Let $\mu$, $\phi$ and $\Phi$ satisfy the Assumptions~\ref{ass:phi},\ref{Ass:Penrose} and \ref{ass:Phi} respectively and let $V_0>0$. Further let  $k>3$, and $f_0\in C^1_k(\R^3 \times \R^3)$ (cf.~\eqref{def:C1k}).
    
     Then there exists a time $T_\ast>0$, a function $f\in C([0,T_\ast);C^1_k(\R^3 \times \R^3)) \cap C^1([0,T_\ast);C(\R^3 \times \R^3)) $ and $X,V\in {C^1}([0,T_\ast))$ uniquely solving the system \eqref{eq:coupledVlasov}.  Moreover, $T_\ast=\infty$ or 
        \begin{align} \label{eq:blowup.criterium}
	\limsup_{t \to T_\ast}\| \rho(t,\cdot) \|_{W^{1,\infty}} =  \infty.
	\end{align}
	Furthermore, for all $0 \leq t < T_\ast$
	\begin{align} \label{eq:straight.line}
	    X(t) \in \mathrm{span}\{e_1\}.
	\end{align}
 \end{theorem}
 The relation \eqref{eq:straight.line} follows immediately from symmetry considerations.

The bootstrap argument consists in estimating  $\rho$ by $S$ and vice versa in a weighted $W^{1,\infty}$-norm which is adapted to the scattering of the plasma by a fast moving charged particle. More precisely, we will assign weights that reflect
that we expect the following decay of both $S$ and $\rho$
\begin{itemize}
    \item For regions with a large component orthogonal to the particle trajectory (i.e. $x^\perp \gg 1$): decay in $x^\perp$ because the charged particle never reaches these regions.
    \item For regions in front of the charged particle (i.e. $x_1 > X_1(t)$): decay in terms of the distance $x_1 - X_1(t)$; the charged particle has not yet significantly disturbed these regions.
    \item For regions behind the charged particle (i.e. $x_1 < X_1(t)$): decay in terms of the time passed since $x_1 = X_1(s)$ due to dispersion.
\end{itemize}

In order to formalize this decay, we introduce several parameters that depend on the trajectory of the charged particle. Because this trajectory is a priori only defined for short times, we first introduce the following linear extension.
First, for $T< T_\ast$, we define the minimum of the first component of the velocity of the charged particle
 \begin{align} \label{def:Vmin}
     \Vmin(T) := \min_{t \in [0,T]} V_1(t).
 \end{align}

\begin{definition} \label{def:extension.X}
    Let $0 < T < T_\ast$ where $T_\ast$ is the maximal existence time from Theorem \ref{thm:local.wellposedness} and assume that $\Vmin(T)>0$. Then, we define
    \begin{align} \label{apriori:X} 
    X^T(t)            := 
    \begin{cases}
        X(t) & \qquad \text{in } [0,T], \\
        X(T) + (t-T)V(T) & \qquad \text{in } [T,\infty), \\
       t V_0 & \qquad \text{in } (-\infty,0].
    \end{cases}
\end{align}
\end{definition}

\begin{figure}
    \centering
    \includegraphics[scale=1.15]{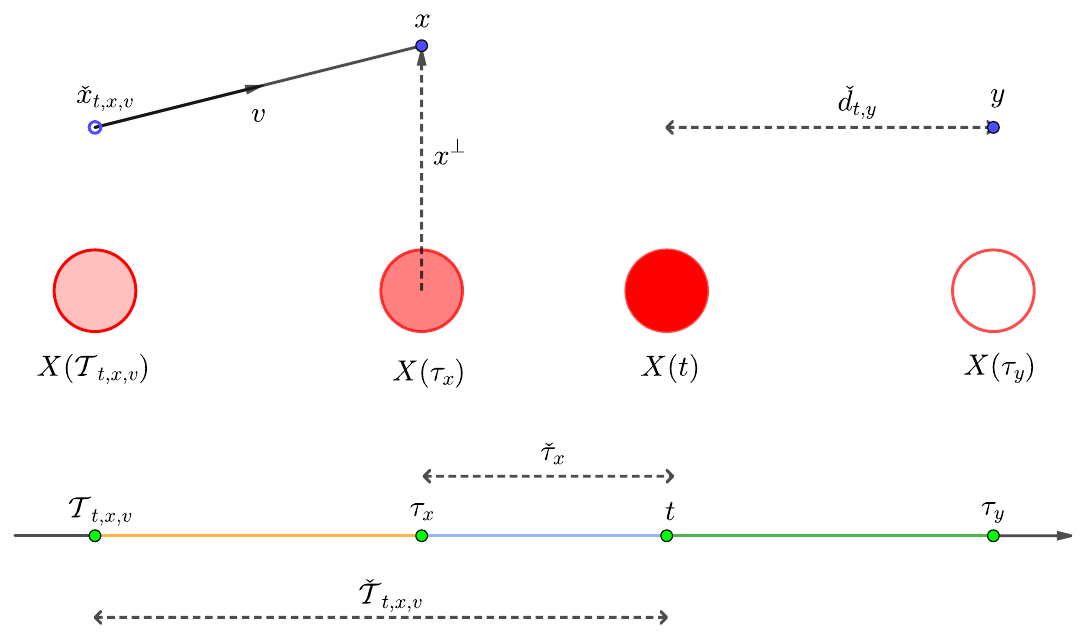}
    \caption{The quantities $\tau_{x}$, $\check \tau_{t,x}$, $\mathcal T_{t,x,v}$, $\check {\mathcal T}_{t,x,v}$, $\check x_{t,x,v}$ and $\check d_{t,x}$}
    \label{fig}
\end{figure}

\begin{definition} \label{defi:passage.times}
    Given $T>0$, and $X^T$ as in Definition \ref{def:extension.X} we introduce the following.
    \begin{enumerate}[(i)]
        \item Let $x \in \R^3$. Then, there exists a unique $\tau_x := \tau_{x_1} \in (-\infty,\infty)$ such that
    \begin{align} \label{def:tau}
        X_1^T(\tau_{x_1}) = x_1,
    \end{align}
    which we call the \emph{passage time} at $x_1$. We also define
    \begin{align} \label{def:taucheck}
        \check \tau_{t,x} = \check \tau_{t,x_1} :=[t - \tau_{x_1}]_+.
    \end{align}
\item   For $x \in  \R^3 ,s \in \R$ we denote the distance to the approaching charged particle with respect to the first component by
    \begin{align} \label{def:d_t,x_1}
	\check d_{s,x} = \check d_{s,x_1} := [x_1 - X_1^T(s)]_+.
    \end{align}
\item For  $\Psi\in L^1_\loc((0,T);W^{1,1}_\loc(\R^3))$ we define the weighted norm
\begin{align} \label{def:norm.Y}
    \|\Psi\|_{Y_T} = \sup_{s\in [0,T],x\in \R^3} |\Psi(s,x)| \langle \check \tau_{s,x}^2+\check d_{s,x}^2+ |x^\perp|^2 \rangle   +|\nabla \Psi(s,x)| \langle \check \tau_{s,x}^3+\check d_{s,x}^3+ |x^\perp|^3\rangle. \quad 
 \end{align}
    \end{enumerate}
\end{definition}

The quantities $\tau_{x}$, $\check \tau_{t,x}$ and $\check d_{s,x}$ are visualized in Figure \ref{fig} (together with further quantities that will be defined later on):
Point $x$ lies \emph{behind} and point $y$ \emph{in front} of the point charge at time $t$ which is located at $X(t)$.
Times $\tau_x$ and $\tau_y$ are then the times when the point charge has passed by $x$ and will pass by $y$, respectively.

On the one hand, the distance $\check d_{t,y}$ is the distance between the point charge at time~$t$ and the first coordinate of the position $y$. Note that by definition $\check d_{t,x} = 0$ as $x$ lies \emph{behind} the point charge at time~$t$.
On the other hand $\check \tau_{t,x} = t - \tau_x$ is the time difference between the present time~$t$ and the  passage time $\tau_x$, i.e. the time  the point charge needs to travel from the  $X(\tau_x)$ to $X(t)$.
Note that by definition $\check \tau_{t,y} = 0$ as $y$ lies \emph{in front} of the point charge at time~$t$.

We point out that in Figure \ref{fig} both time and the first space coordinate is visualized on the horizontal line. However, since the point charge is very fast, the relative times are much smaller than the relative distances, i.e. $\check \tau_{t,x} \ll X_1(t) - X_1(\tau_x)$.

We mark quantities with $\check \cdot$ to indicate that they are differences of quantities inherent to the point charge and external quantities. Consequently,  those quantities appear in the norm $\|\cdot\|_{Y_T}$ which anticipates the expected decay of the perturbation of the background density.
For consistency, one could also replace $x^\perp$ by $\check x^\perp = x^\perp - X^\perp(\tau_x)$ in that norm, but we prefer not to further complicate the notation in this way.

Note that $\tau_{x}$, $\check \tau_{t,x}$ and $\check d_{s,x}$ all implicitly depend on $T$. Since the time $T$ will always be fixed when dealing with these quantities, no confusion will arise from this implicit dependence.

We will for simplicity mostly write $\tau_{x}$, $\check \tau_{t,x}$ and $\check d_{s,x}$ and only use $\tau_{x_1}$, $\check \tau_{t,x_1}$ and $\check d_{s,x_1}$ when we want to emphasize that these quantities only depend on $x_1$.

The quadratic and cubic weight in the definition of $\Psi$ are dictated by the expected dispersion. Indeed, since the fast charged particle effectively creates a disturbance on the whole line, the dispersion only takes place with respect to the orthogonal direction. Since the orthogonal space is $2$-dimensional, this gives rise to $\check \tau_{t,x}^2$ and $\check \tau_{t,x}^3$ in \eqref{def:norm.Y}. The pointwise decay of the Green's function dictates the powers in $\check d_{s,x}$ and $x^\perp$ to be the same.

\medskip

A consequence of Proposition \ref{pro:G} are the following estimates for the linear equation \eqref{eq:rhoRep}.

\begin{corollary}  \label{co:lin.eq}
    There exists a constant $C>0$ with the following property. Let $T>0$, and $X^T$ as in Definition \ref{def:extension.X} and assume in addition that $\Vmin(T) \geq 1$. Then, for all $S \in L^1_{loc}(\R ;L^2(\R^3) \cap L^\infty(\R^3))$ the unique solution $\rho \in L^1_\loc(\R ;L^2(\R^3)$ to \eqref{eq:solution.linear} satisfies
    \begin{align}
        \|\rho\|_{Y_T} \leq C \log^2(2+T) \|S\|_{Y_T}.
    \end{align}
\end{corollary}
This corollary is proved in Section~\ref{subsec:colineq}. 

\medskip

To close a bootstrap argument, we need to estimate $\mathcal S$ (cf. \eqref{def:S}) in terms of $\rho$. This is the content of the following Proposition
which contains two estimates: \eqref{est:Source} gives control of the source term $\mathcal{S}$ on the long time scale $V_0^3 \tau = t$, and the estimate implicitly takes into account the trajectory of the charged particle $X(t)$. 
On the other hand, \eqref{est:force.difference} allows us to approximate the force $E(T,X(T))$ on the charged particle with the short-time linearization of the force. Note that, compared to \eqref{est:Source},  this requires a much finer estimate of the error in the vicinity of $X(T)$.

\begin{proposition} \label{pro:bootstrap}
    There exists $0 <\delta_0< 1,\Vnl > 0,\nnl>0,C>0$  such that the following holds true. Let $T$ and $X^T$ be as in Definition \ref{def:extension.X}.
    Then, if $\Vmin(T) \geq \Vnl$, $\supp \mu \subset B_{\Vmin(T)/5}(0)$, $\Vmin(T)^{-1}<\delta < \delta_0$, $\|\rho[f]\|_{Y_T} < \delta$ and $\log^{\nnl} (2+T) < \delta^{-1}$, then 
    \begin{enumerate}[(i)]
        \item The source $\mathcal{S}$ can be estimated by    \label{case:A}
        \begin{align} \label{est:Source}
        \|\mathcal S\|_{Y_T} \leq C  \Vmin(T)^{-1} + C \delta^{\frac 3 2} .
    \end{align}
        \item If in addition $T \geq 8 \Vmin(T)^{-\frac35}$. Then \label{case:B}
    \begin{align} \label{est:force.difference}
        \lim_{s \to \infty} |(\nabla \phi *\rho[h_{V(T)}])(s,0) + E(T,X(T))| \leq C \delta^{\frac{13}6}. 
    \end{align}
    \end{enumerate}
\end{proposition}
The proof of Proposition~\ref{pro:bootstrap} is the main technical part of the paper. Part \eqref{case:A} will be shown in Subsection~\ref{subsec:PropA}. The proof of Proposition~\ref{pro:bootstrap} \eqref{case:B} can be found in Subsection~\ref{subsec:PropB}. 

The powers of $\delta$ in \eqref{est:Source} and \eqref{est:force.difference} are probably not optimal, neither the restriction $T \geq 8 \Vmin(T)^{-\frac35}$. Indeed, the two terms on the right-hand side of \eqref{est:Source} correspond to the reaction term $\mathcal R$ and the contribution of the point charge $\SI$ respectively (cf. \eqref{def:S}--\eqref{def:SCharge}). The estimate of the contribution of $\SI$ is given in Subsection \ref{sec:SI}. It makes rigorous the heuristics that the point charge only induces a perturbation of order $\Vmin$ since this is the order of the time that it can interact with any given point $x$ where it passes by.
On the other hand,  one could expect $\mathcal R$ to be of order  $\delta^2$. Indeed, the self consistent force field $E$ is bounded by $\|\rho[f]\|_{Y_T} < \delta$ (see Lemma \ref{lem:est.E} below). Additionally, $\mathcal R$ 
is given by the difference obtained by evaluating the same function on the true characteristics $(X_{s,t}(x,v),V_{s,t}(x,v))$ and on their rectilinear counterparts $x - (t-s)v,v$. Since the error in the rectilinear approximation is due to the total force field $\bar E = E + e_0 \nabla \Phi(\cdot - X)$, this yields additional smallness of order $\delta + \Vmin^{-1} \lesssim \delta$.

Estimates on the characteristics will be given in Section \ref{sec:char}. 
These allow us to straighten the characteristics (see Section \ref{sec:straighening}), except on a very small set.

Note that the forces whose difference is estimated in \eqref{est:force.difference} can be rewritten as
\begin{align}
    E(T,X(T)) &= -(\nabla \phi \ast(G \ast \mathcal S + \mathcal S))(t,X(t)), \\
    \lim_{s\rightarrow \infty} \nabla \phi *\rho[h_{V(T)}])(s,0) &= \lim_{R\rightarrow \infty} (\nabla \phi \ast(G \ast S_{R,X(T),V(T)} + S_{R,X(T),V(T)})(R,X(T)). \label{id:hS}
\end{align}
 
Therefore, it is enough to analyze the source term $\mathcal S$ and its counterpart $S_{R,X_\ast,V_\ast}$ for the second part of the above Proposition, too. The estimate \eqref{est:force.difference} cannot be  valid for very small times, since the force $E(T,X(T))$ vanishes at $T = 0$. Instead, we can at best expect the estimate to hold for times that are large compared to the timescale of the point charge, i.e. $T \gg V(0)^{-1}$.

Once we have established estimates for $\nabla \rho$, we immediately obtain estimates for the force field $E$ and its derivative. This convolution estimate is stated in the following Lemma, the proof is elementary and will be skipped. 
\begin{lemma} \label{lem:est.E}
    The force field $E$ given by~\eqref{eq:E} and $\nabla E$ can be estimated by
    \begin{align}\label{eq:Ebyrho}
        |E(t,x)|+|\nabla E(t,x)|\lesssim \frac{1}{ 1+\check \tau_{t,x}^3+\check d_{t,x}^3+ |x^\perp|^3} \left(\sup_{x\in \R^3}|\nabla \rho(t,x)| (1+\check \tau_{t,x}^3+\check d_{t,x}^3+ |x^\perp|^3)  \right) . \qquad
    \end{align}
\end{lemma}

\subsection{Proof of Theorem \ref{thm:main}} \label{subsec:proof}

\begin{proof}[Proof of Theorem \ref{thm:main}]

Let $f$, $(X,V)$ be the solution of~\eqref{eq:coupledVlasov} with maximal time of existence $T_\ast > 0$.

Recall the maximal existence time $T_\ast$ from Theorem~\ref{thm:local.wellposedness} and the constants $\nnl,\delta_0, \Vnl $ from Proposition~\ref{pro:bootstrap}. Define $n = \max\{\nnl,8\}$, $V_\mu := 5 \sup_{v \in \supp \mu} |v|$ and
 $\delta(t) := C_0 \Vmin^{\frac{-n+2}n}(t)$. Then, for all $C_0> 0$ there exists $V_{C_0}$ such that  $\Vmin(t)^{-1} < \delta(t) < \delta_0 $ provided $\Vmin(t) \geq V_{C_0}$. Let $\bar{V} \geq \max\{\Vnl,V_{C_0},1\}$. The constants $ C_0, \bar V$ will be chosen later. 

Consider the time $T>0$ given by
 \begin{align}
    T:= \sup \left\{t\in [0,T_\ast): \|\rho\|_{Y_t} \leq  \delta(t), \Vmin(t) \geq \max\{\bar{V}, V_\mu\}\right\} .
 \end{align}
 Then, by Corollary~\ref{co:lin.eq} and Proposition \ref{pro:bootstrap}, for $0\leq t < \min\{T,e^{\delta^{-1/n}(t)} -2\} $
 \begin{align}
    \|\rho\|_{Y_t} &\leq C \log^2(2+t) \|\mathcal S\|_{Y_t} \leq C \log^2(2+t)  \Vmin(t)^{-1} + C \log^2(2+t)  \|\rho\|^{3/2}_{Y_t} \\
    &\leq  C \delta^{-\frac 2 n}(t)  \Vmin^{-1}(t) + C \delta^{-\frac 2 n}(t)  \delta^{3/2}(t) \\
    &\leq  \frac{C}{C_0} \delta(t) + C  \delta^{5/4}(t).
 \end{align}
 Now pick $C_0 = 4 C $ and choose $\bar{V}$ large enough such that the above estimate implies
 \begin{align}
     \|\rho\|_{Y_t} \leq \frac{\delta(t)}2, \quad \text{for  } t <  \min\{T,e^{\delta^{-1/n}(t)} -2\}.
 \end{align}
 By continuity of $\|\rho\|_{Y_t}$ and the blow-up criterion \eqref{eq:blowup.criterium}, we infer that one of the following statements holds (after possibly further increasing $\bar V$)
 \begin{itemize}
     \item $T \geq e^{\delta^{-1/n}(T)} -2 \geq  e^{\Vmin^{1/(2n)}(T)}$,
     \item $\Vmin(T) = \bar{V}$,
     \item $\Vmin(T) = V_\mu$.
 \end{itemize}

\medskip
It remains to show \eqref{eq:estimate.force}-\eqref{eq:V(t)}, and that $T \geq e^{\Vmin^{1/(2n)}(T)}$ implies $V(T) \leq \log^{3n} V_0$. We first show the latter assuming that \eqref{eq:V(t)} holds.
Indeed, by \eqref{eq:V(t)}, if $T \geq e^{\Vmin^{1/(2n)}(T)}$, then
\begin{align}
    0 \leq \Vmin(T) = V(T) \leq \left(2 V_0^3 - A_{\min} e^{\Vmin^{1/(2n)}(T)}\right)^{1/3},
\end{align}
and thus (after possibly increasing $\bar V$)
\begin{align}
    V(T) \leq (C + 3 \log V_0)^{2n} \leq \log^{3n} V_0 .
\end{align}

It remains to show \eqref{eq:estimate.force}-\eqref{eq:V(t)}. We will first show the  validity of~\eqref{eq:V(t)} for $t\in[0,\min\{8V_0^{-3/5},T\}]$. Then, if $T < 8V_0^{-3/5}$, the proof is complete. Otherwise, we show the validity of \eqref{eq:estimate.force} on $[8 V_0^{-\frac 3 5},T]$ and how that implies \eqref{eq:V(t)} on in $[8 V_0^{-\frac 3 5},T]$ to conclude.

From~\eqref{eq:Ebyrho} we get for all $0 \leq t < T$ (after possibly further increasing $\bar V$)
\begin{align}
            |\dot V(t)| \leq C \delta(t) \leq 1 .
\end{align}
In particular, we can estimate the velocity of the charged particle by
\begin{align} \label{est:V.initially}
    V_0 - t \leq V_1(t) \leq V_0 + t .
\end{align}
This implies the validity of~\eqref{eq:V(t)} for $t\in[0,\min\{8V_0^{-3/5},T\}]$.
Moreover, we infer (after possibly further increasing $\bar V$ and thus $V_0$)
\begin{align}
    t \geq 4 \Vmin^{-3/5}(t) \qquad \text{for all } t \in [8 V_0^{-\frac 3 5},T],
\end{align}
where we first observe that the estimate follows immediately from the definition of $T$ and $\bar V \geq 4$ if $t \geq 1$, and from the first inequality in \eqref{est:V.initially} for $t \in [8 V_0^{-{\frac35}},1]$.
Thus, combining  Propositions \ref{pro:bootstrap} and
 \ref{pro:force.linear} yields \eqref{eq:estimate.force}.
Moreover, \eqref{est:V.initially} and \eqref{eq:estimate.force}
yield that \eqref{eq:V(t)} holds on $[0,T]$.
\end{proof}

%% file: 3.GreensFunction.tex
In this section, we give the proofs of Proposition \ref{pro:G} and Corollary \ref{co:lin.eq}.

\subsection{Proof of Proposition~\ref{pro:G}} \label{subsec:proG}

We start with a simple estimate for the function $a$ defined in~\eqref{eq:a}.
\begin{lemma} \label{lem:a}
	The function $a(r)$  defined in \eqref{eq:a} satisfies $a \in C^\infty (\Reals)$
	and for all $j \in \N$
	\begin{align}
		|a^{(j)}(r) - \frac{d^j}{d^j r}(1/r^2)|	\leq \frac{C_j}{|r|^{3+j}}.
	\end{align}
\end{lemma}
The proof is simply integration by parts and making use of $\hat{\mu}(0)=1$ since $\mu$ is a probability density by assumption.

\begin{proof}[Proof of Proposition~\ref{pro:G}]

   The first part of the assertion, including the $L^1$-estimate \eqref{eq:G_L^1}, is taken from Theorem 2.1 in~\cite{han-kwan_asymptotic_2021}.

    It remains to prove  \eqref{eq:Gpoint} and \eqref{eq:nablaGpoint}. We only present the proof of \eqref{eq:Gpoint}. The proof of~\eqref{eq:nablaGpoint} is similar and will be skipped for the sake of conciseness. 
    
    \step 1 Formulation in space-time Fourier transform
    
	We start with the Fourier representation taken from (c.f. \cite{han-kwan_asymptotic_2021}[Equation (2.9)]) 
	\begin{align}
	\tilde{K}(\tau,\xi) &:= \hat{\phi}(\xi) \int_0^\infty e^{-i\tau t} i \xi \cdot \widehat{\nabla \mu}(t\xi) \ud{t}=\hat{\phi}(\xi) a(\tau/|\xi|), \\
		\tilde{G}(\tau,\xi) &= \frac{\tilde{K}(\tau,\xi)}{1-\tilde{K}(\tau,\xi)},
	\end{align}
	where we used the rotational symmetry of $\mu$.	Now we define $\Psi_\xi(r)$ by
	\begin{align} \label{eq:psi}
		\Psi_\xi(r) = \frac{a(r)}{1-\hat{\phi}(\xi)a(r)}, \quad r\in \Reals,
	\end{align}
	and take $\hat \psi_\xi(p) := (\mathcal F_r^{-1} \Psi_\xi(r))(p)$ the Fourier-Transform in $r$.
	
 Relying on Assumption \ref{Ass:Penrose} and the smoothness of $a$, we observe that $\psi$ as defined in \eqref{eq:psi}  satisfies:
	\begin{align}\label{est:hatPsi}
	|\nabla^j_\xi\frac{d^\ell}{d^\ell p}\hat{\psi}_\xi(p)| \lesssim_{j,l,M}
	\begin{cases}
		\frac{1}{1+|\xi|^{2+j}}\frac{1}{1+|p|^M}, \quad &\text{for any $M \in \Naturals,j\geq 1$}, \\
		\frac{1}{1+|p|^M}\quad &\text{for any $M \in \Naturals,j=0 $}.
	\end{cases}
	\end{align}	
	This allows us to rewrite 
	\begin{align} \label{eq:GConv}
		\hat{G}(t,\xi) &= \hat{\phi}(\xi)|\xi| \hat{\psi}_\xi(t|\xi|).
	\end{align}
	
	
	\step{2} Case $t\leq 1$. 
	
	We take $\eta(\xi)$ a nonnegative bump function which takes value 1 on $B_\frac12$ and vanishes outside of $B_1$. We decompose $\hat{G}$ into
	\begin{align}
		\hat{G}(t,\xi) 
		&= \eta(|\xi|^2) \hat{\phi}(\xi)|\xi| \left( \hat{\psi}_\xi(t|\xi|)-\hat{\psi}_\xi(0)\right)+ \eta(|\xi|^2) \hat{\phi}(\xi)|\xi|  \hat{\psi}_\xi(0)+ (1-\eta(|\xi|^2)) \hat{\phi}(\xi)|\xi| \hat{\psi}_\xi(t|\xi|)\\
						&= R^{(a)}(\xi, t|\xi|)+R^{(b)}(\xi) + R^{(c)}(\xi, t|\xi|).
	\end{align}
	
	We have
	\begin{align}
		 R^{(a)}(\xi, t|\xi|) = \eta(|\xi|^2) \hat{\phi} (\xi) t|\xi|^2 \frac{\hat{\psi}_\xi(t|\xi|)-\hat{\psi}_\xi(0)}{t|\xi|} = \eta(|\xi|^2)\hat{\phi} (\xi) t|\xi|^2 h_\xi( t|\xi|),
	\end{align}
	where $h$ is a smooth function with bounded derivatives to any order. Hence, $\nabla_\xi^4 (R^{(a)}(\xi,t|\xi|)) \in L^1(\R^3)$, uniformly for $t \leq 1$, and thus we can bound 
	\begin{align}
		|\F^{-1}_{\xi} [R^{(a)}(\xi,t|\xi|) ] (x)|\lesssim \frac{1}{1+|x|^4}, \quad \text{for $t\leq 1$}. 
	\end{align}
	Next,
	\begin{align} \label{R^b}
	    R^{(b)}(\xi) =  \eta(|\xi|^2) |\xi| e^{-|\xi|}\hat{\psi}_0(0)+\eta(|\xi|^2) \left( \hat{\phi}(\xi)|\xi|  \hat{\psi}_\xi(0)-  |\xi| e^{-|\xi|}\hat{\psi}_0(0) \right).
	\end{align}
	The Fourier transform of the first term can be explicitly estimated using 
	\begin{align}
	    |\F^{-1}_{\xi} \left(|\xi| e^{-|\xi|} \right)(x)|  \lesssim  |\Delta_x \left(\frac{1}{1+|x|^2}\right)| \lesssim \frac{1}{1+|x|^4}.
	\end{align}
	We then use \eqref{est:hatPsi} and proceed similarly as for $R^{(a)}$ to estimate the second term on the right-hand side of \eqref{R^b} and obtain a total bound
	\begin{align}
		|\F^{-1}_{\xi} [R^{(b)}(\xi) ] (x)| \lesssim \frac{1}{1+|x|^4}. 
	\end{align}
	The contribution of $R^{(c)}(\xi,t|\xi|)$ can be estimated by
	\begin{align}
		|\F^{-1}_{\xi} [R^{(c)}(\xi,t|\xi|) ] (x)| &\lesssim \frac{1}{|x|^4} \|\nabla^4_\xi \left((1-\eta(|\xi|^2))  \hat{\phi}(\xi)|\xi| \hat{\psi}_\xi(t|\xi|)\right) \|_{L^1_\xi}  \lesssim \frac{1}{|x|^4,}
	\end{align}
	due to \eqref{est:hatPsi}. Similarly, we obtain the bound
	\begin{align}
		|\F^{-1}_{\xi} [R^{(c)}(\xi,t|\xi|) ] (x)| &\lesssim  \| \left((1-\eta(|\xi|^2))  \hat{\phi}(\xi)|\xi| \hat{\psi}_\xi(t|\xi|)\right) \|_{L^1_\xi}  \lesssim \frac{1}{t^4},
	\end{align}
	and we obtain the claim, \eqref{eq:Gpoint}, for $t \leq 1$.
	
	\medskip
	
	\step{3} Case $t\geq 1$. 
	
	We rewrite $\hat{G}$ as
	\begin{align}
		\hat{G}(t,\xi) =\frac{\hat{\phi}(\xi)}{t} \zeta_\xi(t|\xi|), \qquad
		\zeta_\xi(p)	= |p| \hat{\psi}_\xi (|p|). 
	\end{align}
	The Fourier transformation of $\zeta$ in $\xi$ can be rewritten as 
	\begin{align}
		|\F^{-1}_\xi \left[ \zeta_\xi (t|\xi| )\right](x)|= \frac{1}{t^3} |\F_\xi \left[ \zeta_{\xi/t} (|\xi| )\right](x/t)|.
	\end{align} 
	Similarly as above, we now decompose the function $\zeta$ further into:
	\begin{align} 
		\zeta_{\xi/t} (|\xi|) &= |\xi| \left(\hat{\psi}_{\xi/t} (|\xi|)-\hat{\psi}_{\xi/t} (0)\right) \eta (|\xi|^2) +|\xi| \hat{\psi}_{\xi/t} (0) \eta (|\xi|^2)+ |\xi|\hat{\psi}_{\xi/t} (|\xi|) (1-\eta (|\xi|^2 )) \\
				&= R^{(a)}_{\xi/t}(|\xi|) + R^{(b)}_{\xi/t} (|\xi|)+ R^{(c)}_{\xi/t} (|\xi|).
	\end{align}
    Arguing as in Step~{2}, we obtain
	\begin{align}
		|\F^{-1}_\xi( \zeta_{\xi/t} (|\xi|))(x)|\lesssim \frac{1}{1+ |x|^4}.
	\end{align}
	We conclude
	\begin{align}
		|G(t,x)| \lesssim  \frac{1}{t} \left|\frac{e^{-|x|}}{|x|} *_x \F_\xi \left[ \zeta_\xi (t|\xi| )\right](x)\right|\lesssim \frac{1}{t^4+|x|^4} \quad t\geq1, x\in \Reals^3,
	\end{align}
	as claimed.
\end{proof}

\subsection{Proof of Corollary~\ref{co:lin.eq}} \label{subsec:colineq} 

We start with a  simple observation on the passage time $\tau_x$ that we will  use frequently.

\begin{lemma} 
Let $T < T_\ast$ from Theorem~\ref{thm:local.wellposedness}. Recall the passage time  $\tau_x$ introduced in Definition~\ref{defi:passage.times}.
    Then, we have  for all $x \in \R^3$
        \begin{align} \label{est.nabla.tau}
        |\nabla_x \tau_x| &\leq \frac 1 {\Vmin(T)}. 
        \end{align}
\end{lemma}
The proof follows immediately from the definition of $\Vmin(T)$ and $\tau_x$.

\begin{proof}[Proof of Corollary \ref{co:lin.eq}]
By definition of $\varrho$ (cf. \eqref{eq:rhoRep}) we have
\begin{align}
    \varrho(t,x)= S(t,x) + \int_0^t \int_{\Reals^3} G(t-s,x-y) S(s,y) \ud{y} \ud{s}.
\end{align}
To prove the assertion, it suffices to estimate the convolution. 
We first consider the integral over the region
     \begin{align}
        B_{t,x} :=\{(s,y) \in (0,t) \times \R^3 : |t-s| + |x-y|\leq \frac12 \max\{1,\check \tau_{t,x}, \check d_{t,x},|x^\perp|\} \}
        \label{est:S.small}.
     \end{align}
We observe that for $s \leq t$, and all $x,y \in \R^3$
     \begin{align}
         |x^\perp| - |y^\perp| \leq |t-s| + |x-y|, &&  \check \tau_{t,x} - \check \tau_{s,y}  \leq |t-s| + |x-y|, &&
         \check d_{t,x} - \check d_{s,y} \leq |t-s| + |x-y|.
     \end{align}
     Indeed, the first inequality follows immediately from the reverse triangle inequality. For the second inequality, we use in addition \eqref{est.nabla.tau} and that $\Vmin \geq 1$ by assumption. For the third inequality, we use in addition that $\check d_{t,x} \leq \check d_{s,x}$ for $s \leq t$ since $\dot X_1 > 0$ by assumption. 
        Thus, 
    \begin{align}
        1 +\check \tau_{s,y}+\check d_{s,y} + |y^\perp|\geq \frac12 \max\{\check \tau_{t,x},\check d_{t,x},|x^\perp|\} \qquad \text{in } B_{t,x}. \label{est:in-B_t,x}
    \end{align}
     Combining this with \eqref{eq:G_L^1} and the definition of $\|\cdot\|_{Y_t}$ (cf.  \eqref{def:norm.Y}) yields
     \begin{align}
         \left| \int_{B_{t,x}} G(t-s,x-y) S(s,y) \ud{s}\ud{y}\right|\lesssim \frac{\log(2+t) \|S\|_{Y_t}}{1+\check \tau_{t,x}^2+\check d_{t,x}^2+ |x^\perp|^2}.
     \end{align}
     It remains to estimate the excluded regions of the integral. By  \eqref{eq:Gpoint} we can estimate
     \begin{align}
         &\left| \int_{B_{t,x}^c} G(t-s,x-y) S(s,y) \ud{s}\ud{y}\right|\\
         \lesssim &\frac{ \|S\|_{Y_t}}{1+\check \tau^2_{t,x}+d^2_{t,x}+|x^\perp|^2} \int_0^t \int_{\Reals^3} \frac{\1_{|y^\perp|\leq \frac12 |x^\perp|}+\1_{\frac12 |x^\perp|\leq |y^\perp| \leq 2 |x^\perp| }+\1_{|y^\perp|\geq 2 |x^\perp|}}{(1+(t-s)^2 + |x-y|^2)(1+|y^\perp|^2)}  \ud{y} \ud{s}\\
         \lesssim &\frac{ \|S\|_{Y_t}}{1+\check \tau^2_{t,x}+d^2_{t,x}+|x^\perp|^2} \left(\int_0^t \frac{2\log(2+ |x^\perp|)}{1+(t-s)+|x^\perp|} \dd s +\int_0^t \int_{\R^2} \frac{1}{(1+s+|y^\perp|)(1+|y^\perp|^2)}dy^\perp ds \right) \\
         \lesssim &\frac{\log^2(2+t) \|S\|_{Y_t}}{1+\check \tau_{t,x}^2+\check d_{t,x}^2+ |x^\perp|^2}.
     \end{align}
     The last inequality follows by separating the cases $|x^\perp|\leq t$, $|x^\perp|\geq t$ for the first term and the regions $|y^\perp|\leq s$, $|y^\perp|\geq s$ for the second term.
     Hence we have 
     \begin{align}
        | \varrho(t,x)|\lesssim \frac{ \log^2(2+t) \|S\|_{Y_t}}{1+\check \tau_{t,x}^2+\check d_{t,x}^2+ |x^\perp|^2}.
     \end{align} 
     For the gradient  $\nabla_x \varrho$ we observe again that by \eqref{eq:G_L^1}
      \begin{align}
         \left| \int_{B_{t,x}} G(t-s,x-y) \nabla S(s,y) \ud{s}\ud{y}\right|\lesssim \frac{ \log(2+t)\|S\|_{Y_t}}{1+\check \tau_{t,x}^3+\check d_{t,x}^3+ |x^\perp|^3}.
     \end{align}
     Again, it remains to estimate the convolution on the excluded region. 
     Integrating by parts yields 
     \begin{align}
          &\left| \int_{B^c_{t,x}} G(t-s,x-y) \nabla S(s,y) \ud{s}\ud{y}\right|
          \leq  \left| \int_{B^c_{t,x}} \nabla G(t-s,x-y)  S(s,y) \ud{s}\ud{y}\right| \\
          &+\left| \int_0^t \int_{|t-s| + |x-y|= \frac12 \max\{1,\check \tau_{t,x},\check d_{t,x},|x^\perp| \}}  G(t-s,x-y)  S(s,y) \ud{y}\ud{s}\right|.
     \end{align}
     The first term on the right-hand side is estimated exactly as above. Moreover, by \eqref{eq:Gpoint} and \eqref{est:in-B_t,x}
     \begin{align}
         & \left| \int_0^t \int_{|t-s| + |x-y|= \frac12 \max\{1,\check \tau_{t,x},\check d_{t,x},|x^\perp| \}}  G(t-s,x-y)  S(s,y) \ud{s}\ud{y}\right| \\
         & \lesssim \int_0^t \int_{|t-s| + |x-y|= \frac12 \max\{1,\check \tau_{t,x},\check d_{t,x},|x^\perp| \}} \frac{\|S\|_{Y_t}}{1+\check \tau_{t,x}^6+\check d_{t,x}^6+ |x^\perp|^6} \dd y\dd s \lesssim \frac{\|S\|_{Y_t}}{1+\check \tau_{t,x}^3+\check d_{t,x}^3+ |x^\perp|^3}.
     \end{align}
     This finishes the proof. 
\end{proof}

%% file: 4.Characteristics.tex
In this and the following sections we will always work under the following bootstrap assumptions: We consider $T>0$, $\delta>0$ as in Proposition \ref{pro:bootstrap}.
More precisely,  recalling the maximal existence time $T_\ast$ from Theorem \ref{thm:local.wellposedness}, and the notation $\Vmin$ and $\|\cdot\|_{Y_T}$ from 
\eqref{def:Vmin} and Definition \ref{defi:passage.times}, respectively, we assume
\begin{align}
    T&<T_\ast,  \label{eq:B1} \tag{B1} \\
    \Vmin^{-1}(T) < \delta &< \min\{\delta_0, \log^{-n}(2+T)\}, \label{eq:B2} \tag{B2}\\
       \|\rho[f]\|_{Y_T} &< \delta, \label{eq:B2.2} \tag{B3} \\
    \supp \mu &\subset  B_{\Vmin(T)/5}(0),\label{eq:B3} \tag{B4}
\end{align}
for some constants $\delta_0 , n> 0$ to be chosen later.
We will refer to \eqref{eq:B1}--\eqref{eq:B3} as the bootstrap assumptions.
In the following, we will often write $\Vmin$ instead of $\Vmin(T)$.

\noeqref{eq:B2} \noeqref{eq:B2.2}

We recall from Lemma \ref{lem:est.E} that $\|\rho[f]\|_{Y_T} < \delta$ implies  for all $0 \leq t \leq T$ and all $x \in \R^3$
\begin{align}\label{apriori:E}
	|E(t,x)| + |\nabla E(t,x)| \leq \frac{\delta }{1 + \check \tau^3_{t,x} + \check d_{t,x}^3  + |x^\perp|^3}.
\end{align}

The objective of this section is to derive estimates for the characteristics defined in \eqref{def:X}--\eqref{def:V}
 which in integrated form read
 \begin{align} \label{eq:charInt}
	X_{s,t}(x,v) &= x-(t-s)v + \int_s^t (\sigma-s) \ol{E}(\sigma,X_{\sigma,t}(x,v)) \ud{\sigma}, \\
	V_{s,t}(x,v) &= v-\int_s^t \ol{E}(\sigma, X_{\sigma,t}(x,v)).
\end{align}
 
 \begin{definition} \label{def:Y.W.tilde}
 We define $\tilde{W}_{s,t}$, $\tilde{Y}_{s,t}$ as the functions given by
 \begin{align}
    V_{s,t}(x,v) &= v + \tilde W_{s,t}(x,v) ,\label{def:Wtilde} \\
    X_{s,t}(x,v) &= x-(t-s)v+\tilde{Y}_{s,t}(x,v). \label{def:Ytilde}
 \end{align}
  \end{definition}

  We are interested in the backwards characteristics, i.e., $0 \leq s \leq t < T$. We distinguish estimates for initial positions $x$ \enquote{in front} and \enquote{behind} the point charge which are characterized by $\check d_{t,x} > 0$ and $t \geq \tau_{x}$, respectively.

We start by giving estimates for the characteristics for background particles which are in front of the point charge at time $t$. This is the easiest case, since those particles stay in front of the point charge along the backwards characteristics.

\subsection{Estimates on the characteristics for particles in front of the point charge} 

\begin{proposition}  \label{pro:char.front}
For all
$\delta_0,n > 0$ sufficiently small and large, respectively,  we have under the bootstrap assumptions \eqref{eq:B1}--\eqref{eq:B3} for all $0 \leq s \leq t \leq T$ and  all $x,v\in \Reals^3$ with $|v|\leq \frac12 \Vmin$ and $x_1>X_1(t)$
    \begin{align}
        |\tilde Y_{s,t}(x,v)| + |\nabla_x \tilde Y_{s,t}(x,v)| &\lesssim \frac{\delta (t-s)}{1+\check d_{t,x}^2+ |x^\perp|^2},  \\
        |\nabla_v \tilde Y_{s,t}(x,v)| &\lesssim  \frac{\delta (t-s)}{1+\check d_{t,x}+ |x^\perp|},\\
        |\tilde W_{s,t}(x,v)| + |\nabla_x \tilde W_{s,t}(x,v)| &\lesssim   \frac{\delta }{1+\check d^2_{t,x}+ |x^\perp|^2}, \\
        |\nabla_v \tilde W_{s,t}(x,v)| &\lesssim  \frac{\delta}{1+\check d_{t,x}+ |x^\perp|}.
    \end{align}
\end{proposition}
\begin{proof}
    Since all the estimates are analogous, we only give the proof of for the estimate of $\tilde{Y}$. By a continuity argument, we have $|\tilde{Y}|\leq \frac12$ for $\delta$ sufficiently small, i.e. for $\delta_0$ in \eqref{eq:B2} sufficiently  small. Therefore,  using $|v|\leq \frac12 \Vmin$ , for all $\sigma \leq t$
    \begin{align}
        \langle \check d_{\sigma,X_{\sigma,t}(x,v)} + |X^\perp_{\sigma,t}(x,v)|\rangle  &= \langle |x_1-(t-\sigma) v_1 -X_1(\sigma)  + (\tilde Y_{\sigma,t})_1(x,v)| + |x^\perp - (t-\sigma)v^\perp + Y^\perp_{\sigma,t}(x,v)| \rangle \\
       & \gtrsim \langle |x_1 - X_1(t)| + \tfrac14(t-\sigma) \Vmin + \tfrac 1 2 |x^\perp| -\tfrac34 \rangle \\
       &\gtrsim \langle  \check d_{t,x} + (t-\sigma) \Vmin  + |x^\perp| \rangle .
    \end{align}
    Starting from the definition~\eqref{def:Ytilde} of $\tilde{Y}$ and using \eqref{apriori:E}, we estimate:
    \begin{align}
        |\tilde{Y}_{s,t}(x,v)| &= \left| \int_{s}^t (\sigma-s) \ol{E}(\sigma,X_{\sigma,t}(x,v)) \ud{\sigma} \right| \\
        &\lesssim \int_{s}^t (\sigma-s) \left(\frac{\delta}{\langle \check d_{t,x}+(t-\sigma) \Vmin+  x^\perp \rangle^3 } + e^{-c\langle \check d_{t,x}+(t-\sigma) \Vmin+|x^\perp|\rangle}\right) \ud{\sigma} \\
        &\lesssim  \frac{\delta (t-s)}{1+\check d_{t,x}^2+ |x^\perp|^2},
    \end{align}
    where we used $\Vmin^{-1} < \delta$ by \eqref{eq:B2}. 
\end{proof}

\subsection{Estimate on the characteristics for particles behind the point charge}

For estimating the characteristics behind the point charge we introduce further notation.

\begin{definition}\label{def:Parameters2}
Let $T$ be as in \eqref{eq:B1}, $t\in [0,T]$, $x \in \R^3$ and $v \in B_{\Vmin/2}(0)$.
\begin{enumerate}
    
    \item Recalling the definition $X^T$ from Definition \ref{def:extension.X}, we define the \emph{collision time} $\mathcal T_{t,x,v} :=\mathcal T_{t,x_1,v_1}$ to be the unique solution to
    \begin{align} \label{def:T}
        X_1^T(\tau) = x_1 - (t-\tau)v_1.
    \end{align}
    We also define 
    \begin{align} \label{eq:Tcheck} 
        \check {\mathcal T}_{t,x,v} :=\check {\mathcal T}_{t,x_1,v_1}:=[t-\mathcal T_{t,x_1,v_1}]_+.
    \end{align}

    \item
    Finally,  we introduce 
\begin{align} \label{def:xcheck}
    \widecheck x_{t,x,v} := x- \widecheck {\mathcal T}_{t,x,v} v,
\end{align}
and we will call $\widecheck x_{t,x,v}^\perp$ the {\emph{impact parameter}} of the collision, following the convention in collisional kinetic theory.
\end{enumerate}
\end{definition}

The quantities $\mathcal T_{t,x,v}$, $ \check {\mathcal  T}_{t,x,v}$ and $\check x_{t,x,v}$ correspond to $\tau_x, \check \tau_{t,x}$ and $x$ (cf. Definition \ref{defi:passage.times}): Instead of considering relations between the point charge and a fixed point in space $x$, these new quantities are the corresponding relations to the straight characteristic $x -(t-s) v$ for a given velocity $v$. The quantities 
$\mathcal T_{t,x,v}$, $ \check {\mathcal  T}_{t,x,v}$ and $\check x_{t,x,v}$ are also visualized in Figure \ref{fig}: The collision time $\mathcal T_{t,x,v}$ is the time where the characteristic \enquote{collides} with the point charge with respect to the first coordinate. This is the time where the distance between the point charge and the characteristic is minimized.  The difference $\check {\mathcal  T}_{t,x,v}$ is the time passed since this collision. We emphasize that  the collision time $\mathcal T_{t,x,v}$ can lie before or after the passage time of $\tau_x$ depending on the sign of $v_1$. Moreover, if $v_1 \ll \Vmin$, then, contrary to the visualization in Figure \ref{fig}, the passage time and collision time are close in relation to the difference to the present time $t$, i.e. $|\tau_x - \mathcal T_{t,x,v}| \ll |t -  \tau_x|$. In particular, $\check \tau_x $ and $\mathcal T_{t,x,v}$ are of the same order (see \eqref{eq:tau.T.comparable} below.

Finally, $\check x_{t,x,v}$
is the  position of the characteristic at  the collision time. In particular $(\check x_{t,x,v})_1 = X_1(\mathcal T_{t,x,v)}$ and therefore we will be mostly interested in the  impact parameter $\check x^\perp_{t,x,v}$.

Note that the collision time and impact parameters are defined with respect to the straight characteristics. These will turn out to be sufficiently good approximations for the collision time and impact parameters for the true characteristics for our purposes.

In order to estimate the error of the backwards characteristics to the straight characteristics for particles in front of the point charge, it is suitable to consider the following error functions $W, Y$. Their definition is inspired by the intuition that the error can be best expressed in terms of the particle positions at the ``collision''.

\begin{definition} \label{def:Y,W}
For $0\leq s\leq t<T$, with $T$ as in \eqref{eq:B1}, and $(x,v) \in \R^3 \times \R^3$ with $\tau_x \leq t$, we define the error functions $Y$ and $W$ by
    \begin{align}
	    W_{s,t}(x-\check {\mathcal T}_{t,x_1,v_1} v,v) &= V_{s,t}(x,v)-v, \\
	    Y_{s,t}(x-\check {\mathcal T}_{t,x_1,v_1} v,v) &= X_{s,t}(x,v) - (x-(t-s)v). 
    \end{align}
\end{definition}

Using \eqref{eq:Ttau}, we infer the representation
\begin{align}
    Y_{s,t}(x,v) = X_{s,t}(x + \check \tau_{t,x} v,v) - (x+ (s - \tau_{x})v),
\end{align}
and hence 
\begin{align} \label{eq:repr.Y}
	Y_{s,t}(x,v) = \int_s^t (\sigma-s) \ol{E}(\sigma ,x+( \sigma - \tau_{x}) v+ Y_{\sigma,t}(x,v)) \ud{\sigma }.
\end{align}

Before we proof estimates for $Y$ and $W$, we give some basic facts regarding the passage time, the impact parameter and the collision time.

\begin{lemma} \label{lem:passage.times}
Recall the quantities $\tau_x$, ${\mathcal T}_{t,x,v}$ introduced in Definition~\ref{defi:passage.times} and Definition~\ref{def:Parameters2}. 
    Then, we have the following identities for all $0 \leq s \leq t \leq T$ with $T$ as in \eqref{eq:B1} and all $x,v \in \R^3$ with $|v| \leq \Vmin/2$ provided $\Vmin(T) \geq 4$:
    \begin{align}
        \tau_x &=  {\mathcal T}_{t,x,0},\quad 
        \check \tau_{t,x} = \check{\mathcal T}_{t,x,0} , \label{eq:Ttau.0}\\
        {\mathcal T}_{t,x,v} &= \tau_{x - \check {\mathcal T}_{t,x,v} v} \quad \text{provided } \check {\mathcal T}_{t,x,v}  > 0. \label{eq:Ttau}
    \end{align}
    Moreover, we can estimate
\begin{align}
 \label{eq:tau.T.comparable}
        \frac 1 2 \check \tau_{t,x} \leq \check {\mathcal T}_{t,x,v} &\leq 2 \check \tau_{t,x},
        \end{align}
    and if $\check {\mathcal T}_{t,x,v}  > 0$
        \begin{align}
        \label{est.nabla_x.T}
    |\nabla_x \check {\mathcal{T}}_{t,x,v}| &\leq \frac 2 {\Vmin}, \\
    \label{est:nabla_v.T}
    |\nabla_ v \check {\mathcal{T}}_{t,x,v}| &\lesssim \frac{\check {\mathcal{T}}_{t,x,v}}{\Vmin }.    
    \end{align}
  Furthermore, we have the lower bound 
    \begin{align} \label{est:tau.T}
           \langle \check \tau_{s,x-(t-s)v}\rangle \gtrsim \langle s - \mathcal T_{t,x,v} \rangle \quad \text{for all } \check {\mathcal T}_{t,x,v} > 0 \quad \text{and } s \geq  \mathcal T_{t,x,v} - 5.
    \end{align}
     Finally, we have for $s\leq t$
   \begin{align} \label{est:d_sigma}
       \langle (\tau_{x}-s)\Vmin +  |x^\perp| \rangle &\lesssim \langle \check d_{s,x-( \tau_{x}-s)v} + |x^\perp - (\tau_{x}-s) v^\perp| \rangle  \qquad &&\text{for } s \leq \tau_{x}, \\
       \label{est:d_sigma.2}
       \langle (\mathcal T_{t,x,v} - s)\Vmin +  |\check x_{t,x,v}^\perp| \rangle &\lesssim\langle \check d_{s,x-(t - s)v} + |x^\perp - (t - s) v^\perp| \rangle   \qquad &&\text{for } s \leq \mathcal T_{t,x,v} \leq t, \qquad  \\
         \label{est:d_sigma.3}
       \langle \check d_{t,x} +  (t-s)\Vmin +  |x^\perp| \rangle &\lesssim\langle \check d_{s,x-(t - s)v} + |x^\perp - (t - s) v^\perp| \rangle \qquad &&\text{for }    \check d_{t,x} > 0.
   \end{align}
\end{lemma}
\begin{proof}
The identities \eqref{eq:Ttau.0}--\eqref{eq:Ttau} follow immediately from the definition of these quantities,
and \eqref{est.nabla.tau}--\eqref{est.nabla_x.T} are a consequence of $\dot X_1 \geq \Vmin \geq 2 |v_1|$.
Estimate \eqref{est:nabla_v.T} follows from the identity \eqref{eq:Ttau}, estimate \eqref{est.nabla.tau} and the chain rule.

For \eqref{eq:tau.T.comparable}, we first observe that $\check \tau_{t,x} = 0$ if and only if $\check { \mathcal T}_{t,x,v} = 0$. Otherwise, \eqref{eq:tau.T.comparable} follows from \eqref{eq:Ttau} and \eqref{est.nabla.tau}.
Regarding \eqref{est:tau.T}, we observe that the estimate trivially holds for $ \mathcal T_{t,x,v} - 5 \leq s \leq \mathcal T_{t,x,v}$. For $s \geq \mathcal T_{t,x,v}$ use once again \eqref{eq:Ttau} and \eqref{est.nabla.tau} to find
\begin{align}
    [s -  \tau_{x-(t-s)v}]_+ \geq  [s - \mathcal T_{t,x,v}]_+ - \frac {|v|}  {\Vmin} |s - \mathcal T_{t,x_1,v_1}| \geq \frac 1 2 (s - \mathcal T_{t,x,v}).
\end{align}
Finally, we turn to \eqref{est:d_sigma}--\eqref{est:d_sigma.3}. Observe that~\eqref{est:d_sigma} follows from~\eqref{est:d_sigma.2} by choosing $t=\tau_x$. For the proof of~\eqref{est:d_sigma.2}, we insert the definition of $\check x_{t,x,v}$ (cf.~\eqref{def:xcheck}) to rewrite
\begin{align}
    x-(t-s)v= \check x_{t,x,v} -(\mathcal T_{t,x,v}-s)v.
\end{align}
Therefore, using the definition of $\mathcal T_{t,x,v}$ and $V_{\min}$,
\begin{align}
     (x-(t-s)v-X(s))_1 = (s - \mathcal T_{t,x,v}) v_1 + X_1(\mathcal T_{t,x,v})-X_1(s) \geq (V_{\min} -v_1)(\mathcal T_{t,x,v} - s).
\end{align}

Since $|v_1| + |v^\perp| \leq \sqrt 2 |v| \leq \frac 1 {\sqrt 2} \Vmin$, this implies
\begin{align}
    \check d_{s,x-(t - s)v} + |x^\perp - (t - s) v^\perp| \geq (\mathcal T_{t,x,v}-s)\Vmin + |\check x_{t,x,v}^\perp|- \frac{1}{\sqrt{2}} (\mathcal T_{t,x,v}-s)\Vmin ,
\end{align}
which proves~\eqref{est:d_sigma.2}.
The estimate~\eqref{est:d_sigma.3} is shown analogously.
\end{proof}

\begin{proposition} \label{prop:char}
For all
$\delta_0,n > 0$ sufficiently small and large, respectively, the following estimates hold   under the bootstrap assumptions \eqref{eq:B1}--\eqref{eq:B3}  for all $0 \leq s \leq t \leq T$, $x,v\in \Reals^3$ such that $|v|\leq \frac12 \Vmin $ and $-\infty < \tau_x \leq t$:
    \begin{align}
        (|Y_{s,t}| + |\nabla_x Y_{s,t}|)(x,v) &\lesssim \delta  \min \left\{ \frac{1}{\langle \check\tau_{s,x}\rangle + \frac{\langle x^\perp \rangle}{\langle |v^\perp|\rangle}}, \frac{ t-s}{ \langle \check\tau_{s,x}\rangle^2  + \frac{\langle x^\perp \rangle^2}{\langle |v^\perp|^2\rangle}} \right\} 
        &&\text{ for } s\geq \tau_x -5\label{eq:YBehind}, \qquad \\
         (|Y_{s,t}| + |\nabla_x Y_{s,t}|)(x,v) &\lesssim \frac{\delta  }{1+\frac{\langle x^\perp \rangle}{\langle v^\perp \rangle} }\left( \frac{ \tau_x -s }{1+\frac{\langle x^\perp \rangle}{\langle v^\perp \rangle}}+ \min\{1,\frac{\check \tau_{t,x} }{1+\frac{\langle x^\perp \rangle}{\langle v^\perp \rangle} }\} \right)  &&\text{ for } s\leq \tau_x, \label{eq:YPassed} \\
     |\nabla_v Y_{s,t}(x,v)| &\lesssim  \log(2+t) \delta  \min \left\{1, \frac{ t-s}{ \langle \check\tau_{s,x}\rangle  + \frac{\langle x^\perp \rangle}{\langle |v^\perp|\rangle}} \right\}  &&\text{ for } s \geq \tau_x-5, \label{eq:nablaYBehind}\\
     |\nabla_v Y_{s,t}(x,v)| &\lesssim \log(2+t) \delta\left( \frac{ \tau_x -s }{1+\frac{\langle x^\perp \rangle}{\langle v^\perp \rangle}}+ \min\{1,\frac{\check \tau_{t,x} }{1+\frac{\langle x^\perp \rangle}{\langle v^\perp \rangle} }\} \right)  &&\text{ for } s\leq \tau_x. \label{eq:nablaYPassed}
\end{align}
Moreover,
\begin{align}
     |W_{s,t}(x,v)|+|\nabla_x W_{s,t}(x,v)| &\lesssim \frac{\delta}{\langle \check\tau_{s,x}\rangle^2 + \frac{\langle x^\perp \rangle^2}{\langle v^\perp\rangle^2}}
        &&\text{ for } s\geq 0, \label{eq:WBehindPassed}  \\
      |\nabla_v W_{s,t}(x,v)| &\lesssim  \frac{\delta }{\langle \check\tau_{s,x}\rangle + \frac{\langle x^\perp \rangle}{\langle v^\perp\rangle}} , &&\text{ for } s\geq 0 .\label{eq:WvBehindPassed} 
\end{align}
\end{proposition}
\begin{proof}
   We observe that for $s \in [\tau_x -5, \tau_x]$, \eqref{eq:YBehind} and \eqref{eq:nablaYBehind} follow from \eqref{eq:YPassed} and \eqref{eq:nablaYPassed}, respectively.
   
   We prove \eqref{eq:YBehind} for $s \geq \tau_x$. For $\delta>0$ sufficiently small, the right-hand side of \eqref{eq:YBehind} is bounded by one. 
   By a standard continuity argument we can therefore use $|Y_{s,t}|+|\nabla_x Y_{s,t}|\leq 1$ for $\tau_x\leq s\leq t $. 
    We use $|v|\leq \frac12 \Vmin$
   and \eqref{est.nabla.tau} to find for all $\sigma \in [s,t]$
   \begin{align}  \label{est.nested.tau}
        1+ \check \tau_{\sigma,x + \check \tau_{\sigma,x} v+ Y_{\sigma,t}(x,v)} \gtrsim 1+\check \tau_{\sigma,x}. 
   \end{align}
   Moreover, $|v|\leq \frac12 \Vmin $ and $|X_1(\sigma) - x_1| \geq \Vmin \check \tau_{\sigma,x}$ implies 
   \begin{align} \label{est:distance.ion.char}
   \begin{aligned}
       |x + \check \tau_{\sigma,x}v - X(\sigma)| &\geq \frac 1 4 |x^\perp + \check \tau_{\sigma,x}v^\perp| +  \frac 1 2  |x_1+ \check \tau_{\sigma,x}v_1 - X_1(\sigma)|   \\
      & \geq \frac  1 4 |x^\perp| + \frac 1 8 \check \tau_{\sigma,x} \Vmin.
      \end{aligned}
   \end{align}
   Resorting to \eqref{eq:repr.Y} and using estimates \eqref{apriori:E}, \eqref{ass:Phi}, \eqref{est.nested.tau} and \eqref{est:distance.ion.char}, we deduce
   \begin{align}
      |Y_{s,t}(x,v)| &= \left|\int_s^t (\sigma - s) \ol{E}(\sigma,x+\check \tau_{\sigma,x}v+ Y_{\sigma,t}(x,v)) \ud{\sigma }\right|\\
      &\lesssim \left|\int_s^t  (\sigma - s)  \left(\frac{\delta }{1 + \check \tau_{\sigma,x}^3  + |x^\perp+\check \tau_{\sigma,x}v^\perp |^3}+ e^{-\frac18 (|x^\perp|+\check \tau_{\sigma,x} \Vmin )} \right) \ud{\sigma }\right|.
   \end{align}
   Observing that for $\tau_x\leq \sigma \leq t$ (by distinguishing the cases $|x^\perp| \geq 2 \check \tau_{\sigma,x}|v^\perp|$ and $|x^\perp| \leq 2 \check \tau_{\sigma,x}|v^\perp|$)
   \begin{align} \label{eq:perplower}
       1 + \check\tau_{\sigma,x}  + |x^\perp+\check \tau_{\sigma,x}v^\perp | \gtrsim 1+\check\tau_{\sigma,x}+ \frac{\langle x^\perp \rangle }{\langle v^\perp \rangle },
   \end{align}
   we obtain 
   \begin{align} \label{est.Y.integral}
      |Y_{s,t}(x,v)| &\lesssim \int_s^t  (\sigma - s) \left(\frac{\delta }{1 + \check\tau^3_{\sigma,x}  + \frac{\langle x^\perp \rangle^3}{\langle v^\perp \rangle^3}} +  e^{-\frac18 (|x^\perp|+\Vmin\check \tau_{\sigma,x})} \right) \dd \sigma.
   \end{align}
   To conclude the estimate \eqref{eq:YBehind} for $|Y|$, we use 
   \begin{align}
       \int_s^t  (\sigma - s) e^{-\frac18 (|x^\perp|+\Vmin\check \tau_{s,x})} \dd \sigma  &\lesssim  e^{-\frac18 (|x^\perp|+\Vmin\check \tau_{s,x})}  \int_s^t (\sigma -s) e^{-\frac18 (\sigma - s) \Vmin} \dd \sigma \\
       &\lesssim \min\left\{\frac{1}{\Vmin^2}, \frac{t-s}{\Vmin}\right\} e^{-\frac18 (|x^\perp|+\Vmin\check \tau_{s,x})} ,
   \end{align}
    similar considerations for the first term in \eqref{est.Y.integral} and we recall that $\Vmin^{-1} \leq \delta$ by \eqref{eq:B2}.
    The estimate of $|\nabla_x Y|$ is analogous. 
   
   For the proof of \eqref{eq:YPassed}, the continuity argument shows $|Y_{s,t}|+|\nabla_x Y_{s,t}|\leq 1 +(\tau_x-s)$. We then split the integral
   \begin{equation}
   \begin{aligned} \label{eq:Ysplit} 
       |Y_{s,t}(x,v)| \leq  &\left|\int_s^{\tau_x}  (\sigma - s)  \ol{E}(\sigma,x+(\sigma -  \tau_{x})v+ Y_{\sigma,t}(x,v)) \ud{\sigma }\right|\\
       + &\left|\int_{\tau_x}^t  (\sigma - s)  \ol{E}(\sigma,x+\check \tau_{\sigma, x}v+ Y_{\sigma,t}(x,v)) \ud{\sigma }\right|.
   \end{aligned}
   \end{equation}
   Arguing similarly as for \eqref{est:d_sigma}, we have for $\sigma \leq \tau_x$
   \begin{align} 
       \langle \check d_{\sigma,x+(\sigma - \tau_{x})v+ Y_{\sigma,t}(x,v)} + |x^\perp + (\sigma - \tau_{x}) v^\perp| \rangle  &\gtrsim  \langle (\tau_{x}-\sigma) \Vmin + |x^\perp|\rangle.
   \end{align}
   Using this, we can bound the first term in \eqref{eq:Ysplit} as
   \begin{align}
        &\left|\int_s^{\tau_x} (\sigma-s) \ol{E}(\sigma,x+(\sigma - \tau_{x})v+ Y_{\sigma,t}(x,v)) \ud{\sigma }\right|\\
        \lesssim &(\tau_x-s) \int_s^{\tau_x} \frac{\delta }{\langle x^\perp\rangle^3 + ((\tau_x-\sigma)\Vmin)^3  } + e^{-\frac18 ( |x^\perp|+(\tau_x-\sigma) \Vmin )} \ud{\sigma} \\
        \lesssim  & \delta \frac{\tau_x-s }{\langle x^\perp \rangle^2 } .
   \end{align}
   Therefore and using again \eqref{est.nested.tau} and \eqref{est:distance.ion.char}, we can bound $|Y|$ by 
   \begin{align}
       |Y_{s,t}(x,v)| \lesssim &\left|\int_{\tau_x}^t \delta \frac{(\tau_x-s) }{1 + \check\tau^3_{\sigma,x}  + |x^\perp+\check \tau_{\sigma,x}v^\perp |^3} \ud{\sigma }\right|+ \left|\int_{\tau_x}^t \delta  \frac{(\sigma-\tau_x)}{1 + \check\tau^3_{\sigma,x}  + |x^\perp+\check \tau_{\sigma,x}v^\perp |^3} \ud{\sigma }\right|\\
      & +\int_{\tau_x}^t  (\sigma-s)e^{-\frac18 (|x^\perp|+ \check \tau_{\sigma,x} \Vmin )} \ud{\sigma} +\delta \frac{\tau_x-s }{\langle x^\perp \rangle^2 }   .
   \end{align}
   Using the inequality \eqref{eq:perplower} as above to bound the remaining integrals yields the desired estimate.

   The remaining inequalities are proved analogously. 
\end{proof}

In the following it will sometimes be convenient to use the estimates above for the functions $\tilde{Y}$, $\tilde{W}$ instead. They satisfy the relations
\begin{align}
    \tilde{W}_{s,t}(x,v) &= W_{s,t}(x-\check {\mathcal{T}}_{t,x,v}v,v) ,\\
    \tilde{Y}_{s,t}(x,v) &= Y_{s,t}(x-\check {\mathcal{T}}_{t,x,v}v,v).
\end{align}

Using Lemma~\ref{lem:passage.times}, we then obtain the following corollary. 
\begin{corollary} \label{co:char.tilde}

    Recall the notation $\check x_{t,x,v}$ from \eqref{def:xcheck}. 
    For all
$\delta_0,n > 0$ sufficiently small and large, respectively, the following estimates hold   under the bootstrap assumptions \eqref{eq:B1}--\eqref{eq:B3}  for all $0 \leq s \leq t \leq T$, $x,v\in \Reals^3$ such that $|v|\leq \frac12 \Vmin $ and $-\infty < \tau_x \leq t$:
\begin{itemize}
    \item If $s \geq  \mathcal{T}_{t,x,v} -5$, 
        \begin{align}
        |\tilde Y_{s,t}(x,v)| + |\nabla_x \tilde Y_{s,t}(x,v)| &\lesssim \delta \min \Biggl\{\frac{1}{\langle s - \mathcal T_{t,x,v}\rangle + \frac{\langle \check x_{t,x,v}^\perp\rangle}{\langle v^\perp \rangle}}, \frac{t-s}{\langle s - \mathcal T_{t,x,v} \rangle^2 + \frac{\langle \check x_{t,x,v}^\perp\rangle^2}{\langle v^\perp \rangle^2}} \Biggr\},  \\
         |\nabla_v \tilde  Y_{s,t}(x,v)| &\lesssim \delta \log(2+t) \min \Biggl\{1, \frac{t-s}{\langle s- \mathcal{T}_{t,x,v} \rangle+\frac{\langle \check x_{t,x,v}^\perp \rangle}{\langle v^\perp \rangle}} \Biggr\} + \check{ \mathcal T}_{t,x,v} |\nabla_x \tilde Y_{s,t}( x,v)|.
        \end{align}
        \item        If  $ s \leq \mathcal{T}_{t,x,v}$, \label{tilde.Y.passed}
        \begin{align}
            |\tilde Y_{s,t}(x,v)| + |\nabla_x \tilde  Y_{s,t}(x,v)|&\lesssim \frac{\delta  }{1+\frac{\langle \check x_{t,x,v}^\perp \rangle}{\langle v^\perp \rangle}} \Biggl( \frac{\mathcal{T}_{t,x,v} -s }{1+\frac{\langle \check x_{t,x,v}^\perp\rangle}{\langle v^\perp \rangle}}+ \min\Biggl\{1,\frac{ \check{\mathcal{T}}_{t,x,v} }{1+\frac{\langle \check x_{t,x,v}^\perp \rangle}{\langle v^\perp \rangle} }\Biggr\} \Biggr)  
 \\
     |\nabla_v \tilde  Y_{s,t}(x,v)| \lesssim \delta \log(2+t) &\Biggl( \frac{\mathcal T_{t,x,v} - s }{1+\frac{\langle \check x_{t,x,v}^\perp \rangle}{\langle v^\perp \rangle}}+ \min\Biggl\{1,\frac{\check {\mathcal T}_{t,x,v}}{1+\frac{\langle \check x_{t,x,v}^\perp \rangle}{\langle v^\perp \rangle}}\Biggr\} \Biggr) + \check{ \mathcal T}_{t,x,v} |\nabla_x \tilde Y_{s,t}( x,v)| .
\end{align}
\item For the function $\tilde  W$ we obtain  for all $0 \leq s \leq t$
\begin{align}
     (|\tilde{W}_{s,t}|+|\nabla_x \tilde{W}_{s,t}|)(x,v) &\lesssim    \frac{\delta}{\langle [s - \mathcal T_{t,x,v}]_+\rangle^2 + \frac{\langle \check x_{t,x,v}^\perp \rangle^2}{\langle v^\perp \rangle^2}},
         \label{eq:tildeWBehindPassed}  \\
      |\nabla_v \tilde{W}_{s,t}(x,v)| &\lesssim  \frac{\delta}{\langle [s - \mathcal T_{t,x,v}]_+\rangle + \frac{\langle \check x_{t,x,v}^\perp \rangle}{\langle v^\perp \rangle}}.  \label{eq:tildeWvBehindPassed} 
\end{align}
\end{itemize}

\end{corollary}

\subsection{Some direct consequences of the error estimates of the characteristics}

As a first consequence of the estimates above, we deduce the following inequalities.

\begin{corollary} \label{cor:char.inequalities}
   For all
$\delta_0,n > 0$ sufficiently small respectively large,  under the bootstrap assumptions \eqref{eq:B1}--\eqref{eq:B3} the following holds true for all $x \in \R^3$, $0 \leq s \leq t \leq T$ and all $v \in B_{\Vmin/2}(0)$ 
    \begin{align}
        |V_{s,t}(x,v) - v| \leq  1,  \qquad \langle V_{s,t}(x,v)\rangle  \geq \frac 1 2 \langle v \rangle, \\ \label{eq:X_s,t.comparable} 
        \langle X_{s,t}(x,v)-X(s) \rangle \geq \frac12 \langle x-(t-s)v-X(s) \rangle,  \\
       \langle \check d_{s,X_{s,t}(x,v)} \rangle \gtrsim \langle \check d_{s,x-(t-s)v}\rangle, \\
       \langle \check \tau_{s,X_{s,t}(x,v)} \rangle \gtrsim \langle \check \tau_{s,x-(t-s)v} \rangle, \\ ||X_{s,t}^\perp(x,v)|  -|x^\perp - (t-s)v^\perp| |  \lesssim  \delta \langle [t \wedge \mathcal T_{t,x,v} -s]_+ \rangle.
    \end{align}
\end{corollary}

As a second consequence, we deduce that the support of $f(t,x,\cdot)$ remains contained in $B_{\Vmin/2}$ under the bootstrap assumptions \eqref{eq:B1}--\eqref{eq:B3}.

\begin{corollary} \label{cor:propagation.support}
    For all
$\delta_0,n > 0$ sufficiently small respectively large,  under the bootstrap assumptions \eqref{eq:B1}--\eqref{eq:B3}  we have for all  $ 0 \leq s \leq t \leq T$, $x,v \in \R^3$
\begin{align}
     |V_{s,t}(x,v)|\geq \frac15 \Vmin,   \qquad \text{for all } |v|\geq \frac14 \Vmin.
\end{align}
In particular, $\supp f(t,x,\cdot) \subset B_{\Vmin/4}(0)$.
\end{corollary}
\begin{proof}
    Assume the contrary, i.e. $|v|\geq \frac14 \Vmin$ and $|V_{s,t}(x,v)|\leq \frac15 \Vmin$. By continuity, there exists $s' \in [s,t]$ and $v' \in \partial {B_{\Vmin/4}(0)}$ such that 
    \begin{align}
        V_{s',t}(x,v)=v', \quad X_{s',t}(x,v)=x', \quad V_{s,t}(x,v) = V_{s,s'}(x',v').
    \end{align}
    In particular we know that
    \begin{align}
        |V_{s,s'}(x',v') - v'| \geq  ||V_{s,s'}(x',v')| - |v'|| \geq \frac 1 {20} \Vmin.
    \end{align}
    However by the corollary above, we have
    \begin{align}
        |V_{s,s'}(x',v')-v'|\leq 1,
    \end{align}
    which is a contradiction for $\Vmin^{-1} < \delta_0$ sufficiently small.
\end{proof}

To future reference, we summarize the decay of the background field in the following lemma.

\begin{lemma} \label{lem:E.char} Under the bootstrap assumptions \eqref{eq:B1}--\eqref{eq:B3} with  $\delta_0,n > 0$ sufficiently small we have for all $0 \leq s \leq t \leq T$ and all $x \in \R^3$, $v \in B_{\Vmin/2}$
      \begin{align} \label{est:mu(V)}
       |\nabla_v \mu (V_{s,t}(x,v))| + |\nabla^2_v \mu (V_{s,t}(x,v))| \lesssim e^{-|v|}.
   \end{align}
   Moreover,
   \begin{enumerate}[(i)] 
       \item \label{it:E.front} For $\check d_{t,x} > 0$
       \begin{align}
             &|E(s,x-(t-s)v)| + |\nabla E(s,x-(t-s)v)| + |E(s,X_{s,t}(x,v))| + |\nabla E(s,X_{s,t}(x,v))| \\
       & \lesssim \frac{\delta }{\langle \check d_{t,x} +  (t - s)V_{\min} + |x^\perp|\rangle^3}.
       \end{align}
       \item \label{it:E.post.collision} For $\check d_{t,x} = 0 $ and $ s \geq \mathcal T_{t,x,v} - 5 $
       \begin{align} 
           &|E(s,x-(t-s)v)| + |\nabla E(s,x-(t-s)v)| +  |E(s,X_{s,t}(x,v))| + |\nabla E(s,X_{s,t}(x,v))| \\
       & \lesssim \frac{\delta}{\langle s - \mathcal T_{t,x,v} \rangle^3 + \langle x^\perp - (t-s) v^\perp \rangle^3} .
       \end{align}
              \item \label{it:E.pre.collision} For $\check d_{t,x} = 0$ and $ s \leq \mathcal T_{t,x,v}$
   \begin{align}
        &|E(s,x-(t-s)v)| + |\nabla E(s,x-(t-s)v)| +  |E(s,X_{s,t}(x,v))| + |\nabla E(s,X_{s,t}(x,v))| \\
       & \lesssim \frac{\delta}{\langle \Vmin (\mathcal T_{t,x,v} - s) \rangle^3 + \langle \check x_{t,x,v}^\perp \rangle^3}.
   \end{align}
   \end{enumerate}
\end{lemma}
\begin{proof}
   The estimate \eqref{est:mu(V)} follows immediately from the decay of $\mu$ from Assumption \ref{Ass:Radial} together with the estimate $|\tilde W_{s,t}(x,v)| \lesssim 1$ due to Proposition \ref{pro:char.front} and Corollary \ref{co:char.tilde}.
   
   The estimates in items \eqref{it:E.front}--\eqref{it:E.pre.collision} for the background field $E$ along the straight characteristics $x-(t-s)v$ follow immediately from the decay of $E$, \eqref{apriori:E}, together with the estimates \eqref{est:d_sigma.3},  \eqref{est:tau.T} and \eqref{est:d_sigma.2}, respectively.
   Finally, along the true characteristics, estimates analogous to \eqref{est:d_sigma.3},  \eqref{est:tau.T} and \eqref{est:d_sigma.2} also hold which can be seen by combining them with the estimates from Corollary \ref{cor:char.inequalities}.
\end{proof}

%% file: 5.Straightening2.tex
As explained in Subsection \ref{subsec:bootstrapest}, a key ingredient for the proof of Proposition \ref{pro:bootstrap} consists of straightening of the characteristics except for a small region which is the purpose of this section.
 As we will see, 
roughly speaking, this straightening is possible in regions where
we both  $|\tilde Y_{s,t}|, |\nabla_v \tilde Y_{s,t}|$ are small compared to $t-s$.
We distinguish four four regions in which we can make use of this. 
In the following, we first give the characterization and necessary estimates in these regions, then prove an abstract result about the possibility of straightening the characteristics and finally applying the abstract result to those four regions.
\begin{enumerate}
    \item 
Due to Proposition \ref{pro:char.front}, this is guaranteed in the region $\check d_{t,x} > 0$.
More precisely, under the assumptions of that proposition, we have 
\begin{equation} \label{est:Y.on.d>0}
\begin{aligned} 
    |\tilde Y_{s,t}(x,v)| 
    &\lesssim  \frac{ \delta (t-s)}{\langle \check d_{t,x} \rangle^2 + \langle  x^\perp\rangle^2},\\
    |\nabla_v \tilde Y_{s,t}(x,v)| &\lesssim   \frac{\delta (t-s)}{\langle \check d_{t,x} \rangle + \langle x^\perp\rangle}.
\end{aligned}
\end{equation} 
\item 
Moreover, Corollary \ref{co:char.tilde} provides sufficient estimates for straightening the characteristics on the set given by $\check \tau_{t,x} > 0$ and $s > \mathcal T_{t,x,v} - 5$. 

More precisely, we have 
\begin{align} \label{est:Y.on.A}
|\tilde Y_{s,t}(x,v)| 
&\lesssim \delta \min  \left\{\frac{t-s  }{\langle s - \mathcal T_{t,x,v} \rangle^2 + \frac{\langle \check x_{t,x,v}^\perp \rangle^2}{\langle v^\perp \rangle^2}}, \frac{t-s }{\check \tau_{t,x}} \right\} , \\ \label{est:nabla.Y.on.A}
    |\nabla_v \tilde Y_{s,t}(x,v)|  &\lesssim \frac{\log(2 +t) \delta  (t-s)}{\langle s - \mathcal T_{t,x,v} \rangle + \frac{\langle \check x_{t,x,v}^\perp \rangle}{\langle v^\perp \rangle}}.
\end{align}

 Here, we have used \eqref{eq:tau.T.comparable} and distinguished the cases $t-s \geq \check {\mathcal T}_{t,x,v}/2 $ and $t-s \leq \check {\mathcal T}_{t,x,v}/2 $ to obtain the second term in the minimum of the right-hand side in \eqref{est:Y.on.A}. Moreover, regarding the estimate \eqref{est:nabla.Y.on.A}, we have split the factor $\check {\mathcal T}_{t,x,v} = (t- s)  + (s - \mathcal T_{t,x,v})$ to estimate the term $\check {\mathcal T}_{t,x,v}  |\nabla_x \tilde Y_{s,t}(x,v)|$.

\item

The straightening is more subtle in the regions where $\check \tau_{t,x} > 0$ and  $s < \mathcal T_{t,x,v} - 1$. Indeed, since the error $\nabla_v \tilde Y$ (due to the term $\check {\mathcal T}_{t,x,v} |\nabla_x \tilde Y|$)   grows linearly in $(\mathcal T_{t,x,v}-s) \check {\mathcal T}_{t,x,v}$, the straightening only works well if this factor is balanced by a sufficiently large impact parameter~$\check x_{t,x,v}^\perp$.

This is the case if the time $\check {\mathcal T}_{t,x,v}$ (or equivalently $\check \tau_{t,x}$) is small compared to $|x^\perp|$ such that the impact parameter $\check x_{t,x,v}^\perp$ is still comparable to $x^\perp$.
We therefore introduce
    \begin{align}\label{def:K}
        K_{s,t,x}:= \{ v : s <  {\mathcal T}_{t,x,v} - 1, |v| < \delta^{-\beta},  \check \tau_{t,x} \langle  v^\perp  \rangle < \langle x^\perp\rangle /4 \},
    \end{align}
for some $\beta >0$.
    
Then, using \eqref{eq:tau.T.comparable}, we observe that on $K_{s,t,x}$ 
\begin{align} \label{est:check.x.K'}
     \langle \check x_{t,x,v}^\perp \rangle  \geq \frac12  \langle x^\perp\rangle.
\end{align}
Hence, Corollary \ref{co:char.tilde} implies (recalling $\Vmin^{-1} \leq \delta$) on $K_{s,t,x}$ 
\begin{equation}
\begin{aligned} \label{est:Y.on.E'}
    |\tilde Y_{s,t}(x,v)| 
    &\lesssim \frac{\delta (t-s)}{1+\frac{\langle  x^\perp\rangle^2 }{\langle v^\perp\rangle^2}}, \\
    |\nabla_v \tilde Y_{s,t}(x,v)| &\lesssim \log(2 + t) \frac{\delta (t-s)}{1+\frac{\langle  x^\perp\rangle }{\langle v^\perp\rangle}} .
\end{aligned}
\end{equation}

\item

If the time $\check {\mathcal T}_{t,x,v}$ is not small compared to $|x^\perp|$, there are still a lot of trajectories with large $\check x_{t,x,v}^\perp$.
To characterize these, we introduce
\begin{align} \label{def:v_ast}
    v_*^\perp(t,x,v_1) = \frac{x^\perp}{\check {\mathcal  T}_{t,x_1,v_1}}.
\end{align}
Then, if $v$ is far from $v^\ast(t,x,v_1)$, $\check x_{t,x,v}^\perp$ will be large.
More precisely, by \eqref{eq:tau.T.comparable}
\begin{align} \label{est:check.x.v_ast}
    |\check x^\perp | = |x^\perp - \check {\mathcal T}_{t,x_1,v_1} v^\perp| = \check {\mathcal T}_{t,x_1,v_1}|v_*^\perp(t,x,v_1) - v^\perp | \geq \frac 1 2  \check \tau_{t,x}|v_*^\perp(t,x,v_1) - v^\perp |.
\end{align}
Therefore, we define for $s< \mathcal T_{t,x_1,v_1} - 1$
\begin{align} \label{def:F}
   F_{s,t,x}:= \{v \in \R^3 :s <  {\mathcal T}_{t,x_1,v_1} - 1,  |v| < \delta^{-\beta}, |v_*^\perp(t,x,v_1) - v^\perp | \check \tau_{t,x} > \sqrt{\mathcal T_{t,x_1,v_1} - s}\}. 
\end{align}
Let us assume that $n$ from the bootstrap assumption \eqref{eq:B2} satisfies $n \geq \frac 1 \beta$, and thus $\log (2 + t) \leq \delta^{-\beta}$.
Then, under the assumptions of Corollary \ref{co:char.tilde}, we find on $F_{s,t,x}$,
\begin{align}\label{est:Y.on.F}
  \begin{aligned} |\tilde Y_{s,t}(x,v)| 
  &\lesssim   \frac{\delta^{1-3\beta}}{\langle \check x_{t,x,v}^\perp \rangle} \left(1 +  \frac{\mathcal T_{t,x_1,v_1} - s}{\langle \check x_{t,x,v}^\perp \rangle} \right) \\
   &\lesssim  \frac{\delta^{1-3\beta}}{\langle \check \tau_{t,x}|v_*^\perp(t,x,v_1) - v^\perp |\rangle} \left(1 +  \frac{\mathcal T_{t,x_1,v_1} - s}{\langle\check \tau_{t,x}|v_*^\perp(t,x,v_1) - v^\perp | \rangle} \right) \\
   & \lesssim \frac{\delta^{1-3\beta}}{\langle \sqrt{\mathcal T_{t,x_1,v_1} - s} \rangle} \left(1 +  \frac{\mathcal T_{t,x_1,v_1} - s}{\langle \sqrt{\mathcal T_{t,x_1,v_1} - s} \rangle} \right) \leq \delta^{1-3\beta} ,
 \\
    |\nabla_v \tilde Y_{s,t}(x,v)|
    & \lesssim \delta^{1-3\beta} \left(1 + \frac{\check{\mathcal T}_{t,x_1,v_1}}{\langle \check \tau_{t,x}|v_*^\perp(t,x,v_1) - v^\perp | \rangle} \right) \left(1 +  \frac{\mathcal T_{t,x_1,v_1} - s}{\langle\check \tau_{t,x}|v_*^\perp(t,x,v_1) - v^\perp | \rangle} \right)  \\
    &\lesssim \delta^{1-3\beta} \left(1 + \frac{\check{\mathcal T}_{t,x_1,v_1}}{\langle \sqrt{\mathcal T_{t,x_1,v_1} - s} \rangle} 
    \right) \frac{\mathcal T_{t,x_1,v_1} - s}{{\langle \sqrt{\mathcal T_{t,x_1,v_1} - s} \rangle} }.
    \end{aligned}
\end{align}
We emphasize the right-hand sides above are bounded by $\delta^{1-3\beta}(t-s)$ since $t-s \geq 1$.

Notice that the  derivative of the average velocity deviation $\frac{\nabla_v \tilde{Y}_{s,t}}{t-s}$ satisfies stronger estimates than the derivative of the deviation $\nabla_v \tilde{W}_{s,t}$. Intuitively, a deviation only significantly affects the average velocity if $\mathcal{T}_{t,x,v}-s$ becomes large. This gain of decay will be crucial to our argument.

\end{enumerate}

\begin{lemma} \label{lem:psi}
    Let $x\in \R^3$ be arbitrary and $0\leq s\leq t$. 
    Suppose there are open sets $\Omega'' \subset \Omega'\subset\Omega\subset \R^3 $ such that
    \begin{enumerate}[(i)]
        \item 
        for some $0<\eta< \frac12$, we have the estimate
        \begin{align}
   \sup_{v\in \Omega} \frac {|\tilde Y_{s,t}(x,v)| } {t-s} &<\eta , \label{eq:OmegaEstimates} \quad    &\sup_{v\in \Omega} \frac {| \nabla_v \tilde Y_{s,t}(x,v)| } {t-s} &\leq \frac 1 2 ,
  \end{align}
        \item and the following inclusions hold
    \begin{align} \label{eq:Omega.inclusions}
         \{v \in \Reals^3 : \dist(v,\Omega'') \leq \eta\} \subset \Omega'  , &&
    \{v \in \Reals^3 : \dist(v,\Omega') \leq  \eta\} \subset \Omega .
   \end{align} 
    \end{enumerate}
    Then there exists an open set $\Omega'' \subset \Omega^*  \subset \Omega$, and a diffeomorphism $\Psi_{s,t}(x,\cdot) \colon \Omega^* \to \Omega'$ such that for all $v \in \Omega^\ast$
    \begin{align} \label{Identity:Psi}
		X_{s,t}(x,\Psi_{s,t}(x,v)) = x - (t-s)v.
	\end{align}
	Moreover, $\Psi$ satisfies the estimates
	\begin{align}
     | \Psi_{s,t}(x,v)  - v| &\leq 2 \frac{|\tilde Y_{s,t}(x,v)|}{t-s} , \label{est:Psi-v}\\
     | \nabla_v \Psi_{s,t}(x,v)  - \Id| &\leq\sup_{w \in \Omega': |w - v| \leq  2 \frac{|\tilde Y_{s,t}(x,v)|}{t-s}} \frac{|\nabla_v \tilde Y_{s,t}(x,w)|}{t-s}.
\end{align}
\end{lemma}
\begin{proof}
Let $\zeta_{s,t,x}(v)$ be the mapping defined by
\begin{align}
    \zeta_{s,t,x}(v) := v - \frac{\tilde Y_{s,t}(x,v)}{t-s}.
\end{align}
With this definition, $X_{s,t}(x,v)$ can be rewritten as
\begin{align}
    X_{s,t}(x,v) = x - (t-s) \zeta_{s,t,x}(v).
\end{align}

Due to the second inequality in \eqref{eq:OmegaEstimates}, $\zeta_{s,t,x}(v)$ is injective on $\Omega'$. Therefore, the function has an inverse $\psi_{s,t}(x,\cdot)$ on the set  $\Omega^* = \zeta(\Omega')$ which satisfies $\Omega^\ast \subset \Omega$ due to the first inequality in \eqref{eq:OmegaEstimates}. Moreover,  for any $w\in \Omega''$ the mapping 
\begin{align*}
    \Gamma: \ol{B}_\eta(w)  &\longrightarrow \ol{B}_\eta(w) \\
            v           &\mapsto w + \frac{\tilde Y_{s,t}(x,v)}{t-s}
\end{align*}
is a contraction and thus there exists $v \in \ol{B}_\eta(v_\ast) \subset \Omega'$ such that $\zeta_{s,t,x}(v) = w$. Therefore  $\Omega'' \subset \Omega^* \subset \Omega$.

By \eqref{eq:OmegaEstimates}, the inverse mapping $\Psi_{s,t}$ satisfies the estimate
\begin{align}
    | \Psi_{s,t}(x,v)  - v| =  | \Psi_{s,t}(x,v)  - \zeta( \Psi_{s,t}(x,v))| \leq \frac{ |\tilde Y_{s,t}(x,\Psi_{s,t}(x,v))|}{t-s} \leq \frac{|\tilde Y_{s,t}(x,v)|}{t-s} + \frac 1 2 | \Psi_{s,t}(x,v)  - v|,
\end{align}
which yields \eqref{est:Psi-v}.
Similarly, we can estimate its derivative in $v$ by 
\begin{align}
     | \nabla_v \Psi_{s,t}(x,v)  - \Id| \lesssim \frac{|\nabla_v \tilde Y_{s,t}(x,\Psi_{s,t}(x,v) )|}{t-s} \leq \sup_{w \in \Omega': |w - v| \leq  2\frac{|\tilde Y_{s,t}(x,v)|}{t-s}} \frac{|\nabla_v \tilde Y_{s,t}(x,w)|}{t-s},
\end{align}
which finishes the proof.
\end{proof}

\begin{corollary} \label{cor:psi}
For all $\beta > 0$,
 under the bootstrap assumptions \eqref{eq:B1}--\eqref{eq:B3} with $\delta_0,n > 0$ sufficiently small respectively large, the following holds true for all $0 \leq s \leq t$ and $x \in \R^3$
\begin{enumerate}
    \item \label{it:G}
{\bf if $\check d_{t,x} > 0$, } then there exists an open set $B_{\Vmin/4} \subset \mathcal G_{s,t,x} \subset B_{\Vmin/2}$ and a diffeomorphism $\Psi_{s,t}(x,\cdot) \colon \mathcal G_{s,t,x} \to B_{\Vmin/3}$
such that
	\begin{align} \label{eq:Psi}
		X_{s,t}(x,\Psi_{s,t}(x,v)) = x - (t-s)v.
	\end{align}
Moreover, $\Psi$ satisfies the estimates
	\begin{align} \label{est:Psi}
		 |\Psi_{s,t}(x,v) - v|  + |\nabla_v \Psi_{s,t}(x,v) - \mathrm{Id}| \lesssim  \frac {\delta} { (1+\check d_{t,x}+ |x^\perp|)}.
	\end{align}
\item
{ \bf if $\check d_{t,x} = 0$: } and
\begin{enumerate}
    \item $s>\mathcal{T}_{t,x_1,v_1}-5$ 	\label{it:A}
\begin{align}
A_{s,t,x} := \{ v \in B_{\Vmin/2}: s> \mathcal T_{t,x,v} -5 \}, \\
A'_{s,t,x} := \{ v \in B_{\Vmin/3}: s> \mathcal T_{t,x,v} -4 \} ,\\
A''_{s,t,x} := \{ v \in B_{\Vmin/4}: s> \mathcal T_{t,x,v} -3 \}.
\end{align}
 Then, if $\tau_x \leq t$, there exists an open set $A''_{s,t,x} \subset \mathcal A_{s,t,x} \subset A_{s,t,x}$
and a diffeomorphism $\Psi_{s,t}(x,\cdot) \colon \mathcal A_{s,t,x} \to A'_{s,t,x}$
such that \eqref{eq:Psi} holds.
Moreover, $\Psi$ satisfies the estimate
	\begin{align} \label{est:PsiA}
		 |\Psi_{s,t}(x,v) - v| + |\nabla_v \Psi_{s,t}(x,v) - \mathrm{Id}|\lesssim \frac{\delta^{1-\beta}}{\langle s - \mathcal T_{t,x,v} \rangle + \frac{\langle \check x_{t,x,v}^\perp \rangle}{\langle v^\perp \rangle}}.
	\end{align}
 \item $s<\mathcal{T}_{t,x_1,v_1}-2$
 \begin{enumerate}[(i)]
     \item 
 \label{it:K}
 Next, 
let $K_{s,t,x}$ be as in \eqref{def:K},
\begin{align}
    K_{s,t,x}' &:= \{ v : s <  {\mathcal T}_{t,x_1,v_1} - 2, |v| < \frac{\delta^{-\beta}} 2, \check \tau_{t,x} \langle v^\perp  \rangle < \langle x^\perp \rangle /5 \}, \\
     K''_{s,t,x}&:= \{ v : s <  {\mathcal T}_{t,x_1,v_1} - 3, |v| <\frac{\delta^{-\beta}} 3, \check \tau_{t,x}  \langle v^\perp  \rangle < \langle x^\perp\rangle /6 \}.
\end{align}
 Then, if $\tau_x \leq t$ and $\langle x^\perp\rangle \delta^\beta \geq \check \tau_{t,x}$, there exists an open set $K''_{s,t,x} \subset \mathcal K_{s,t,x} \subset K_{s,t,x}$ and a diffeomorphism $\Psi_{s,t}(x,\cdot) \colon \mathcal K_{s,t,x} \to K'_{s,t,x}$
such that \eqref{eq:Psi} holds.
Moreover, $\Psi$ satisfies the estimate
	\begin{align} \label{est:PsiK}
	|\Psi_{s,t}(x,v) - v|  +   |\nabla_v \Psi_{s,t}(x,v) - \mathrm{Id}| \lesssim \frac{\delta^{1-2\beta}}{\langle x^\perp\rangle}.
	\end{align}
\item \label{it:F}
Similarly, let $F_{s,t,x}$ be as defined in \eqref{def:F} and recall the definition of $v^\perp_*=v_*^\perp(t,x,v_1)$ (cf. \eqref{def:v_ast}) and define
\begin{align}
   F'_{s,t,x}:= \{v \in   \R^3 : s <  {\mathcal T}_{t,x_1,v_1} - 2, |v| < \delta^{-\beta}/2,  
   |v^\perp - v^\perp_\ast| \check \tau_{t,x} > 2 \sqrt{\mathcal T_{t,x_1,v_1} - s}\}, \\
      F''_{s,t,x}:= \{v \in  \R^3 : s <  {\mathcal T}_{t,x_1,v_1} - 3, |v| < \delta^{-\beta}/3, |v^\perp - v^\perp_\ast| \check \tau_{t,x} > 3\sqrt{\mathcal T_{t,x_1,v_1} - s}\}.
\end{align}
Then, if $\langle x^\perp\rangle \leq \check \tau_{t,x} \delta^{-\beta}$ there exists an open set $F''_{s,t,x} \subset \mathcal F_{s,t,x} \subset F_{s,t,x}$ and a diffeomorphism $\Psi_{s,t}(x,\cdot) \colon \mathcal F_{s,t,x} \to F'_{s,t,x}$ which satisfies \eqref{eq:Psi} and 	\begin{align} \label{est:PsiF}
     |\Psi_{s,t}(x,v) - v| + |\nabla_v \Psi_{s,t}- \mathrm{Id}| &\lesssim  \frac{ \delta^{1-3\beta}}{t-s} \left( 1  + \frac {t-s} {\langle  \check {\tau}_{t,x}| v^\perp - v^\perp_\ast| \rangle}  +  \frac{\check \tau_{t,x} (\mathcal T_{t,x_1,v_1} -s)}{\langle  \check {\tau}_{t,x}| v^\perp - v^\perp_\ast| \rangle^2} \right) .
	\end{align}
	\end{enumerate}
	\end{enumerate}
\end{enumerate}

\end{corollary}

\begin{proof}
We choose $n \geq \beta^{-1}$ and assume also that $\delta_0$ and $n$ are chosen sufficiently small respectively large such that we can apply Proposition \ref{pro:char.front} and Corollary \ref{cor:char.inequalities} and such that in particular the estimates \eqref{est:Y.on.d>0}--\eqref{est:Y.on.F} hold.
We then apply Lemma \ref{lem:psi} as follows.

\emph{Proof of \ref{it:G}:} 
We apply Lemma \ref{lem:psi} first in the case $\check d_{t,x} > 0$ to $\Omega =B_{\Vmin/2}(0), \Omega' =B_{\Vmin/3}(0)$, $\Omega'' =B_{\Vmin/4}(0)$. By \eqref{est:Y.on.d>0}, there is $C>0$ such that we may choose $\eta=C \delta$ to satisfy the assumptions of Lemma \ref{lem:psi} and the first assertion follows.

\emph{Proof of \ref{it:A}:} 
Next, we apply Lemma \ref{lem:psi} to $\Omega = A_{s,t,x}$, $\Omega' = A'_{s,t,x}$, $\Omega'' = A''_{s,t,x}$. By \eqref{est:Y.on.A}, we may choose $\eta= C \frac{\delta}{\langle \check \tau_{t,x}\rangle}$ to satisfy \eqref{eq:OmegaEstimates}.

Using~\eqref{est:nabla_v.T} and~\eqref{eq:tau.T.comparable}, we have for $v,v' \in \R^3$ with $|v' - v| \leq \eta$ 
\begin{align} \label{est:difference.collision.time}
     |\check{\mathcal{T}}_{t,x,v} - \check {\mathcal{T}}_{t,x,v'}| \lesssim \frac{\eta \check \tau_{t,x}}{\Vmin} \lesssim \frac{\delta}{\Vmin},
\end{align}
which guarantees that \eqref{eq:Omega.inclusions} is satisfied.
Combining Lemma \ref{lem:psi} with \eqref{est:Y.on.A}--\eqref{est:nabla.Y.on.A} yields the second assertion.

\emph{Proof of \ref{it:K}:} 
We apply Lemma \ref{lem:psi} to $\Omega = K_{s,t,x}$, $\Omega' = K'_{s,t,x}$, $\Omega'' = K''_{s,t,x}$, with $\eta=C \frac{\delta^{1-\beta}}{\langle x^\perp \rangle}$ for some $C>0$ sufficiently large. Using~\eqref{est:Y.on.E'} we verify~\eqref{eq:OmegaEstimates}. Since $\langle x^\perp \rangle \delta^\beta\geq \check{\tau}_{t,x}$ by assumption, \eqref{est:difference.collision.time} and therefore~\eqref{eq:Omega.inclusions} are satisfied. The estimate then follows from~\eqref{est:Y.on.E'}.

\emph{Proof of \ref{it:F}:}
Finally, we choose $\Omega = F_{s,t,x}$, $\Omega' = F'_{s,t,x}$, $\Omega'' = F''_{s,t,x}$, and set $\eta =  C \frac{\delta^{1-3\beta}}{t-s} \leq  C \frac{\delta^{1-3\beta}}{\langle \check \tau_{t,x}\rangle }$. 
Using~\eqref{est:Y.on.F} we verify that~\eqref{eq:OmegaEstimates} is satisfied.

For the inclusions~\eqref{eq:Omega.inclusions} we observe that for $v,v' \in B_{\Vmin/2}$ with $|v' - v| \leq \eta$, \eqref{def:v_ast}, \eqref{est:nabla_v.T} and \eqref{eq:tau.T.comparable} yield
\begin{align}
    \check \tau_{t,x} |v_\ast^\perp(t,x,v') - {v'}^\perp| &\geq  \check \tau_{t,x} |v_\ast^\perp(t,x,v) - {v}^\perp| -  \eta \check \tau_{t,x} - \frac{|x^\perp| \eta}{\Vmin} \\
    &\geq \check \tau_{t,x} |v_\ast^\perp(t,x,v) - {v}^\perp| - C \delta^{1-3\beta} -  \frac{C \delta^{1-4\beta}}{\Vmin},
\end{align}
where we used the assumption $\langle x^\perp\rangle \leq \check \tau_{t,x} \delta^{-\beta}$ in the last inequality.

Note that by assumption we have $\sqrt{\mathcal T_{t,x_1,v_1} - s} \geq 1$ in $F_{s,t,x}$. 
Hence, for $\delta$ sufficiently small, the last two terms on the right-hand side are smaller than $1$ and~\eqref{eq:Omega.inclusions} follows. 
Combining the assertion of Lemma \ref{lem:psi} with \eqref{est:Y.on.F} yields the desired estimates.
\end{proof}

%% file: 6.Source.tex
In the subsections below, we estimate the reaction term $\mathcal R$ (cf. \eqref{def:Reaction}), which we rewrite as
\begin{align}
    \Reac(t,x) &= R_L(t,x) - R_{NL}(t,x), \quad  \text{where}  \\
    R_L(t,x)&= \int_0^t \int_{\R^3} E(s,x-(t-s)v) \cdot \nabla_v \mu (v)  \\
		 R_{NL}(t,x)&=  \int_0^t \int_{\R^3} E(s,X_{s,t}(x,v)) \cdot \nabla_v \mu (V_{s,t}(x,v)).
\end{align}
We need to estimate both $\Reac$ and $\nabla \Reac$.
The general strategy is as follows. In regions where the change of variables $v \mapsto \Psi_{s,t}(x,v)$ is well-defined and $\Psi_{s,t}(x,v) \approx v,$ we can use cancellations between the linear and nonlinear reaction terms.  Here we rely on the analysis of $\Psi_{s,t}$ in the previous section.
Otherwise, we do not control well the deviations of the straightened characteristics from  the linear characteristics, and we cannot exploit cancellations between the linear and nonlinear reaction terms, $R_L$ and $R_{NL}$. In that case, the desired estimates will follow from ``smallness'' of these regions and from the decay of $\mu$.

\subsection{Estimates for the reaction term \texorpdfstring{$\Reac$}{R}}

\begin{proposition} \label{pro:reaction}
     For all $\gamma \in (0,1)$ there exists $C>0$ such that the following estimate holds under the bootstrap assumptions \eqref{eq:B1}--\eqref{eq:B3} for all  $\delta_0,n > 0$ sufficiently small
    \begin{align} \label{est:Reaction}
        |\Reac(t,x)| \leq C \frac{\delta^{1 + \gamma}  }{1 + \check\tau^2_{t,x} +\check d_{t,x}^2  + |x^\perp|^2}.
    \end{align}
\end{proposition}
\begin{proof}
   
\emph{Step 1: Structure of the proof:}

It suffices to show that there exists $M>0$ such that for all $\beta >0$ (sufficiently small) 
    \begin{align} 
        |\mathcal R(t,x)| \lesssim \frac{\delta^{2 - M\beta}  }{1 + \check\tau^2_{t,x} + \check d_{t,x}^2  + |x^\perp|^2}.
    \end{align}
We use this peculiar reformulation for the sake of analogy of the parameter $\beta$ with the one from Corollary \ref{cor:psi}. 
We emphasize that throughout the proof (implicit) constants may depend on $\beta$. By choosing $n \geq \beta^{-1}$, we can always absorb logarithmic errors in time due to \eqref{eq:B2} by
\begin{align} \label{est:log.beta}
    \log(2+t) \leq \delta^{-\beta}.
\end{align}

We split the proof into three different cases, depending on which of the terms in $\check \tau^2_{t,x}$, $\check d_{t,x}$,  $|x^\perp|$ is dominant, and whether the point charge has already passed $x$, i.e. whether $x_1 \geq X_1(t)$.

    In each of these cases, we will make use of the estimate
    \begin{align}
        |\mathcal R(t,x)| &\leq \int_{G_k} | E(s,x-(t-s)v)| \left|\nabla_v \mu (v) -  \nabla_v \mu(V_{s,t}(x,\Psi_{s,t}(x,v))) \det(\nabla_v \Psi_{s,t}(x,v)) \right| \dd v \dd s  \\
        &+ \int_{B_k} | E(s,x-(t-s)v) \nabla_v \mu (v)| +   |E(s,X_{s,t}(x,v)) \nabla_v \mu (V_{s,t}(x,v))| \dd v \dd s \\
        &=: \int_{G_k} r_d(s,x,v) \dd v \dd  s + \int_{B_k} r_s(s,x,v) \dd v \dd s,
    \end{align}
    where the choice of $G_k, B_k \subset [0,t] \times \R^3$
    depends on the case $k$ under consideration, $k = 1,2,3$,
such that
 the change of variables $\Psi_{t,s}(x,\cdot)$ from Corollary \ref{cor:psi} is well-defined on $G_k^s := \{v : (s,v) \in G_k\}$ and 
    \begin{align} \label{eq:decomposition}
        B_k^s \cup (G_k^s \cap \Psi(G_k^s)) \supset B_{\Vmin/4}(0),
    \end{align}
      where we also denote $B_k^s:=\{v : (v,s) \in B_k\}$. Note that Corollary \ref{cor:propagation.support} together with the bootstrap assumption \ref{eq:B3} implies $\mu(v) = \mu(V_{s,t}(x,v)) = 0$ for all $v \in B^c_{\Vmin/4}(0)$.

In the following we will only rely on the estimates in Lemma \ref{lem:E.char} with squares in the denominator of all the estimates instead of cubes.
This will prove useful for drawing analogies to the estimate of $\nabla \mathcal R$ later on.

\emph{Step 2: The case $\check d_{t,x} > 0$:}

In this case, we choose $G_1^s = \mathcal G$
from Corollary \ref{cor:psi} and $B_1^s = \emptyset$.
  By Corollary \ref{cor:psi}, we have \eqref{eq:decomposition}.
  Combining \eqref{est:mu(V)} with the estimates from Corollary \ref{cor:psi} and Proposition~\ref{pro:char.front}, we infer on $G_{1}$
    \begin{align}
       &|\nabla_v \mu (v) -  \nabla_v \mu(V_{s,t}(x,\Psi_{s,t}(x,v))) \det(\nabla_v \Psi_{s,t}(x,v))| 
       \\ &\lesssim  \left(|\Psi_{s,t}(x,v) - v| + |\tilde W_{s,t}(x,\Psi_{s,t}(x,v))| + |1-\det(\nabla_v \Psi_{s,t}(x,v))| \right)e^{-|v|} 
        \lesssim \delta \frac{e^{-|v|}}{1+\check d_{t,x}+ |x^\perp|}.
    \end{align}
Using now Lemma \ref{lem:E.char} \eqref{it:E.front} yields

    \begin{align}
        \int_{G_{1}} r_d(s,x,v) \dd s \dd v & \lesssim  \frac{\delta^2}{\langle \check d_{t,x} + |x^\perp| \rangle} \int_{\R^3} \int_0^t    \frac{1}{\langle \check d_{t,x} +  (t - s)V_{\min} + |x^\perp|)^2}  e^{-|v|} \dd s \dd v \\
        &\lesssim \frac{\delta^{3}}{\langle \check d_{t,x} + |x^\perp| \rangle^2},
    \end{align}
    where we used that $\Vmin^{-1} \leq \delta$ by the bootstrap assumption \eqref{eq:B2}.
    For future reference, we point out that, had we used that $E$ actually decays with the third power of $\check d_{t,x} + |x^\perp|$,
    we would have gained one power more in the denominator.
    
    \medskip

\emph{Step 3: The case $\check d_{t,x} = 0$ and $\langle x^\perp\rangle  \delta^{\beta} \leq  \check \tau_{t,x}$}

    Note that the  assumption $\langle x^\perp\rangle  \delta^{\beta} \leq  \check \tau_{t,x} $ implies that it suffices to show that
    \begin{align}
         \int_{G_3} r_d(s,x,v) \dd v \dd  s + \int_{B_3} r_s(s,x,v) \dd v \dd s \lesssim \frac{\delta^{2 - M \beta}}{ \check \tau_{t,x}^2},
    \end{align}
for some $M$  independent of $\beta$ (note that there is no Japanese bracket in the denominator).
We write  $G_{3} = G_{3,1} \cup G_{3,2}$ 
with
\begin{align}
    G_{3,1} := \{(s,v) \in [0,t] \times \R^3 : v \in \mathcal F_{s,t,x} \}, \\
            G_{3,2} := \{(s,v) \in [0,t] \times \R^3 : v \in \mathcal A_{s,t,x} \}, \\
    B_{3}:=\{(s,v) \in [0,t]  \times \R^3 :  v \in B_{\Vmin/4} \setminus F''_{s,t,x}, s < \mathcal T_{t,x_1,v_1} - 3 \},
\end{align}
 with the sets $\mathcal F_{s,t,x}, F''_{s,t,x}, \mathcal A_{s,t,x}$ as defined in Corollary \ref{cor:psi}.
    By Corollary \ref{cor:psi},  we have {$\mathcal A_{s,t,x} \cap \Psi_{s,t}(x,\mathcal A_{t,s,x}) \supset \{v \in  B_{\Vmin/4} : s > \mathcal T_{t,x_1,v_1} - 3 \}$} and $\mathcal F_{s,t,x} \cap \Psi_{s,t}(x,\mathcal F_{t,s,x}) \supset F''_{s,t,x}$.
    In particular, we verify the  condition~\eqref{eq:decomposition}.
    
    \medskip
    
    We first deal with  the estimate on the set $ G_{3,2}$.
    By Corollary \ref{cor:psi} and Corollary \ref{co:char.tilde} we have on $G_{3,2}$
    \begin{align}
       |\nabla_v \mu (v) -  \nabla_v \mu(V_{s,t}(x,\Psi_{s,t}(x,v))) \det(\nabla_v \Psi_{s,t}(x,v))| 
       \lesssim  \frac{\delta^{1-\beta} }{\langle s - \mathcal T_{t,x_1,v_1} \rangle} e^{-|v|}.
    \end{align}

Combination with Lemma \ref{lem:E.char} \eqref{it:E.post.collision} yields
    \begin{align} \label{est:G_32.0}
         \int_{G_{3,2}} r_d(s,x,v) &\lesssim \int_{B_{\frac{\Vmin}{2}}} \int_{[\mathcal T_{t,x_1,v_1} - 3]_+}^t \frac{\delta^{2-\beta}  }{\langle s - \mathcal T_{t,x_1,v_1} \rangle} \frac{e^{-|v|}}{\langle s - \mathcal T_{t,x_1,v_1} \rangle^2 + \langle x^\perp - (t-s) v^\perp \rangle^2} \dd s \dd v. \qquad
   \end{align} 
     For $\check \tau_{t,x} \leq 1$, the desired estimate follows immediately. If $\check \tau_{t,x} \geq 1$, we split the time integral: We observe that by~\eqref{eq:tau.T.comparable}, we have for $|v| \leq \Vmin/2$
   \begin{align}
        s - \mathcal T_{t,x_1,v_1} = \check{\mathcal T}_{t,x_1,v_1} -  (t-s) \geq \frac 1 2 \check{\mathcal T}_{t,x_1,v_1} \geq \frac 1 4 \check \tau_{t,x} \quad \text{for all } s \geq t - \frac{\check{\mathcal T}_{t,x_1,v_1}} 2 =: s^\ast_{t,x,v_1}.
    \end{align}
    Thus, changing variables $w= x^\perp - (t-s) v^\perp$ for $s < s^\ast_{t,x,v_1}$,  and using once more \eqref{eq:tau.T.comparable}, we have
    \begin{equation} \label{est:G_32}
    \begin{aligned}
         \int_{G_{3,2}} r_d(s,x,v)  &  \lesssim \frac{1}{\check \tau_{t,x}^2} \int_{-{\Vmin/2}}^{\Vmin/2} \int_{(\mathcal T_{t,x_1,v_1} - 3) \vee 0}^{s^\ast_{t,x,v_1}\vee 0} \int_{B_{\Vmin/2(t-s)}(x^\perp)} \frac{\delta^{2-\beta}  }{\langle s - \mathcal T_{t,x_1,v_1} \rangle} \frac{e^{-\frac {|w-x^\perp|}{4|t-s|}} e^{-\frac 1 4 |v_1|}}{\langle w \rangle^2} \dd w  \dd s \dd v_1 \\
        &+ \frac{1}{\check \tau_{t,x}^2} \int_{B_{\Vmin/2}} \int_{s^\ast_{t,x,v_1}\vee 0}^t \frac{\delta^{2-\beta}  }{\langle s - \mathcal T_{t,x_1,v_1} \rangle} e^{-|v|} \dd v \dd s \lesssim \frac{\delta^{2-2\beta}}{ \check \tau_{t,x}^2}.
    \end{aligned}
    \end{equation}
    The last inequality follows by separating the region $|w-x^\perp|\geq |t-s|$ and its complement.

    \medskip
    
    We now turn to the estimate on the set $G_{3,1}$. By definition of $\mathcal F_{s,t,x} \subset F_{s,t,x}$ we have the inclusion $G_{3,1} \subset \{(s,v): |v| \leq \Vmin/2,  0< s < \mathcal  T_{t,x_1,v_1} -1\}$. 
    In particular, if $\check \tau_{t,x} \geq 2 t$, we have by \eqref{eq:tau.T.comparable}
    $\mathcal T_{t,x_1,v_1} = t - \check{\mathcal T}_{t,x_1,v_1} \leq 0$ and thus
    $G_{3,1} = 0$. Therefore, it suffices to consider the case $\check \tau_{t,x} \leq 2 t$.
    
    In this case, Corollary~\ref{cor:psi} implies that
    \begin{align}
        G_{3,1} \subset \{(s,v) \in [0,t] \times \R^3 : v \in F_{s,t,x} \}.
    \end{align}
    We introduce
    \begin{align}
        w = \delta^{\beta} \langle \check \tau_{t,x} \rangle ( v^\perp - v^\perp_\ast(t,x,v_1)),
    \end{align}
      where $v_\ast^\perp= v_\ast^\perp(t,x,v_1)$ is defined as in \eqref{def:v_ast}.
    Since $\langle x^\perp\rangle  \delta^{\beta} \leq  \check \tau_{t,x}$ by the assumption in this step, we can estimate  $\delta^{\beta}| v^\perp - v_\ast^\perp| \leq \delta^{\beta}(| v^\perp| + | v_\ast^\perp|) \lesssim 1$ on $G_{3,1}$. Therefore, and by \eqref{est:check.x.v_ast},
    \begin{align}
        |w| \lesssim\langle \check \tau_{t,x} \rangle, && \langle w \rangle \lesssim \langle \check \tau_{t,x} ( v^\perp - v_\ast^\perp) \rangle, && \frac{\langle \check x_{t,x,v}^\perp \rangle}{\langle v^\perp\rangle}   \geq |w|.
    \end{align} 
    Therefore Corollary \ref{cor:psi} and Corollary \ref{co:char.tilde} imply on $G_{3,1}$
    \begin{align}
       & |\nabla_v \mu (v) -  \nabla_v \mu(V_{s,t}(x,\Psi_{s,t}(x,v))) \det(\nabla_v \Psi_{s,t}(x,v))| \\
       & \lesssim  \frac{ \delta^{1-3\beta}}{t-s} \left( 1  + \frac {t-s} {\langle  w \rangle}  +  \frac{\check \tau_{t,x} (\mathcal T_{t,x_1,v_1} - s)}{\langle w \rangle^2} \right)  e^{-|v|}.
    \end{align}
   We note also that by Lemma \ref{lem:E.char} \eqref{it:E.pre.collision}, on $G_{3,1}$
    \begin{align}
        |E(s,x - (t-s) v)| \leq  \frac{\delta }{\langle |w| + (\mathcal T_{t,x_1,v_1} - s)V_{\min}\rangle^2}. 
    \end{align}
    Combining the two preceding estimates, we obtain, using also $\eqref{eq:tau.T.comparable}$ and $\mathcal T_{t,x,v} \leq t$ due to the assumption $\check d_{t,x} = 0$ in this step,
    \begin{align}
        &\int_{G_{3,1}} r_d(s,x,v) \dd s \dd v\\
        & \lesssim  {\delta^{2-3\beta}} \int_{\R^3}  \int_0^{\mathcal T_{t,x_1,v_1} - 1} \frac{\1_{|w| \leq C \langle   \check \tau_{t,x} \rangle}}{t-s} \left( 1  + \frac {t-s} {\langle  w \rangle}  +  \frac{\check \tau_{t,x} (\mathcal T_{t,x_1,v_1} - s)}{\langle w \rangle^2} \right)  \frac{e^{-|v|} }{\langle |w| + (\mathcal T_{t,x_1,v_1} - s)V_{\min}\rangle^2} \dd s \dd v\\
        & \lesssim {\delta^{2-3\beta}} \int_{\R^3}  \int_1^{t} \frac{\1_{|w| \leq C \langle   \check \tau_{t,x} \rangle}}{\check{\mathcal T}_{t,x_1,v_1} + \sigma} \left( 1  + \frac {\check{\mathcal T}_{t,x_1,v_1} + \sigma} {\langle  w \rangle}  +  \frac{\check \tau_{t,x} \sigma}{\langle w \rangle^2} \right)  \frac{e^{-|v|}}{\langle |w| + \sigma V_{\min}\rangle^2} \dd \sigma \dd v\\
        &\lesssim \delta^{2-4\beta}\int_{\R^2} \left(\frac{1}{\Vmin \langle w \rangle \langle   \check \tau_{t,x} \rangle }  + \frac{1}{\langle  w  \rangle^2}\right) \1_{|w| \leq C \langle   \check \tau_{t,x} \rangle}\dd v_{\perp}\\
     &\lesssim \frac{\delta^{2-6\beta}}{ \langle   \check \tau_{t,x} \rangle^2 } \int_{\R^2} \left(\frac{1}{\Vmin \langle w \rangle \langle \check \tau_{t,x} \rangle}  + \frac{1}{\langle  w  \rangle^2}\right)\1_{|w| \leq C \langle   \check \tau_{t,x} \rangle} \dd w   \lesssim \frac{\delta^{2-5\beta}}{\langle   \check \tau_{t,x} \rangle^2},
    \end{align}
    where we used in the last inequality that $\log (2 +\check \tau_{t,x}) \lesssim \log(2+t) \lesssim \delta^{-\beta}$ since we can assume $\check \tau_{t,x} \leq 2 t$ as argued above.
    
\medskip
Finally, we turn to $B_3$. We split $B_3 = B_{3,1} + B_{3,2}$ where
   \begin{align}
       B_{3,1} &= \{(s,v) \in [0,t]  \times \R^3 :  v \in B_{\Vmin/2} \setminus B_{\delta^{-\beta}}, s < \mathcal T_{t,x_1,v_1} - 3 \} ,\\
       B_{3,2} &= \{(s,v) \in [0,t]  \times \R^3 :  v \in B_{\delta^{-\beta}} \setminus F''_{s,t,x}, s < \mathcal T_{t,x_1,v_1} - 3 \}.
   \end{align}
   As above, we observe that \eqref{eq:tau.T.comparable} implies that $B_{3,1} = B_{3,2} = \emptyset$  if $\tau_{x_1} < - t$.
   If $\tau_{x_1} \geq - t$, and thus $\check \tau_{t,x_1} \leq 2 t$, the desired estimate on $B_{3,1}$ is trivial, since  the exponential decay of $\mu$ together with the choice of $n,\delta_0$ in \eqref{eq:B2} gives  $e^{-\delta^{-\beta}} \leq C \frac{\delta^{2}} {\langle \check \tau \rangle^2}$ for $n$ sufficiently large and $\delta_0$ sufficiently small.

   On $B_{3,2} $ we estimate using the definition of $F''_{s,t,x}$, Lemma \ref{lem:E.char} \eqref{it:E.pre.collision} and $ \Vmin^{1-} \leq \delta$
   \begin{align}
       &\int_{ B_{3,2}} r_s(s,x,v) \dd v \dd s \\ \lesssim &\delta \int_\R \int_0^{[\mathcal T_{t,x_1,v_1} - 3]_+} \int_{|v^\perp - v^\perp_\ast|  |\check \tau_{t,x}| \leq \sqrt{\mathcal T_{t,x_1,v_1} - s}} \dd v^\perp \frac 1 {\langle \Vmin(\mathcal T_{t,x_1,v_1} - s )\rangle^2} \dd s e^{-|v_1|} \dd v_1 
       \lesssim  \frac{\delta^{2-  \beta}}{ | \check \tau_{t,x}|^2}.
   \end{align}

\emph{Step 4: The case $\check d_{t,x} = 0$ and $ \langle x^\perp \rangle\delta^{\beta} \geq \check \tau_{t,x}$ }.

Arguing as above, we write again $G_{4} = G_{4,1} \cup G_{4,2}$ and
\begin{align}
    G_{4,1} &:= \{(s,v) \in [0,t] \times \R^3 : v \in \mathcal K_{s,t,x} \}, \\
            G_{4,2} &:= \{(s,v) \in [0,t] \times \R^3 : v \in \mathcal A_{s,t,x} \}, \\
    B_{4}&:=\{(s,v) \in [0,t]  \times \R^3 :  v  \in B_{\Vmin} \setminus K''_{s,t,x}, s < \mathcal T_{t,x_1,v_1} - 3 \},
\end{align}
 with the sets $\mathcal K_{s,t,x}, K''_{s,t,x}, \mathcal A_{s,t,x}$ as defined in Corollary \ref{cor:psi}.
 
We first turn to $G_{4,2}$. Note that $G_{4,2} = G_{3,2}$ so in particular \eqref{est:G_32.0} holds in $G_{4,2}$. However, this time we want to gain a factor $\langle x^\perp \rangle^2$ instead of the factor $\langle \check \tau_{t,x}\rangle^2$ from the previous step. As above, the case $|x^\perp| \leq 1$ is straightforward and we therefore only consider $|x^\perp| \geq 1$ in the following.
We split $G_{4,2}$ further 
and first consider 
\begin{align}
   G_{4,2}^1 := G_{4,2} \cap \left\{(s,v) \in [0,t] \times \R^3 : |v| \geq \frac{ |x^\perp|}{2 (\check{\mathcal T}_{t,x_1,v_1} +5)}\right\}.
\end{align}
On this set, by \eqref{eq:tau.T.comparable}, we have $e^{-|v|} \lesssim \frac{\langle \check \tau_{t,x}^2 \rangle}{|x^\perp|^2}e^{-\frac{|v|}{2}}$.
Inserting this estimate within  \eqref{est:G_32.0} yields the desired estimate on $G_{4,2}^1$.

On the other hand, on $G_{4,2}^2 := G_{4,2}^1 \setminus G_{4,2}$ we have
\begin{align} \label{est:x_perp.backwards}
    | x^\perp - (t-s)v^\perp| \geq \frac{| x^\perp|}{2} \quad \text{for all } s \in [\mathcal T_{t,x_1,v_1} -5,t].
\end{align}
Resorting to \eqref{est:G_32.0} leads again to the desired estimate.

\medskip

We next turn to $G_{4,1}$.
    By Corollary \ref{cor:psi}, we have 
    \begin{align}
        G_{4,1} \subset \{(s,v) \in [0,t] \times \R^3 : v \in K_{s,t,x} \}.
    \end{align}
Combining the estimates from Corollary \ref{cor:psi} and Corollary \ref{co:char.tilde} with \eqref{est:check.x.K'} yields
    \begin{align}
       |\nabla_v \mu (v) -  \nabla_v \mu(V_{s,t}(x,\Psi_{s,t}(x,v))) \det(\nabla_v \Psi_{s,t}(x,v))| \lesssim \frac{\delta^{1- 2\beta}}{\langle x^\perp\rangle} e^{-|v|}.
    \end{align}
    By Lemma \ref{lem:E.char} \eqref{it:E.pre.collision} and \eqref{est:check.x.K'} we have
    \begin{align}
        |E(s,x-(s-t)v)| \lesssim \frac{\delta}{\langle x^\perp\rangle^2 +\langle \mathcal T_{t,x_1,v_1}-s)|V_{\min}|\rangle^2}.
    \end{align}
    Therefore,
    \begin{align}
        \int_{G_{4,1}} r_d(s,x,v) \dd s \dd v 
	& \lesssim \frac{\delta^{2-2\beta}}{\langle x^\perp \rangle} \int_{\R^3}  \int_{0 }^{ [\mathcal T_{t,x_1,v_1}- 1]_+} 
	\frac{e^{-|v|}}{\langle x^\perp \rangle^2 + \frac 1 2 \langle (\mathcal T_{t,x_1,v_1}-s)|V_{\min}|\rangle^2}    \dd s \dd v  \\
	& \lesssim  \frac{\delta^{2-2\beta}}{V_{\min}\langle x^\perp \rangle^2}.
\end{align}

\medskip

Regarding $B_{4}$, we first argue
   \begin{align}
       B_{4} \subset \{(s,v) \in [0,t]  \times \R^3 :  v  \in B_{\Vmin/4} \setminus B_{\delta^{-\beta}/6}, s < \mathcal T_{t,x_1,v_1} - 3 \}.
   \end{align}
Indeed, for $(s,v) \in  B_{4}$ with $s < \mathcal T_{t,x_1,v_1} - 1$ and $|v| \leq \delta^{-\beta}/6$, we find, due to the assumption
$\langle x^\perp\rangle \delta^\beta \geq \check \tau_{t,x}$ that we made in this step, that 
\begin{align}
    \check \tau_{t,x} \langle v^\perp \rangle  \leq  \check \tau_{t,x} + \check \tau_{t,x}   |v|\leq \langle x^\perp\rangle \delta^\beta \left(1+\frac{\delta^{-\beta}} 6 \right)  \leq  \frac {\langle x^\perp\rangle } 6,
\end{align}
and thus $v \in K''_{s,t,x}$.
Moreover, as above, either $B_{4} = \emptyset$  or $\check \tau_{t,x} \leq 2 t$, the latter we assume in the following. 
We split again, similarly as for $G_{4,2}$,
\begin{align}
   B_{4,1} := B_{4} \cap \{(s,v) \in [0,t] \times \R^3 : |v| \geq \frac{ \langle x^\perp \rangle}{2 \langle\check{\mathcal T}_{t,x_1,v_1}\rangle }\}.
\end{align}
On this set, by \eqref{eq:tau.T.comparable}, we have for $n \geq 3 \beta^{-1}$
\begin{align} \label{eq:gain.through.exp}
    e^{-|v|} \lesssim e^{-\delta^{-\beta}/18} \frac{\langle \check \tau_{t,x} \rangle^3} {\langle x^\perp \rangle^3}e^{-\frac{|v|}{3}} \lesssim e^{-\delta^{-\beta}/18} \frac{\langle t \rangle^3} {\langle x^\perp \rangle^3}e^{-\frac{|v|}{3}} \lesssim \frac {1} {\langle x^\perp \rangle^3}e^{-\frac{|v|}{3}}.
\end{align}
Combining this estimate with Lemma \ref{lem:E.char}\eqref{it:E.pre.collision} and using $\Vmin^{-1} \leq \delta$
yields
\begin{align}
    \int_{B_{4,1}} r_s(s,x,v) \dd v \dd s \lesssim \frac {\delta} {\langle x^\perp \rangle^3} \int_{\R^3} \int_0^{[\mathcal T_{t,x_1,v_1} - 3]_+}\frac{e^{-\frac{|v|}{3}} }{\langle\Vmin(\mathcal T_{t,x_1,v_1} -s) \rangle^2} \dd v \dd s \lesssim \frac {\delta^2} {\langle x^\perp \rangle^3}. 
\end{align}
Finally, on $B_{4,2} := B_4 \setminus B_{4,1}$, we estimate
\begin{align}
    \int_{B_{4,2}} r_s(s,x,v) \dd v \dd s \lesssim \delta \int_{B_{\Vmin/2} \setminus B_{\delta^{-\beta}/6}} \int_0^{[\mathcal T_{t,x_1,v_1} - 3]_+}\frac{e^{-|v| }}{{\langle x^\perp \rangle^2} + \langle\Vmin(\mathcal T_{t,x_1,v_1} -s) \rangle^2} \dd v \dd s \lesssim \frac {\delta^2} { \langle x^\perp \rangle^2},
\end{align}
where we used  $e^{-\delta^{-\beta}} \lesssim \frac{\delta}{\langle t \rangle^2} $ for $n$ sufficiently large.
\end{proof}

\subsection{Estimates for \texorpdfstring{$\nabla \Reac$}{nabla R}}

\begin{proposition} \label{pro:gradR} 
For all $0 < \gamma < 1$ there exists $C>0$ such that under the bootstrap assumptions \eqref{eq:B1}--\eqref{eq:B3} with  $\delta_0,n > 0$ sufficiently small  we have for all $t \leq T$
    \begin{align} \label{est.nabla.R}
        |\nabla \Reac(t,x)| \leq C \frac{\delta^{1 + \gamma}}{1 + \check\tau^3_{t,x} + \check d_{t,x}^3  + |x^\perp|^3}.
    \end{align}
\end{proposition}
\begin{proof}

The proof is in large parts analogous to the proof of Proposition \ref{pro:reaction}.
We therefore just highlight the main differences. The main difficulty consists in extracting the third power in the denominator of \eqref{est.nabla.R} in comparison with the second power obtained in Proposition \ref{pro:reaction}. To this end we must exploit once more the dispersion.

We will again distinguish the same three different cases as in the proof of Proposition \ref{pro:reaction}. In the case $\check d_{t,x} > 0$, the estimates are easiest, since the backwards characteristics do not come close to the point charge. Therefore, the error estimates along the backwards characteristics are sufficient in this case.

For $\check d_{t,x} = 0$, the estimates are more delicate.
Let us briefly explain, why we expect better decay for $\nabla_x \mathcal{R}$ than for $\mathcal{R}$ itself also in this case. Basically, for characteristics close to free transport, we can make use of $\nabla_x \sim \frac 1 t \nabla_v$.
More precisely, by integration by parts we find 
\begin{align}
    \int \nabla_x  f_0(x - t v) g(v) \dd v &=  \frac{1}{t}\int f_0(x - t v) \nabla g(v) \dd v.
\end{align}
This can also obtained through the following change of variables
\begin{align}
    \nabla_x \int f_0(x - t v) g(v) \dd v &= \frac{1}{t} \nabla_x \frac 1 {t^3} \int f_0(w) g\left(\frac{ x - w}{t}\right) \dd w \\
    &=   \frac 1 {t^4} \int f_0(w) \nabla g\left(\frac{ x - w}{t}\right) \dd w =   \frac 1 {t} \int f_0(x - t v) \nabla g(v) \dd v.
\end{align}
{This argument still works well in our setting when we are close to free transport.} We will thus manipulate $\mathcal R(t,x)$ through a change of variables before taking the gradient. Roughly speaking, the change of variables consists in replacing $v$ by a point along the straightened backwards
characteristics which corresponds to a time after the (potential) "approximate collision" along this characteristics. By taking the gradient after this change of variables we will gain the desired power.
Indeed, we know that the time after an approximate collision is  $\check\tau_{t,x}$.

Moreover, if $|x^\perp|$ is dominant over $\tau_{t,x}$ (and $|v|$ is of order $1$) 
the decay of $E$ in $|x^\perp|$ allows us to choose $\langle x^\perp \rangle$ as this corresponding time. Indeed, in view of the estimates \eqref{tilde.Y.passed}, the error for the backwards characteristics until times $s \geq t - |x^\perp|$ can still be controlled by $\delta \log(2+t)$, whereas for larger times, this error grows linearly in $s$, just as if there was a collision at time $t - |x^\perp|$.

\medskip

\emph{Step 1: The case $\check d_{t,x} > 0$.}

We have
\begin{align}
	\nabla \mathcal R(t,x) &=\int_0^{t} \int_{\R^3}  \nabla E(s,x-(t-s)v) \cdot \nabla_v \mu(v) -  \nabla E(s,X_{s,t}(x,v)) \cdot \nabla_v \mu(V_{s,t}(x,v)) \dd v \dd s\\
	&- \int_0^{t} \int_{\R^3} \nabla_x \tilde Y_{s,t}(x,v) \cdot \nabla E(s,X_{s,t}(x,v)) \nabla_v \mu(V_{s,t}(x,v)) \dd v \dd s\\
	&- \int_0^{t} \int_{\R^3}  E(s,X_{s,t}(x,v)) \nabla_x \tilde W_{s,t}(x,v) \cdot \nabla_v \mu(V_{s,t}(x,v)) \dd v \dd s\\
	&=: \mathcal R_a + \mathcal R_c + \mathcal R_d.
\end{align}

The estimate of $\mathcal R_a$ works exactly as before. Indeed, as we pointed out above in Step 2 of the proof of Proposition \ref{pro:reaction}, we could have already gained three powers of $\check d_{t,x} + |x^\perp|$ for $\mathcal R(t,x)$ in this case.

The terms $\mathcal R_c$ and $\mathcal R_d$
are estimated analogously, since the estimates of $\nabla \tilde W$ and $\nabla \tilde Y$ bring an additional power of $\check d_{t,x} + |x^\perp|$.
More precisely, combining Proposition \ref{pro:char.front}, Lemma \ref{lem:E.char} \eqref{it:E.front} yields
\begin{align}
    |\nabla E(s,X_{s,t}(x,v))| |\nabla_x \tilde Y_{s,t}(x,v)| \lesssim \frac{\delta \log^2(2 + t)}{\Vmin} \frac{1}{\langle \check d_{t,x} + |x|^\perp \rangle^2} \frac{t -s}{\langle \check d_{t,x} + (t-s) \Vmin + |x|^\perp \rangle^3},
\end{align}
and the same bound holds for $ | E(s,X_{s,t}(x,v))|  |\nabla_x \tilde W_{s,t}(x,v)|$
Integrating this bound in $s$ and using the exponential decay of $\mu$ for the integration in $v$ immediately yields the desired estimate.

\emph{Step 2: The case $\check d_{t,x} = 0$ and $\langle x^\perp \rangle \delta^{\beta} \leq \check \tau_{t,x}$}
Analogously as in the proof of Proposition \ref{pro:reaction}, in this case it suffices to show 
    \begin{align}
         |\nabla \mathcal R(t,x)| \lesssim \frac{\delta^{2 - M \beta}}{ \check \tau_{t,x}^3},
    \end{align}
for some $M$ independent of $\beta$.

The key idea is to use the change of variables
\begin{align}
    \omega = \check x_{t,x,v} =  x- \check{\mathcal T}_{t,x_1,v_1}v \quad \Leftrightarrow \quad  v = \tfrac{x-\omega}{t-\tau_\omega},
\end{align}
since by \eqref{eq:Ttau} $\tau_\omega=\mathcal T_{t,x,v}$. Performing this change of variables (and recalling the definition of the error functions $Y,W$ from \ref{def:Y,W}) yields
\begin{align}
	\mathcal R_{NL}(x)&=\int_0^{t} \int_{\R^3}  E(s,x-(t-s)v+Y_{s,t}(x-\check{\mathcal T}_{t,x_1,v_1}v,v)) \cdot \nabla_v \mu(v+W_{s,t}(x-\check{\mathcal T}_{t,x_1,v_1}v,v)) \dd v \dd s\\
	&=\int_0^{t} \int_{\R^3} \frac{1}{ \check \tau_{t,\omega}^3} E(s,\omega-(\tau_\omega-s)\tfrac{x-\omega}{t-\tau_\omega}+Y_{s,t}(\omega,\tfrac{x-\omega}{t-\tau_\omega}))\cdot  \nabla_v \mu(\tfrac{x-\omega}{t-\tau_\omega}+W_{s,t}(\omega,\tfrac{x-\omega}{t-\tau_\omega}))\dd \omega \dd s.
\end{align}
Taking the gradient in $x$ yields
\begin{align*}
    &\nabla_x \mathcal R_{NL}(x)=\int_0^{t} \int_{\R^3} \tfrac{s - \tau_\omega}{\check \tau_{t,\omega}^4} \nabla_x E(s,\omega-(\tau_\omega-s)\tfrac{x-\omega}{t-\tau_\omega}+Y_{s,t}(\omega,\tfrac{x-\omega}{t-\tau_\omega}))\cdot  \nabla_v \mu(\tfrac{x-\omega}{t-\tau_\omega}+W_{s,t}(\omega,\tfrac{x-\omega}{t-\tau_\omega})) \dd \omega \dd s\\
    &+\int_0^{t} \int_{\R^3} \tfrac{1}{\check \tau_{t,\omega}^4} E(s,\omega-(\tau_\omega-s)\tfrac{x-\omega}{t-\tau_\omega}+Y_{s,t}(\omega,\tfrac{x-\omega}{t-\tau_\omega}))\cdot  \nabla^2_v \mu(\tfrac{x-\omega}{t-\tau_\omega}+W_{s,t}(\omega,\tfrac{x-\omega}{t-\tau_\omega})) \dd \omega \dd s\\
    &+\int_0^{t} \int_{\R^3} \tfrac{\nabla_v Y_{s,t}(\omega,\tfrac{x-\omega}{t-\tau_\omega})}{\check \tau_{t,\omega}^4}\cdot  \nabla_x E(s,\omega-(\tau_\omega-s)\tfrac{x-\omega}{t-\tau_\omega}+Y_{s,t}(\omega,\tfrac{x-\omega}{t-\tau_\omega})) \cdot \nabla_v \mu(\tfrac{x-\omega}{t-\tau_\omega}+W_{s,t}(\omega,\tfrac{x-\omega}{t-\tau_\omega})) \dd \omega \dd s\\
    &+\int_0^{t} \int_{\R^3} \tfrac{\nabla_v W_{s,t}(\omega,\tfrac{x-\omega}{t-\tau_\omega})}{\check \tau_{t,\omega}^4}\cdot  E(s,\omega-(\tau_\omega-s)\tfrac{x-\omega}{t-\tau_\omega}+Y_{s,t}(\omega,\tfrac{x-\omega}{t-\tau_\omega}))\cdot  \nabla^2_v \mu(\tfrac{x-\omega}{t-\tau_\omega}+W_{s,t}(\omega,\tfrac{x-\omega}{t-\tau_\omega})) \dd \omega \dd s,
\end{align*}
and changing back to the original set of variables we obtain
\begin{align}
    \nabla_x \mathcal R_{NL}(x)=&\int_0^{t} \int_{\R^3} \frac{s - \mathcal T_{t,x_1,v_1}}{\check {\mathcal T}_{t,x_1,v_1}} \nabla_x E(s,X_{s,t}(x,v)) \cdot \nabla_v \mu(V_{s,t}(x,v))  \dd \omega \dd s\\
    &+\int_0^{t} \int_{\R^3} \frac{1}{\check {\mathcal T}_{t,x_1,v_1}}  E(s,X_{s,t}(x,v)) \cdot \nabla^2_v \mu(V_{s,t}(x,v)) \dd v \dd s \\
    &+\int_0^{t} \int_{\R^3} \frac{\nabla_v Y_{s,t}(x-\check {\mathcal T}_{t,x_1,v_1}v,v)}{\check {\mathcal T}_{t,x_1,v_1}} \cdot \nabla_x E(s,X_{s,t}(x,v)) \cdot \nabla_v \mu(V_{s,t}(x,v)) \dd v \dd s \\
    &+ \int_0^{t} \int_{\R^3} \frac{\nabla_v W_{s,t}(x-\check {\mathcal T}_{t,x_1,v_1}v,v)}{\check {\mathcal T}_{t,x_1,v_1}} \cdot  E(s,X_{s,t}(x,v)) \cdot  \nabla^2_v \mu(V_{s,t}(x,v))\dd v \dd s .
\end{align}
Performing the same manipulations on the linear term leads to 
\begin{align*}
    \nabla \mathcal R(t,x) 
    =&\int_0^{t} \int_{\R^3} \frac{s- \mathcal T_{t,x_1,v_1}}{\check {\mathcal T}_{t,x_1,v_1}} (\nabla_x E(s,x - (t-s)v)\cdot  \nabla_v \mu(v) - \nabla_x E(s,X_{s,t}(x,v)) \cdot \nabla_v \mu(V_{s,t}))  \dd v \dd s \\
    &+\int_0^{t} \int_{\R^3} \frac{1}{\check {\mathcal T}_{t,x_1,v_1}}  (E(s,x - (t-s)v) \cdot  \nabla^2_v \mu(v)- E(s,X_{s,t}(x,v)) \cdot \nabla^2_v \mu(V_{s,t}(x,v))) \dd v \dd s  \\
    &-\int_0^{t} \int_{\R^3} \frac{\nabla_v Y_{s,t}(x-\check {\mathcal T}_{t,x_1,v_1}v,v)}{\check {\mathcal T}_{t,x_1,v_1}} \nabla_x E(s,X_{s,t}(x,v)) \cdot \nabla_v \mu(V_{s,t}(x,v)) \dd v \dd s  \\
    &- \int_0^{t} \int_{\R^3} \frac{\nabla_v W_{s,t}(x-\check {\mathcal T}_{t,x_1,v_1}v,v)}{\check {\mathcal T}_{t,x_1,v_1}}  E(s,X_{s,t}(x,v))\cdot  \nabla^2_v \mu(V_{s,t}(x,v)) \dd v \dd s \\
    =:& \mathcal R_a + \mathcal R_b + \mathcal R_c + \mathcal R_d.
\end{align*}
The estimates of $\mathcal R_a$ 
and $\mathcal R_b$ are analogous as in Proposition \ref{pro:reaction}. 
The additional factor $\mathcal T_{t,x_1,v_1} -s$ in $\mathcal R_a$ does not pose a problem if one uses the third power of the decay of $E$ instead of the second power as in the proof of Proposition \ref{pro:reaction}. More precisely, by Lemma \ref{lem:E.char} \eqref{it:E.post.collision} and \eqref{it:E.pre.collision}
we have 
\begin{align}
    |\mathcal T_{t,x_1,v_1} -s| (|\nabla_x E(s,x-(t-s)v| + |\nabla_x E(s,X_{s,t}(x,v)|) \lesssim \frac{\delta}{\langle s - \mathcal T_{t,x_1,v_1} \rangle^2 + \langle x^\perp - (t-s) v^\perp \rangle^2} ,
\end{align}
and 
\begin{align}
    |\mathcal T_{t,x_1,v_1} -s| (|\nabla_x E(s,x-(t-s)v| + |\nabla_x E(s,X_{s,t}(x,v)|) \lesssim \frac{\delta }{\langle \Vmin (\mathcal T_{t,x_1,v_1} - s) \rangle^2 + \langle \check x^\perp_{t,x,v} \rangle^2} ,
\end{align}
for $ s \geq \mathcal T_{t,x_1,v_1} - 5 $ and $ s \leq \mathcal T_{t,x_1,v_1} $ respectively, which are precisely the estimates we used for $E$ in the proof of Proposition \ref{pro:reaction}.

\medskip

It remains to estimate $ \mathcal R_c$  and  $\mathcal R_d$. We use  the estimates from Proposition~\ref{prop:char} for $\nabla_v  Y$, $\nabla_v  W$. Since the estimates for $\nabla_v  Y$ are weaker than those for  $\nabla_v  W$, it suffices to show the desired estimates for $\mathcal R_c$.
We write
\begin{align}
    \mathcal R_c =: \int_0^t \int_{\R^3} \mathfrak r_c(s,x,v) \dd v \dd s.
\end{align}
Similarly as before, we split the integral in $\tilde G_{3,1}:= \{(s,v) : |v| \leq \Vmin/4, 0 \leq s \leq \mathcal T_{t,x_1,v_1}\}$ and $\tilde G_{3,2}:=\{(s,v) : |v| \leq \Vmin/4,  t \geq s \geq [\mathcal T_{t,x_1,v_1}]_+\}$.

We note that the identity \eqref{eq:Ttau} implies that we can use \eqref{eq:nablaYBehind} in the set $\tilde G_{3,2}$ to estimate the term $\nabla_v Y_{s,t}(x-\check {\mathcal T}_{t,x_1,v_1}v,v)$, and thus
\begin{align}
	|\nabla_v Y_{s,t}(x-\check {\mathcal T}_{t,x_1,v_1}v,v)| \lesssim \delta \log(2+t) \leq \delta^{1-\beta} ,
\end{align}
where we used again \eqref{est:log.beta}.

Combining this estimate with  Lemma \ref{lem:E.char} \eqref{it:E.pre.collision} and \eqref{eq:tau.T.comparable} yields on $\tilde G_{3,2}$
\begin{align}
    | \mathfrak r_c(s,x,v)| &\leq  \frac 1 { \check \tau_{t,x} } \frac{\delta^{2-\beta}}{\langle  (\mathcal T_{t,x_1,v_1} - s)^3 + | x^\perp - (t - s)   v^\perp|^3\rangle } e^{-|v|} .
\end{align}
Thus, $ \check \tau_{t,x} \int_{\tilde G_{3,2}} | \mathfrak r_c(s,x,v)|$ is bounded by the right-hand side in \eqref{est:G_32.0}. We thus the desired estimate by the estimates after \eqref{est:G_32.0}.

Similarly, on $\tilde G_{3,1}$,  Lemma \ref{lem:E.char} \eqref{it:E.post.collision} and \eqref{eq:nablaYPassed} imply
\begin{align}
    | \mathfrak r_c(s,x,v)| &\leq  \frac 1 { \check \tau_{t,x} } \left( \frac{(\mathcal T_{t,x_1,v_1} - s)\langle v^\perp \rangle}{\langle \check x_{t,x,v}^\perp \rangle}+ 1  \right)\frac{\delta^2 \log(2+t)}{\langle (V_{\min} (\mathcal T_{t,x_1,v_1} - s))^3 + |\check x_{t,x,v}^\perp|^3\rangle } e^{-|v|} \\
    & \lesssim \frac 1 { \check \tau_{t,x} } \frac{\delta^{2-\beta}}{\langle (V_{\min} (\mathcal T_{t,x_1,v_1} - s))^2 + |\check x_{t,x,v}^\perp|^2\rangle } \frac{1}{\langle \check x^\perp_{t,x,v}\rangle} e^{-|v|}.
\end{align}
We now proceed similarly as in the estimate on $G_{3,1}$ in the proof of Proposition \ref{pro:reaction}.
 We recall that either $\tilde G_{3,1} = \emptyset$ or $\tau_{x} \geq - t$.
Thus, using the change of variables 
\begin{align}
    \omega^\perp = \check x^\perp = x^\perp - \check {\mathcal T}_{t,x_1,v_1} v^\perp,
\end{align}
and \eqref{eq:tau.T.comparable}, we obtain the estimate
\begin{align}
    \int_{\R^3}  \int_0^{[\mathcal T_{t,x,v}]_+} \mathfrak r_c(s,x,v) \dd v \dd s  
     &\lesssim \frac{\delta^{2-\beta}}{\check \tau^3_{t,x}}  \int_{-\infty}^\infty  \int_{\R^2} \int_0^{{\mathcal T}_{t,x_1,v_1}} \frac{1}{\langle \omega^\perp \rangle } \frac{e^{-\frac{|x^\perp - \omega^\perp|}{4 \check{\mathcal  T}_{t,x,v}}}}{({\mathcal T}_{t,x_1,v_1}-s)^2 + \langle \omega^\perp\rangle^2}   \ud{s} \ud{\omega^\perp} e^{-\frac {|v_1|} 4 }\dd v_1 \\
     &\lesssim \frac{\delta^{2-\beta}}{\check \tau^3_{t,x}} \int_{\R^2}  \frac{1}{\langle \omega^\perp\rangle^2} e^{-\frac{|x^\perp - \omega^\perp|}{\check{ \tau}_{t,x}}} \ud{\omega^\perp}  
     \lesssim \frac{\delta^{2-2\beta}}{\check \tau^3_{t,x}} .
\end{align}

\emph{Step 3: The case $\check d_{t,x} = 0$ and $\langle x^\perp \rangle \delta^{\beta} \geq \check \tau_{t,x}$.}

In this case, we use a different change of variables before taking the gradient. More precisely, for $R>0$  (which we will later choose as $R = \langle x^\perp \rangle)$, we write $v = \frac{x - \omega} R$ to find
\begin{align}
	\mathcal R_{NL}(t,x)=\frac{1}{R^3} \int_0^{t} \int_{\R^3}  E(s,\omega -(t-R-s)\tfrac{x-\omega}{R}+\tilde{Y}(x,\tfrac{x-\omega}{R})) \nabla_v \mu(\tfrac{x-\omega}{R} + \tilde{W}(x,\tfrac{x-\omega}{R})) \dd \omega \dd s.
\end{align}

Taking the gradient on this term as well as the corresponding linear term $\mathcal R_L$, then reverting the change of variables and finally setting $R = \langle x^\perp \rangle$ yields
\begin{align*}
    \nabla \mathcal R(t,x) 
    &=\int_0^{t} \int_{\R^3} \frac{t  - s - \langle x^\perp \rangle}{\langle x^\perp \rangle} ( \nabla_x E(s,x - (t-s)v)\cdot  \nabla_v \mu(v) -\nabla_x E(s,X_{s,t}(x,v)) \cdot \nabla_v \mu(V_{s,t})) \dd v \dd s \\
    &+\int_0^{t} \int_{\R^3} \frac{1}{\langle x^\perp \rangle}  (E(s,x - (t-s)v)\cdot  \nabla^2_v \mu(v) - E(s,X_{s,t}(x,v))\cdot  \nabla^2_v \mu(V_{s,t}(x,v)) )  \dd v \dd s \\
    &-\int_0^{t} \int_{\R^3} \left(\nabla_x \tilde Y_{s,t}(x,v) + \frac{\nabla_v \tilde Y_{s,t}(x,v)}{\langle x^\perp \rangle} \right) \cdot \nabla_x E(s,X_{s,t}(x,v))\cdot  \nabla_v \mu(V_{s,t}(x,v))  \dd v \dd s\\
    &- \int_0^{t} \int_{\R^3} \left( \nabla_x \tilde W +  \frac{\nabla_v \tilde W_{s,t}(x,v)}{\langle x^\perp \rangle} \right) \cdot  E(s,X_{s,t}(x,v)) \cdot  \nabla^2_v \mu(V_{s,t}(x,v)) \dd v \dd s \\
    &=: \mathcal R_a + \mathcal R_b + \mathcal R_c + \mathcal R_d .
\end{align*}
We argue that $\mathcal R_a + \mathcal R_b$ can be estimated as in the proof of Proposition \ref{pro:reaction}. For $\mathcal R_b$ this is obvious. For $\mathcal R_a$, we consider the set
\begin{align}
    \tilde G_4 := \{(s,v) \in [0,t] \times B_{\Vmin/4}(0) : |v| \leq \frac{ \langle x^\perp \rangle}{10 \langle\check{\mathcal T}_{t,x_1,v_1}\rangle }\}.
\end{align}
Then, on $\tilde G_4$, we have, recalling that we are in the case $\langle x^\perp \rangle \delta^{\beta} \geq \check \tau_{t,x}$ and using \eqref{eq:tau.T.comparable},
\begin{align}
    |\check x^\perp_{t,x,v}| &\gtrsim |x^\perp|,  \label{est.x_perp.good}\\ 
    |x^\perp - (t-s) v| &\gtrsim |x^\perp| \quad \text{for all } \mathcal T_{t,s,x} - 5 \leq s \leq t, \\
     |(t-s)-\langle x^\perp\rangle| &\leq  \langle x^\perp\rangle + [\mathcal T_{t,x_1,v_1} -s]_+.
\end{align}
Thus, by Lemma \ref{lem:E.char} \eqref{it:E.post.collision} and \eqref{it:E.pre.collision},
\begin{align}
    |(t-s)-\langle x^\perp\rangle| (|\nabla_x E(s,x-(t-s)v| + |\nabla_x E(s,X_{s,t}(x,v)|) \lesssim \frac{\delta }{\langle s - \mathcal T_{t,x_1,v_1} \rangle^2 + \langle x^\perp \rangle^2} ,
\end{align}
and 
\begin{align}
    |(t-s)-\langle x^\perp\rangle| (|\nabla_x E(s,x-(t-s)v| + |\nabla_x E(s,X_{s,t}(x,v)|) \lesssim \frac{\delta^{1 - \beta} \log (2 + t)}{\langle \Vmin (\mathcal T_{t,x_1,v_1} - s) \rangle^2 + \langle x^\perp \rangle^2} ,
\end{align}
for $ s \geq \mathcal T_{t,x_1,v_1} - 5 $ and $ s \leq \mathcal T_{t,x_1,v_1} $ respectively. With these bounds at hand, the desired  estimate follows precisely as in the proof of Proposition \ref{pro:reaction}.

We continue with the estimate on the set 
\begin{align}
    \tilde B_4 := \{(s,v) \in [0,t] \times B_{\Vmin/4}(0) : |v| \geq \frac{ \langle x^\perp \rangle}{10 \langle\check{\mathcal T}_{t,x_1,v_1}\rangle }\}.
\end{align}
Notice that on this set we have $|v| \geq \tfrac{1}{20} \delta^{-\beta}$ due to the assumption $\langle x^\perp \rangle \delta^{\beta} \geq \check \tau_{t,x}$ in the case under consideration. 
This allows us to argue analogous to the estimate on $B_{4,1}$ in the proof of Proposition~\ref{pro:reaction}:
Analogously as we have obtained \eqref{eq:gain.through.exp}, we find on $\tilde B_4$
\begin{align}
    e^{-|v|} \lesssim \frac {\delta e^{-\frac{|v|}{3}}} {\langle t \rangle^2 \langle x^\perp \rangle^3},
\end{align}
which allows us to deduce the desired estimate by just using the estimate $|\nabla E| \lesssim \delta$.

\medskip 
Regarding, $\mathcal R_c$ and $\mathcal R_d$, we use the estimates from Corollary \ref{co:char.tilde}. Again, since the estimates on $\tilde W$ are better than those on $\tilde Y$, it suffices to estimate $\mathcal R_c$. Moreover, since $\nabla_v \tilde Y_{s,t}(x,v)\lesssim \delta t^2$, we can argue as above on the set $\tilde B_4$ and it therefore suffices to consider the set $\tilde G_4$.
 Since by the assumption $\langle x^\perp\rangle \delta^\beta \geq \check \tau_{t,x}$ and \eqref{eq:tau.T.comparable}, we have $\check {\mathcal T}_{t,x_1,v_1} \leq \langle x^\perp \rangle$, 
Corollary \ref{co:char.tilde} and \eqref{est.x_perp.good} yield on $\tilde G_4$
\begin{align}
  |\nabla_x \tilde Y_{s,t}(x,v)| + \left|\frac{\nabla_v \tilde Y_{s,t}(x,v)}{\langle x^\perp \rangle}\right| &\lesssim
  \begin{cases}
    \frac{\delta \log (2 + t)\langle v^\perp \rangle}{\langle  x^\perp \rangle} \left( \frac{(\mathcal T_{t,x_1,v_1} -s) \langle v^\perp \rangle}{\langle  x^\perp \rangle} + 1\right) &\quad \text{for } 0<s < \mathcal T_{t,x_1,v_1}, \\
    \frac{\delta \log (2 + t)\langle v^\perp \rangle}{\langle  x^\perp \rangle}  &\quad \text{for } \mathcal T_{t,x_1,v_1}<s < t.
  \end{cases}
\end{align}
Combining again with the estimates from Lemma \ref{lem:E.char} and using $|v^\perp|^2e^{-|v|} \lesssim e^{-|v|/2}$ yields on both $\tilde G_{3,1}$ and $\tilde G_{3,2}$
\begin{align}
    | \mathfrak r_c(s,x,v)|     & \lesssim 
      \frac{\delta^{2 - \beta} \log(2+t)^2}{\langle (\mathcal T_{t,x,v} - s)^2 + | x^\perp|^2\rangle } \frac {e^{-|v|/2}}{\langle x^\perp \rangle^2} .
\end{align}
Integrating over $\tilde G_{3,1}$ and $\tilde G_{3,2}$ yields the desired estimate.
\end{proof}

\subsection{Contribution of the point charge} \label{sec:SI}
In this section we derive estimates for the function $\SI$ defined in~\eqref{def:SCharge}. For future reference, we also introduce the function $\bS (t,x)$ defined by
\begin{align} \label{def:bS} 
    \bS (t,x) = - \int_{-\infty}^t \int_{\Reals^3} \nabla \Phi(x-(t-s)v-(X(t)-(t-s)V(t))) \cdot \nabla_v \mu(v) \ud{v} \ud{s}.
\end{align}
Compared to $\SI$, this corresponds to a linearization of both the characteristics and the trajectory of the point charge and in addition to a extension to all negative times. In particular, $\bS$ resembles the function $S_{R,X_\ast,V_\ast}$ from \eqref{eq:S_R}.
\begin{proposition} \label{pro:S.ion}
 Under the bootstrap assumptions \eqref{eq:B1}--\eqref{eq:B3} with  $\delta_0,n > 0$ sufficiently small,  we have for all $t \leq T$
    \begin{align}
        |\SI(t,x)|+|\bS(t,x)| &\leq \frac{C}{\Vmin(1+|x^\perp|^2+\check d^2_{t,x_1}+\check{\tau}^2_{t,x_1})}, \label{est:SCharge}\\
        |\nabla \SI(t,x)| &\leq\frac{C}{\Vmin(1+|x^\perp|^3+\check d^3_{t,x_1}+\check{\tau}^3_{t,x_1})} \label{est:gradSCharge}.
    \end{align}
\end{proposition}
\begin{proof}

\emph{Step 1: Proof of \eqref{est:SCharge}:}
   Recall the definition of $\SI$
   \begin{align}
       \SI(t,x)=- \int_0^t \int_{\Reals^3} \nabla \Phi(X_{s,t}-X(s))\cdot  \nabla_v \mu( V_{s,t}) \ud{v} \dd s.
   \end{align}
   We observe that by \eqref{eq:X_s,t.comparable}
    and the definitions of $\mathcal T_{t,x,v}$ and $v_\ast = v_\ast(t,x,v_1)$ from Defintion \ref{def:Parameters2} and \eqref{def:v_ast}, for $|v|\leq \frac12 \Vmin$
   \begin{align}
      \langle X_{s,t}(x,v)-X(s) \rangle &\gtrsim  |s-\mathcal {T}_{t,x_1,v_1}| \Vmin +|v^\perp-v^\perp_*| \check  {T}_{t,x_1,v_1} \quad \text{if $\check d_{t,x_1}=0$},  \label{est:distance1}\\
      \langle X_{s,t}(x,v)-X(s) \rangle &\gtrsim \check d_{t,x_1}+ |x^\perp| + \Vmin (t-s),\quad \text{if $\check d_{t,x_1}>0$} \label{est:distance2}.
   \end{align}
   Consider first the case $\check d_{t,x} = 0$, i.e. $\tau_x \leq t$.
   which by \eqref{eq:tau.T.comparable} is equivalent to $\check{\mathcal{T}}_{t,x_1,v_1} \geq 1/2 \check \tau_{t,x}$ and thus $\check{\mathcal{T}}_{t,x_1,v_1} > 0$ is equivalent to $\check \tau_{t,x} > 0$. Then, by \eqref{est:distance1} and the decay of $\Phi$ (cf. \eqref{ass:Phi}), 
   \begin{align*}
       |\SI(t,x)| &\lesssim \int_0^t \int_{\Reals^3}e^{-|s-\mathcal {T}_{t,x_1,v_1}| \Vmin -|v^\perp-v^\perp_*| \check {\mathcal{T}}_{t,x_1,v_1}}  e^{-\frac14 (|v_1|+|v^\perp|)} \ud{v} \dd s\\
       &\lesssim \frac{1}{\Vmin(1+\check\tau^2_{t,x_1})}.
   \end{align*}
   For the desired decay in $|x^\perp|$, we consider again the sets
   \begin{align}
       G := \{v \in B_{\Vmin/2}(0) : |v| \leq \frac{ \langle x^\perp \rangle }{2 \check{\mathcal T}_{t,x_1,v_1}}\}, \\
       B := 	\{v \in B_{\Vmin/2}(0) : |v| \geq \frac{ \langle x^\perp \rangle }{2 \check{\mathcal T}_{t,x_1,v_1}}\}.
   \end{align}
For $v \in G$, we have $\langle|v^\perp-v^\perp_*| \check {\mathcal{T}}_{t,x_1,v_1}\rangle = \langle x^\perp - \check {\mathcal{T}}_{t,x_1,v_1} v^\perp \rangle \gtrsim \langle x^\perp\rangle$ and thus
\begin{align}
    \int_0^t \int_{G}e^{-|s-\mathcal {T}_{t,x_1,v_1}| \Vmin -|v^\perp-v^\perp_*| \check {\mathcal{T}}_{t,x_1,v_1}}  e^{-\frac14 (|v_1|+|v^\perp|)} \ud{v} \dd s
       &\lesssim \frac{1}{\Vmin(1+|x^\perp|^2)}.
\end{align}
Moreover, in $B$ we use that $e^{-|v|} \lesssim e^{-|v|/2} \frac{\check \tau_{t,x}^2}{\langle x^\perp\rangle^2}$ to deduce
\begin{align}
    \int_0^t \int_{B}e^{-|s-\mathcal {T}_{t,x_1,v_1}| \Vmin -|v^\perp-v^\perp_*| \check {\mathcal{T}}_{t,x_1,v_1}}  e^{-\frac14 (|v_1|+|v^\perp|)} \ud{v} \dd s
       &\lesssim \frac{\check \tau_{t,x}^2}{\Vmin(1+\check\tau^2_{t,x_1})\langle x^\perp\rangle^2} \lesssim \frac{1}{\Vmin(1+|x^\perp|^2)}.
\end{align}
   
Finally, if $\check d_{t,x}  > 0$, we use \eqref{est:distance2} to deduce the following estimate for $\SI$:
   \begin{align}
       |\SI(t,x)| &\lesssim \frac{1}{\Vmin(1+\check d^2_{t,x_1}+|x^\perp|^2}.
   \end{align}
   Collecting the above estimates, we obtain~\eqref{est:SCharge} for $\SI$. The estimate for $\bS$ is analogous.

   \emph{Step 2:  Proof of \eqref{est:gradSCharge}:} 
   We observe that Proposition~\ref{pro:char.front} and  Corollary~\ref{co:char.tilde} gives
   \begin{align} \label{est:xDer}
       |\nabla_x X_{s,t} |+|\nabla_x V_{s,t} |\lesssim 1+ |t \wedge \mathcal{T}_{t,x_1,v_1} - s|.
   \end{align}
   Since this term can be always absorbed by the exponential decay coming from $\nabla^2\Phi$, the desired estimates are analogous as above in the case 
    $\check \tau_{t,x_1}\leq 1$ (so in particular for $\check d_{t,x} > 0$).
    On the other hand, if $\check  \tau_{t,x_1}\geq 1$, then we rewrite $\SI$ similarly as in Step 1 of the proof of Proposition \ref{pro:gradR} as 
   \begin{align}
       &-\SI(t,x)= \int_0^t \int_{\Reals^3} \nabla \Phi(x-(t-s)v+Y_{s,t}(x-\check{\mathcal T}_{t,x,v}v,v)-X(s)) \nabla_v \mu( v+W_{s,t}(x-\check{\mathcal T}_{t,x,v}v,v)) \ud{v}\\
       &=\frac{1}{(t-\tau_\omega)^3} \int_0^t \int_{\Reals^3} \nabla \Phi(\omega-(\tau_\omega-s)\tfrac{x-\omega}{t-\tau_\omega}+Y_{s,t}(\omega ,\tfrac{x-\omega}{t-\tau_\omega})-X(s)) \nabla_v \mu( \tfrac{x-\omega}{t-\tau_\omega}+W_{s,t}(\omega,\tfrac{x-\omega}{t-\tau_\omega})) \ud{\omega}.
   \end{align}
   Taking the gradient in $x$ we obtain (omitting the arguments of $Y_{s,t}$ and $W_{s,t}$)
   \begin{align}
       |\nabla \SI(t,x)|&\lesssim \int_0^t \int_{\Reals^3} \frac{(\tau_\omega-s)}{(t-\tau_\omega)^4}  |\nabla^2 \Phi(\omega-(\tau_\omega-s)\tfrac{x-\omega}{t-\tau_\omega}+Y_{s,t}-X(s))| |\nabla_v \mu( \tfrac{x-\omega}{t-\tau_\omega}+W_{s,t})| \ud{\omega}\\
       +&  \int_0^t \int_{\Reals^3}  \frac{1}{(t-\tau_\omega)^4}|\nabla^2 \Phi(\omega-(\tau_\omega-s)\tfrac{x-\omega}{t-\tau_\omega}+Y_{s,t}-X(s))| |\nabla_v Y_{s,t}||\nabla_v \mu( \tfrac{x-\omega}{t-\tau_\omega}+W_{s,t})| \ud{\omega}\\
       +&  \int_0^t \int_{\Reals^3}  \frac{1}{(t-\tau_\omega)^4} |\nabla \Phi(\omega-(\tau_\omega-s)\tfrac{x-\omega}{t-\tau_\omega}+Y_{s,t}-X(s))| |\nabla_v^2 \mu( \tfrac{x-\omega}{t-\tau_\omega}+W_{s,t})| \ud{\omega}\\
       +&  \int_0^t \int_{\Reals^3} \frac{1}{(t-\tau_\omega)^4} |\nabla \Phi(\omega-(\tau_\omega-s)\tfrac{x-\omega}{t-\tau_\omega}+Y_{s,t}-X(s))| |\nabla^2_v \mu( \tfrac{x-\omega}{t-\tau_\omega}+W_{s,t})|\nabla_v W_{s,t}| \ud{\omega}.
   \end{align}
   We change variables back to $v$ and find
   \begin{align}
       |\nabla \SI(t,x)|&\lesssim\int_0^t \int_{\Reals^3} \frac{|s - \mathcal T_{t,x_1,v_1}|}{\check \tau_{t,x}}|\nabla^2 \Phi(X_{s,t}-X(s))| |\nabla_v \mu( V_{s,t})| \ud{v}\\
       +&  \int_0^t \int_{\Reals^3}\frac{1}{\check \tau_{t,x}} |\nabla^2 \Phi(X_{s,t}-X(s))| |\nabla_v Y_{s,t}(x-\check {\mathcal T}_{t,x_1,v_1}v,v)||\nabla_v \mu(V_{s,t})| \ud{v}\\
       +&  \int_0^t \int_{\Reals^3}\frac{1}{\check \tau_{t,x}} |\nabla \Phi(X_{s,t}-X(s))| |\nabla_v^2 \mu(V_{s,t})| \ud{v}\\
       +&  \int_0^t \int_{\Reals^3}\frac{1}{\check \tau_{t,x}} |\nabla \Phi(X_{s,t}-X(s))| |\nabla^2_v \mu(V_{s,t})|\nabla_v W_{s,t}(x-\check {\mathcal T}_{t,x_1,v_1}v,v)| \ud{v}.
   \end{align}
 By Proposition~\ref{prop:char} and \eqref{eq:Ttau} we have, using $\log (2+t) \leq \delta$,
   \begin{align}
       |\nabla_v Y_{s,t}(x-\check {\mathcal T}_{t,x_1,v_1}v,v)| + |\nabla_v W_{s,t}(x-\check {\mathcal T}_{t,x_1,v_1}v,v)| \lesssim \langle s - \mathcal T_{t,x_1,v_1} \rangle,
   \end{align}
   which can be absorbed again in the exponential decay coming from $\Phi$.
    The claim~\eqref{est:gradSCharge}  then follows by  repeating the argument of Step 1. 
\end{proof}

\subsection{Proof of Proposition~
\texorpdfstring{\ref{pro:bootstrap}(\ref{case:A})
 }
 {27i } 
}
\label{subsec:PropA}

Resorting to the definition of $\mathcal S$ in \eqref{def:S} and the definition of the norm $\|\cdot\|_{Y_T}$ in Definition \ref{defi:passage.times}, the proof of Proposition~\ref{pro:bootstrap}\eqref{case:A} just consists in combining the estimates from Propositions \ref{pro:reaction}, \ref{pro:gradR} with $\gamma=1/2$ with Proposition \ref{pro:S.ion}. 

%% file: 8.ErrorForce.tex
In this section, we prove Proposition \ref{pro:bootstrap}\eqref{case:B} which asserts that the force acting on the point charge is given, to the leading order, by the linearization of the system. To this end, we recall \eqref{def:S}--\eqref{eq:rhoRep} and  to rewrite for $R>0$
\begin{align}
    E(t,X(t))   &= -\nabla (\phi * \rho(t,\cdot))(X(t)) \\
                &= -\nabla (\phi * (G *_{s,x} \mathcal S + \mathcal S))(X(t))\\
                &=  \mathcal{F}^R(t) +        \mathcal{E}^R_1(t) + \mathcal{E}_2(t)+\mathcal{E}_3(t),
\end{align} 
where the linearized friction force $\mathcal{F}(t)$ and the error terms $\mathcal{E}^R_1$, $\mathcal E_2$ and $\mathcal{E}_3$ are given by
\begin{align}
    \mathcal{F}^R(t)  &= -(\nabla \phi \ast \rho[g_{R,X(t),V(t) }])(R,X(t)), \label{def:Fric} \\
    \mathcal{E}^R_1(t) &= -(\nabla \phi * (\SI+ G*\SI))(t,X(t)) + (\nabla \phi *(S_{R,X(t),V(t)} + G \ast S_{R,X(t),V(t)}))(R,X(t)) ,\label{def:E1} \qquad \\
    \mathcal{E}_2(t) &= -(\nabla \phi * \mathcal R)(t,X(t)), \label{def:E2}\\
    \mathcal{E}_3(t) &= -(\nabla \phi *G* \mathcal R))(t,X(t)). \label{def:E3}
\end{align}

\subsection{Contribution of the self-consistent field}

\begin{lemma}\label{lem:E2}
    Under the bootstrap assumptions \eqref{eq:B1}--\eqref{eq:B3} with  $\delta_0,n > 0$ sufficiently small respectively large,   
    the error term $\mathcal{E}_2$ (cf.~\eqref{def:E2}) can be estimated for all $t \in [0,T]$ by
    \begin{align}
        |\mathcal{E}_2(t)| \lesssim \log (2 +t) \delta^2 \Vmin^{-\frac12}.
    \end{align}
\end{lemma}
\begin{proof}
    \step{1}
    We start by rewriting $\mathcal R$ (cf.~\eqref{def:Reaction}) as
    \begin{equation} \label{eq:SFDec}
    \begin{aligned}
       \mathcal R(t,x)
        =&\int_0^t \int_{\R^3} \left( E(s,x - (t-s) v) \cdot \nabla_v \mu (v) -  E(s,X_{s,t}) \cdot \nabla_v \mu(V_{s,t})\right)  \dd v \dd s  \\
        =&\int_0^t \int_{\R^3} \dv E(s,x - (t-s) v) (t-s) \mu (v)-\dv  E(s,X_{s,t}) (t-s)\mu(V_{s,t}) \dd v \dd s \\
        &+\int_0^t \int_{\R^3}    E(s,X_{s,t}) \cdot \nabla_v \tilde W_{s,t}(x,v)\cdot  \nabla_v \mu(V_{s,t}) \dd v \dd s \\
        &-\int_0^t \int_{\R^3} \nabla_x   E(s,X_{s,t})\cdot \nabla_v \tilde Y_{s,t}(x,v)  \mu(V_{s,t}) \dd v \dd s.
    \end{aligned}
    \end{equation}
    \step{2} We show that for $|x-X(t)|\leq \Vmin^\frac12$
    \begin{align} \label{est:R.close.to.ion}
        |\mathcal R(t,x)|\lesssim \log(2+t)\delta^2 \Vmin^{-\frac12}.
    \end{align}
    We observe that $|x-X(t)|\leq \Vmin^\frac12$ implies for $|v|<\Vmin/2$ by \eqref{eq:tau.T.comparable} and \eqref{est.nabla.tau}
    \begin{align}\label{est:dtau}
        |\check{T}_{t,x,v}| \leq 2 |\check{\tau}_{t,x}|\leq 2 \Vmin^{-\frac12}, 
        \end{align}
    and thus by Lemma \ref{lem:E.char}
    \begin{align} \label{est:E.close.ion}
       | E(s,X_{s,t})| + |\nabla E(s,X_{s,t})| \lesssim \begin{cases}
        \frac \delta{\langle \Vmin (t-s) \rangle ^3}&\qquad \text{for } {s \leq t - 4 \Vmin^{-1/2}}, \\
        \delta &\qquad \text{for } {s \geq t - 4 \Vmin^{-1/2}},
        \end{cases}
    \end{align}    
    and by Proposition \ref{pro:char.front} and Corollary \ref{co:char.tilde}
    \begin{align} \label{est:char.linear.grow}
       |\tilde Y_{s,t}(x,v) | + |\tilde W_{s,t}(x,v) | + |\nabla_v \tilde Y_{s,t}(x,v) | + |\nabla_v \tilde W_{s,t}(x,v)| \lesssim \log(2+t) \delta  \langle t-s \rangle.
    \end{align}
   The last two terms in~\eqref{eq:SFDec} can thus be estimated by
    \begin{align}
        &\int_0^t \int_{\R^3}    \left| E(s,X_{s,t})\cdot \nabla_v \tilde W_{s,t}(x,v) \cdot \nabla_v \mu(V_{s,t})\right|+ \left|\nabla_x   E(s,X_{s,t}) \cdot \nabla_v \tilde Y_{s,t}(x,v)  \mu(V_{s,t})\right| \dd v \dd s\\
        \lesssim & \log(2+t) \delta^2 \left(\int_0^t \int_{\R^3}  \frac{\langle t-s\rangle  e^{-|v|}}{\langle \Vmin (t-s)\rangle^3 } \dd v \dd s + \int_{t-4 \Vmin^{-1/2}}^t \int_{\R^3} \langle t-s \rangle  e^{-|v|} \dd v \dd s \right) \\
        &\lesssim \log(2+t){\delta^2}{\Vmin^{-1/2}}. \label{eq:split.E_2}
    \end{align}
    
    For the first term on the right-hand side of \eqref{eq:SFDec}, we furthermore use that since $\phi$ is the fundamental solution to  $-\Delta  +1$, we have
    \begin{align}
        \dv  E = - ( \rho * \phi - \rho),
    \end{align}
    and in particular $\nabla \dv E = E + \nabla \rho$
Using the assumption \eqref{eq:B2.2} together with Lemma \ref{lem:est.E}, the same arguments that lead to the estimates in Lemma \ref{lem:E.char} and thus to \eqref{est:E.close.ion} also show uniformly for all $\lambda \in [0,1]$
     \begin{align}
       |\nabla \dv  E(s,\lambda (x - (t-s) v) + (1-\lambda)X_{s,t})| \lesssim \begin{cases}
        \frac \delta{\langle \Vmin (t-s) \rangle ^3}&\qquad \text{for } {s > t - 4 \Vmin^{-1/2}}, \\
        \delta &\qquad \text{for } {s < t - 4 \Vmin^{-1/2}}.
        \end{cases}
    \end{align}
    Combining these inequalities with \eqref{est:E.close.ion} and \eqref{est:char.linear.grow} and splitting the integral as in \eqref{eq:split.E_2} we obtain
    \begin{align}
        &\left|\int_0^t \int_{\R^3} \dv  E(s,x - (t-s) v) (t-s) \mu (v)-\dv E(s,X_{s,t}) (t-s)\mu(V_{s,t}) \dd v \dd s\right|
        \lesssim \log(2+t) \delta^2 \Vmin^{-1/2},
    \end{align}
   which finishes the proof of \eqref{est:R.close.to.ion}.
    
\step{3} Conclusion of the proof. We insert the estimate \eqref{est:R.close.to.ion} into the definition of $\mathcal{E}_2(t)$ (cf.~\eqref{def:E2}) and use the exponential decay of $\phi$ to find
\begin{align}
    |\mathcal{E}_2(t)| &= \left|\int_{\Reals^3} \nabla \phi(y) * \mathcal R(X(t)-y) \ud{y} \right| \\
    &\leq \left|\int_{\{|y|\leq \Vmin^\frac12\}} \nabla \phi(y) * \mathcal R(X(t)-y) \ud{y} \right|+\left|\int_{\{|y|\geq \Vmin^\frac12\}} \nabla \phi(y) * \mathcal R(X(t)-y) \ud{y} \right|\\
    &\lesssim \log(2+t) \delta^2 \Vmin^{-\frac12},
\end{align}
where we used for the estimate of the second term that $|\mathcal R(X(t)-y)| \leq 1$ by Proposition \ref{pro:reaction} and that $\Vmin^{-1} \leq \delta$.
\end{proof}
\begin{lemma}\label{lem:E3} 
Under the bootstrap assumptions \eqref{eq:B1}--\eqref{eq:B3} with  $\delta_0,n > 0$ sufficiently small respectively large,     the error term $\mathcal{E}_3$ (cf.~\eqref{def:E3}) can be estimated   
 for all $t \in [0,T]$
 by
    \begin{align}
        |\mathcal{E}_3(t)| = |(\phi * (\nabla \mathcal R *_{t,x} G))(X(t)| \lesssim \delta^{\frac 7 4} \Vmin^{-1/2}.
    \end{align}
\end{lemma}
\begin{proof}
   Let $|z|\leq \frac18 \Vmin^\frac12$, and consider
   \begin{align}
        (\nabla \mathcal R *_{t,x} G)(X(t)+z) = \int_0^t \int_{\Reals^3} \nabla \mathcal R (t-s,X(t)+z-y) G(s,y) \ud{y} \ud{s}. 
    \end{align}
    We split the integral in the regions
    \begin{align}
        A_1 =\{(s,y)\in [0,t]\times \Reals^3: |t-s|\geq  \Vmin^{-\frac12},\, |y|\leq \frac14\Vmin^\frac12  \}, \quad A_2 = [0,t]\times \Reals^3) \setminus A_1.
    \end{align}
    In the region $A_1$ we have
    \begin{align}
        \check d_{t-s,X(t)+z-y}\geq \frac12 \Vmin^\frac12.
    \end{align}
    Using the apriori estimate on $\mathcal R$ from Proposition \ref{pro:gradR} together with \eqref{eq:G_L^1} we therefore have for any $\beta \in (0,1)$
    \begin{align} \label{est:SF.A}
        \left|\int_{A_1} \nabla \mathcal R(t-s,X(t)+z-y) G(s,y) \ud{y} \ud{s}\right| &\lesssim \delta^{2-\beta} \Vmin^{-3/2}  \int_0^t \int_{B_{\Vmin^\frac12}} |G(s,y)|\ud{y} \\
        &\lesssim \log(2+t) \delta^{2-\beta} \Vmin^{-3/2} \lesssim \delta^{2-2\beta} \Vmin^{-3/2} ,
    \end{align}
    by using \eqref{eq:B2} with $n \geq \beta^{-1}$.
    
    On the complement of $A_1$, we can use that $|t-s|\leq \Vmin^{-\frac12}$ or $|y|\geq \frac14 \Vmin^\frac12$. Therefore using  \eqref{eq:G_L^1}
 and \eqref{eq:Gpoint} together with Proposition \ref{pro:gradR}
    \begin{align} \label{est:SF.A^c}
        &\left|\int_{A_1^c} \nabla \mathcal R(t-s,X(t)+z-y) G(s,y) \ud{y} \ud{s}\right| \\
        \lesssim &\delta^{2-\beta} \left( \Vmin^{-\frac12} \sup_{0\leq s\leq t} \int_{\Reals^3} |G(s,y)| \ud{y}  +\int_0^t \int_{B^c_{\frac14\Vmin^\frac12}} \frac{1}{1+|z^\perp-y^\perp|^3} |G(s,y)|\ud{y} \ud{s} \right)  \lesssim \delta^{2-\beta} \Vmin^{-\frac12}.
    \end{align}
   Choosing $\beta = 1/8$ and combining the estimates \eqref{est:SF.A}--\eqref{est:SF.A^c} yields
for all $|z|\leq \Vmin^\frac12$
    \begin{align} 
        |(G \ast \nabla \mathcal R)(t,X(t) + z)|\lesssim \delta^{\frac 7 4} \Vmin^{-\frac12}.
    \end{align}
    Moreover, combining Propositions \ref{pro:G} and \ref{pro:gradR} yields
      $  |(G \ast \nabla \mathcal R)(t,X(t) + z)|\lesssim 1 $  for all $z \in \R^3$ .
           
   Combining these estimates  with the decay of $\phi$ as in Step 3 of the previous proof yields the assertion.
\end{proof}

\subsection{Estimate of \texorpdfstring{$\mathcal E^R_1$}{E}.}

In order to estimate $\mathcal E_1^R$, we will first provide separate estimates for $S_R - \bS$ and for $\bS - \SI$, where $\bS$ is defined in \eqref{def:bS} and where we denote for shortness $S_R = S_{R,X(T),V(T)}$.

\begin{lemma}\label{lem:S.R.S.bar}
    Under the bootstrap assumptions \eqref{eq:B1}--\eqref{eq:B3} with  $\delta_0,n > 0$ sufficiently small respectively large,   
      the function $S_R = S_{R,X(T),V(T)}$ defined in~\eqref{eq:S_R} can be estimated, for all $x \in \R^3$ and all $0 \leq t \leq T \leq R$ by 
    \begin{align}\label{eq:Strivial} 
        |S_{R}(t,x)|\lesssim \frac{1}{\Vmin \langle x^\perp \rangle^2}.
    \end{align}
Moreover, for all $x \in \R^3$      
       with $|x-X(T)|\leq \Vmin^\frac25 $ and all $0 \leq t \leq T \leq R$, we can estimate
    \begin{enumerate}
        \item For $t\geq 4 \Vmin^{-\frac35}$
        \begin{align} \label{eq:SRtlarge}
            |\bS(T-t,x)|+|S_R(R-t,x)| \lesssim e^{-c t \Vmin } .
        \end{align} 
        \item For $t\leq 4 \Vmin^{-\frac35}$
        \begin{align} \label{eq:SRtsmall} 
            |\bS(T-t,x)-S_R(R-t,x)| \lesssim \delta \Vmin^{-\frac{6}{5}}. 
        \end{align}
    \end{enumerate}
\end{lemma}
\begin{proof}
    We rewrite $\bS$ and $S_R$ as
    \begin{align}
        \bS(T-t,x)&=-\int_{-\infty}^{R-t}\int_{\Reals^3} \nabla \Phi (x-X(T-t)-(R-t-s)(v- V(T-t))) \cdot \nabla \mu(v)\ud{v} \ud{s} , \\
        S_R(R-t,x) &= -\int_{0}^{R-t}\int_{\Reals^3} \nabla \Phi (x-X(T)-(R-t-s)v+(R-s) V(T)) \cdot \nabla \mu(v)\ud{v} \ud{s} .
    \end{align}
    For $|v| \leq \Vmin/2$ and $s \leq R-t$, we have
    \begin{align}
       & |x-X(T-t)-(R-t-s)(v- V(T -t))| \\
        &\geq \left|(R-t-s) V(T-t) + \int_{T -t}^{T} V(\tau) \dd \tau \right| -  |x-X(T)| - |(R-t-s)v| \\
        &\geq \frac 1 2 (R-s) \Vmin -  |x-X(T)|,
    \end{align}
    and
    \begin{align}
            |x-X(T)-(R-t-s)v+(R-s) V(T)|  & \geq \frac 1 2 (R-s) \Vmin -  |x-X(T)|.
    \end{align}
    Since $|x-X(T)|\leq \Vmin^\frac25$ and $\supp \mu \subset B_{\frac{\Vmin}5}$, the integrands of both integrals above thus satisfy the bound
    \begin{align}
       |\nabla \Phi (x-X(T-t)-(R-t-s)(v- V(T-t))| &\lesssim e^{-c(R-s)\Vmin}, \quad \text{for $R-s\geq 4 \Vmin^{-\frac35}$}, \\
       |\nabla \Phi (x-X(T)-(R-t-s)v+(R-s) V(T))| &\lesssim e^{-c(R-s)\Vmin}, \quad \text{for $R-s\geq 4 \Vmin^{-\frac35}$}.
    \end{align}
    In particular, for $t\geq 4\Vmin^{-\frac35}$ we immediately \eqref{eq:SRtlarge}.
    On the other hand, for $s \leq R$
    \begin{align}
            &|x-X(T-t)-(R-t-s)(v- V(T -t)) - (x-X(T)-(R-t-s)v+(R-s) V(T))| \\
            &\leq \int_{T - t}^{T} |V(\sigma) - V(T - t)| \dd \sigma + (R-s) |V(T) - V(t-T)| \\
            &\leq t( t + R - s) \sup |\dot{V}|\leq  C \delta t( t + R - s),
    \end{align}
   where we used $|\dot{V}| \leq \|E\|_\infty \leq \delta$ by \eqref{apriori:E}.
    Therefore, if $t\leq 4\Vmin^{-\frac35}$, the difference is bounded by
    \begin{align}
        |\bS(T-t,x)-S_R(R-t,x)| \lesssim e^{-c \Vmin^\frac25} 
        + \int_{R-4\Vmin^{-\frac35}}^{R-t}\int_{\Reals^3} I_\Phi |\nabla \mu(v)|\ud{v} \ud{s},
    \end{align}
    where by Taylor expansion of $\nabla \Phi$, $I_\Phi$ is given by
    \begin{align}
        I_\Phi := \|\nabla^2 \Phi\|_{L^\infty} \delta (4 \Vmin^{-\frac35})^2 \sup \lesssim \delta \Vmin^{-\frac{6}5},
    \end{align}
    and the claim~\eqref{eq:SRtsmall} follows.
    Finally, the proof of~\eqref{eq:Strivial} follows analogous to~\eqref{est:SCharge}. 
\end{proof}

The following Lemma shows that $\bS$ is a good approximation for $\SI$. 
\begin{lemma} \label{lem:S.bar.S}
    Under the bootstrap assumptions \eqref{eq:B1}--\eqref{eq:B3} with  $\delta_0,n > 0$ sufficiently small, if  $T\geq 4 \Vmin^{-\frac35}$, we have for all $x \in \R^3$ with $|x-X(T)|\leq \Vmin^{\frac25}$ and all $0\leq t\leq T$
    \begin{enumerate}
        \item if $t\geq 4 \Vmin^{-\frac35}$ we have the following estimate
   \begin{align}\label{eq:BSSIlarge} 
       |\bS(T-t,x)|+|\SI(T-t,x)|&\lesssim e^{-ct \Vmin}.
   \end{align}
        \item if $t\leq 4 \Vmin^{-\frac35}$ we have the following estimate
   \begin{align}\label{eq:BSSIsmall}
       |\bS(T-t,x)-\SI(T-t,x)|&\lesssim \log(2+T) \delta  \Vmin^{-\frac{6}5}.
   \end{align}
   \end{enumerate} 
\end{lemma}
\begin{proof}
    The proof is largely analogous to the previous Lemma and we only detail the differences. 
    We first observe that for $|x-X(T)|\leq \Vmin^{\frac25}$ and $T-s\geq 4 \Vmin^{-\frac35}$, we have
    \begin{align} \label{est:arguments.Phi}
        | \nabla \Phi(x-X(T-t)-(T-t-s)(v-V(T-t)))| + |\nabla \Phi(X_{s,T-t}(x,v)-X(s))| \lesssim e^{-c(T-s)\Vmin},
    \end{align}
    and~\eqref{eq:BSSIlarge} follows as above. 
    
    It remains to show~\eqref{eq:BSSIsmall}.  Let $t\leq 4 \Vmin^{-\frac35}$. With the notation $\lambda = T-s$ and omitting arguments of $X_{s,\lambda}$ and $V_{s,\lambda}$, we split the error into 
   \begin{align} \label{eq:BS-SI}
       &|(\bS-\SI)(T-t,x)|\leq \left| \int_{-\infty}^0 \int_{\Reals^3} \nabla \Phi(x-X(\lambda )-(\lambda-s)(v-V(\lambda))) \cdot \nabla_v \mu(v) \ud{v} \ud{s}\right| \\ 
       +&\left| \int_{0}^{T-4 \Vmin^{-\frac35}} \int_{\Reals^3}\nabla \Phi(x-X(\lambda)-(\lambda-s)(v-V(\lambda)))\cdot  \nabla_v \mu(v)-\nabla \Phi(X_{s,\lambda }-X(s))\cdot \nabla_v \mu(V_{s,\lambda }) \ud{v} \ud{s} \right|\\
       +&\left| \int_{T-4 \Vmin^{-\frac35}}^{\lambda} \int_{\Reals^3}\nabla \Phi(x-X(\lambda)-(\lambda-s)(v-V(\lambda)))\cdot  \nabla_v \mu(v)-\nabla \Phi(X_{s,\lambda }-X(s))\cdot \nabla_v \mu(V_{s,\lambda }) \ud{v} \ud{s} \right|.
   \end{align}
 Relying on \eqref{est:arguments.Phi}, the first two lines can be estimated as before, by
  \begin{align*}
  &\left| \int_{0}^{T-4 \Vmin^{-\frac35}} \int_{\Reals^3}[\nabla \Phi(x-X(\lambda)-(\lambda-s)(v-V(\lambda))) \cdot \nabla_v \mu(v)-\nabla \Phi(X_{s,\lambda }-X(s)) \cdot \nabla_v \mu(V_{s,\lambda })] \ud{v} \ud{s} \right| \\
      +&\left| \int_{-\infty}^0 \int_{\Reals^3} \nabla \Phi(x-X(\lambda )-(\lambda-s)(v-V(\lambda))) \cdot \nabla_v \mu(v) \ud{v} \ud{s}\right|
             \lesssim e^{-c\Vmin^\frac25}. 
  \end{align*}
  For the last term in \eqref{eq:BS-SI}, we first take a closer look at the velocity integral. We integrate by parts
  \begin{align}
      &\int_{\Reals^3}\nabla \Phi(x-X(\lambda)-(\lambda-s)(v-V(\lambda)))\cdot \nabla_v \mu(v)-\nabla \Phi(X_{s,\lambda }-X(s)) \cdot \nabla_v \mu(V_{s,\lambda }) \ud{v} \\
      =& -\int_{\Reals^3}(\lambda-s)\left[\Delta \Phi(x-X(\lambda)-(\lambda-s)(v-V(\lambda))) \mu(v)-\Delta \Phi(X_{s,\lambda }-X(s)) \mu(V_{s,\lambda })\right] \ud{v}\\
      & +\int_{\Reals^3}\nabla \Phi(X_{s,\lambda }-X(s)) \cdot \nabla_v \tilde{W}_{s,\lambda }\mu(V_{s,\lambda }) \ud{v} -\int_{\Reals^3}\nabla_v \tilde{Y}_{s,\lambda }\nabla^2 \Phi(X_{s,\lambda }-X(s)) \cdot \nabla_v \mu(V_{s,\lambda })\ud{v}\\
      =&:I_1+I_2+I_3,
  \end{align}
  and estimate $I_1$, $I_2$, $I_3$ separately. 
 For $I_1$ we use again $|\dot V(s)|\lesssim \delta$
  as well as the estimates from Proposition~\ref{pro:char.front} and Corollary \ref{co:char.tilde} to deduce that, for $|\lambda-s|\leq 4 \Vmin^{-\frac35}$, we have 
  \begin{align}
      &|x-X(\lambda)-(\lambda-s)(v-V(\lambda))-(X_{s,\lambda }-X(s))| + |v  - V_{s,\lambda}|   \lesssim \delta .
  \end{align}
  This yields the bound 
  \begin{align}
      |I_1| \lesssim \|\nabla^3 \Phi\|_{L^\infty(\Reals^3)}\Vmin^{-\frac{3}{5}} \delta .
  \end{align}

  For $I_2$, $I_3$ we observe that $|x-X(T)|\leq \Vmin^{\frac25}$ and $t \leq 4 \Vmin^{- 3/5}$ implies $|\check{\mathcal{T}}_{\lambda,x,v}|\lesssim \Vmin^{-\frac35} $ due to \eqref{est.nabla_x.T}. Combining this with Corollary~\ref{co:char.tilde} and Proposition~\ref{pro:char.front} we obtain for $\lambda - s \leq \Vmin^{-3/5}$ 
  \begin{align}
      |I_2|+|I_3|\lesssim \log (2 + T) \delta  \Vmin^{-\frac{3}5}.
  \end{align}
  Using these estimates in the last term in \eqref{eq:BS-SI} finishes the proof.
\end{proof}

Inserting the estimates from Lemmas \ref{lem:S.R.S.bar} and \ref{lem:S.bar.S} into the definition of the error term $\mathcal{E}_1$ (cf. \eqref{def:E1}) yields the following estimate.
\begin{corollary}\label{cor:E1}  Under the bootstrap assumptions \eqref{eq:B1}--\eqref{eq:B3} with  $\delta_0,n > 0$ sufficiently small, and if $ T \geq 4\Vmin^{-\frac35}$, we have for all $R \geq T$
    \begin{align}
        |\mathcal{E}^R_1(T)| \lesssim \delta \log (2 + T)  \Vmin^{-\frac{6}5}.
    \end{align}
\end{corollary}
\begin{proof}
   We split $\mathcal E_1^R$ into
   \begin{align}
      \mathcal E_1^R &=  \mathcal E_1^1 + \mathcal E_1^2,\\
      \mathcal E_1^1(T)  &:= (\nabla \phi * \SI)(T,X(T)) - (\nabla \mathcal \phi *(S_{R}))(R,X(T)),\\
     \mathcal E_1^2(T) &:=  ( \nabla\phi *  (G*\SI))(t,X(T)) - ( \nabla\phi *( G \ast S_{R}))(R,X(T)).
   \end{align}
   Then the desired estimate for $\mathcal E^1_1$ follows directly from the decay of $\phi$ and Lemmas \ref{lem:S.R.S.bar} and \ref{lem:S.bar.S} applied with $t=0$.
   
   To estimate $\mathcal E_1^2$, we write $S= \SI - \bS $ and first observe that we can split the convolution as
   \begin{align} \label{est:E_1^2.1}
   \begin{aligned}
      (\nabla \phi *  (G* S))(T,X(T)) &= \int_0^{4 \Vmin^{-\frac 35}} ((\nabla \phi) * G(t,\cdot)* S(T - t, \cdot))(X(T)) \dd t\\
      &+ \int_{4 \Vmin^{-\frac 35}}^{\infty}  (\phi * (\nabla G(t,\cdot))* S(T - t,\cdot))(X(T)) \dd t.
   \end{aligned}
   \end{align}
   
   Denoting $B =B_{\frac 1 2 \Vmin^{\frac 2 5}}(0)$, using Proposition \ref{pro:S.ion} and Lemma \ref{lem:S.bar.S} as well as Proposition  \ref{pro:G} we estimate for $|x| \leq  \frac  1 4 \Vmin^{\frac 2 5} $
   \begin{align}\label{est:E_1^2.2}
      \begin{aligned} 
         \int_0^{4 \Vmin^{-\frac 35}} |(G(t)* S(T - t))(X(T) &-x)| \dd t \leq \int_0^{4 \Vmin^{-\frac 35}} \int_{B} |G(t,y)| |S|(T-t,X(T) - x-y) \dd y \dd t\\
         &+\int_0^{4 \Vmin^{-\frac 35}} \int_{B^c} |G(t,y)| |S|(T-t,X(T)-x-y)  \dd y \dd t\\
         &\lesssim \int_0^{4 \Vmin^{-\frac 35}} \delta \log(2+T) \Vmin^{- \frac {6}5} \dd t + \int_0^{4 \Vmin^{-\frac 35}}\int_{B^c} \frac1 {|y|^4} \frac 1 {\Vmin |y^\perp|^2}  \dd y\dd t\\
         &\lesssim \log(2+T)\delta \Vmin^{- \frac {7}5} .
   \end{aligned}
     \end{align}
   Moreover, relying on the pointwise estimates for $\nabla G$ from Proposition  \ref{pro:G}, we find
     \begin{align}\label{est:E_1^2.3}
   \begin{aligned}
         \int_{4 \Vmin^{-\frac 35}}^{\infty} |(\nabla G(t,\cdot)* S(T - t,\cdot))(X(T) -x)| \dd t &\leq \int_{4 \Vmin^{-\frac 35}}^{T} \int_{B} |\nabla G(t,y)| |S|(T-t,X(T) - x-y) \dd y \dd t\\
         &+\int_{4 \Vmin^{-\frac 35}}^{\infty} \int_{B^c} |\nabla G(t,y)| |S|(T-t,X(T)-x-y) \dd y \dd t \\ 
         &\lesssim \int_{4 \Vmin^{-\frac 35}}^{\infty} \int_{\R^3} \frac{1}{|y|^5 + t^5} e^{-c t \Vmin} \dd y \dd t \\ &+\int_{0}^{\infty}\int_{B^c} \frac{1}{|y|^5 + t^5} \frac{1}{\Vmin \langle y_\perp\rangle^2} \dd y \dd t \lesssim \Vmin^{- \frac {11}{5}}.
   \end{aligned}
      \end{align}
   For $|x|\geq \frac  1 4 \Vmin^{\frac 2 5}$, Propositions  \ref{pro:S.ion} and \ref{pro:G} imply
   \begin{align} 
       \int_0^{4 \Vmin^{-\frac 35}} |(G(t,\cdot)* S(T - t,\cdot))(X(T) -x)| \dd t &\lesssim 1 , \label{est:E_1^2.4}\\
       \int_{4 \Vmin^{-\frac 35}}^{\infty} |( \nabla G(t,\cdot)* S(T - t,\cdot))(X(T) -x)|\dd t &\lesssim 1. \label{est:E_1^2.5}
   \end{align}
   \noeqref{est:E_1^2.4} \noeqref{est:E_1^2.3}
   Inserting \eqref{est:E_1^2.2}-\eqref{est:E_1^2.5} into \eqref{est:E_1^2.1} and using the exponential decay of $\phi$ yields
\begin{align} \label{est:E_1^2.6}
      |(\nabla \phi *  (G* (\SI - \bS)))(T,X(T))| &\lesssim \delta \log(2+T) \Vmin^{- \frac {6}{5}}.
   \end{align}
   Similarly, relying on Lemma \ref{lem:S.R.S.bar} yields
   \begin{align} \label{est:E_1^2.7}
      |(\nabla \phi *  G*\bS)(T,X(T)) - (\nabla \phi *  G*S_R)(R,X(T))| &\lesssim \delta \Vmin^{- \frac {6}{5}}.
   \end{align}
   Combining \eqref{est:E_1^2.6}-\eqref{est:E_1^2.7} yields the desired bound for $\mathcal E_1^2(T)$ which concludes the proof.
\end{proof}

\subsection{Proof of Proposition~\ref{pro:bootstrap}(\ref{case:B})} \label{subsec:PropB}

We recall the identities~\eqref{eq:rho.g.S}--\eqref{eq:rhoRep} and \eqref{id:hS} to rewrite
\begin{align}
    \lim_{s \to \infty} |(\nabla \phi *\rho[h_{V(T)}])(s,0) + E(T,X(T))| \leq \sup_{R \geq T} |\mathcal{E}_1^R(T)| +|\mathcal{E}_2(T)| + |\mathcal{E}_3(T)|  .
\end{align}
Now it remains to apply Corollary~\ref{cor:E1} for $\mathcal{E}^R_1$, Lemma~\ref{lem:E2}  for $\mathcal{E}_2$ and Lemma~\ref{lem:E3} for $\mathcal{E}_3$. Since we assume $T\geq 4\Vmin(T)^{-\frac35}$ and we have $\Vmin^{-1} \leq \delta$, we obtain 
\begin{align}
    \sup_{R \geq T} |\mathcal{E}_1^R(T)| +|\mathcal{E}_2(T)| + |\mathcal{E}_3(T)| \lesssim \delta \log(2+T) \Vmin^{-\frac65}+ \delta^2\log(2+T) \Vmin^{-\frac12}+ \delta^\frac72 \Vmin^{-\frac12} \lesssim \delta^\frac{11}{5} \log(2+T).
\end{align}
Hence, by a suitable choice of $\delta_0$ and $n$, from  \eqref{eq:B2} we deduce
\begin{align}
     \sup_{R \geq T} |\mathcal{E}_1^R(T)| +|\mathcal{E}_2(T)| + |\mathcal{E}_3(T)| \leq C \delta^{\frac{13}6},
\end{align}
which proves the claim.

%% file: 7.LinearFriction.tex
\begin{proof}[Proof of Proposition \ref{pro:force.linear}]
Fix $ V_\ast \in \R^3$, and recall the defining equation for $h = h_{V_\ast}$ from \eqref{eq:TravelingWave}:
   \begin{align}
       \partial_s h + (v-V_\ast) \cdot \nabla_x h - \nabla (\phi*_x \rho[h]) \cdot \nabla_v \mu &= - e_0\nabla \Phi(x) \cdot \nabla_v \mu, \quad h(0,\cdot) = 0.
   \end{align}
   We extend $h$ by zero for negative times. 
   The equation for $h$ can be explicitly solved in space-time Fourier variables. Let $\tilde{h}(z,k,v)$ be given according to~\eqref{eq:FourierTF}, then
   \begin{align}
      (\tau  + k\cdot (v-V_\ast))  \tilde{h}- \hat{\phi}(k) \rho[\tilde{h}](\tau,k)k\cdot \nabla_v \mu = \frac{-e_0 \hat{\Phi} (k) k \cdot\nabla_v \mu }{i\tau },
   \end{align}
   for negative imaginary part, $\Im(\tau)<0$. This yields the explicit representation 
   \begin{align}
       \rho[\tilde{h}](\tau,k) &= \frac{-e_0\hat{\Phi}(k)}{i\tau \eps(\tau,|k|,\hat{k}\cdot V_*)} \int_{\Reals^3}\frac{k \cdot \nabla_v \mu(v)}{\tau + k \cdot (v-V_\ast)} \ud{v} = \frac{-e_0\hat{\Phi}(k)}{i\tau \eps(\tau,|k|,\hat{k}\cdot V_*)} \frac{1- \eps(\tau,|k|,\hat{k}\cdot V_*)}{\hat \phi(|k|)},
   \end{align}
   where $\hat{k} = \tfrac{k}{|k|}$, $k\neq 0$ and the dielectric function $\eps(\tau,|k|,\hat{k}\cdot V_*)$ is given by
   \begin{align}
       \eps(\tau,r,\hat{k}\cdot V_*)=1-\hat{\phi} (r)\int_{\Reals^3}\frac{\hat{k} \cdot \nabla_v \mu(v)}{\tau/r+\hat{k} \cdot (v-V_\ast)} \dd v .
   \end{align}
 Notice that the integral indeed only depends on $V_*$ and $\hat k$ through $\hat{k}\cdot V_*$ since by~\eqref{eq:monotone}
   \begin{align}
      \int_{\Reals^3}\frac{\hat{k} \cdot \nabla_v \mu(v)}{\tau/r+\hat{k} \cdot (v-V_\ast)} \dd v = \int_{\Reals^3}\frac{-(\hat{k} \cdot v) \psi(v)}{\tau/r+\hat{k} \cdot v- \hat{k} \cdot V_\ast} \dd v,
   \end{align}
   and $\psi$ is radially symmetric by Assumption~\ref{Ass:Radial}.
   We remark, that by elementary computation $\eps$ and $a$~(cf.~\eqref{eq:a}) are related by  
   \begin{align}\label{eq:aeps}
       \eps(\tau,|k|,\hat{k}\cdot V_*)=1- \hat{\phi}(k) a(\tau/|k|-\hat{k}\cdot V_*).
    \end{align}
   The Penrose condition~\eqref{eq:Penrose}, and Assumption~\ref{Ass:Radial} then ensure a uniform bound for $|\eps|$
   \begin{align}\label{eq:epsbound} 
       0<\kappa \leq |\eps|\leq C.
   \end{align}
   
   We now compute the limit $s\rightarrow \infty$ of the associated force.
   Using Lemma \ref{lem:Laplace} yields
   \begin{align}
       \lim_{s \rightarrow \infty}(\rho[h(R,\cdot)]* \nabla \phi)(0)&= \lim_{s \rightarrow \infty} \frac{1}{(2\pi)^3}\int_{\Reals^3} i k \hat{\rho}[ h(s,\cdot)] \hat{\phi} \ud{k} 
       = \lim_{i\tau \rightarrow 0^+} \frac{i\tau }{(2\pi)^3}\int_{\Reals^3} i k \rho[\tilde{h}](\tau ,k) \hat{\phi} \ud{k}\\
       &= \lim_{i\tau  \rightarrow 0^+}\frac{e_0}{(2\pi)^3} \int_{\Reals^3} i k \hat{\Phi}(k) \ud{k}-\lim_{i\tau \rightarrow 0^+} \frac{e_0}{(2\pi)^3}  \int_{\Reals^3} \frac{i k  \hat{\Phi}(k)}{\eps(\tau,|k|,\hat{k}\cdot V_*)} \ud{k}.
   \end{align}
   The first term vanishes since $\hat{\Phi}(k)=\hat{\Phi}(-k)$, and we can simplify 
   \begin{align} \label{eq:force.0}
   \begin{aligned}
      \lim_{s \rightarrow \infty}(\rho[ h(s,\cdot)]* \nabla \phi)(0)&=\lim_{i\tau \rightarrow 0^+}  \frac{-e_0}{(2\pi)^3}\int_{\Reals^3} \frac{i k  \hat{\Phi}(k)}{\eps(\tau,|k|,\hat{k}\cdot V_*)} \ud{k}\\
      &=\lim_{i\tau  \rightarrow 0^+} \frac{-e_0}{(2\pi)^3} \int_{\Reals^3} \frac{i k  \hat{\Phi}(k)\eps^*(\tau,k)}{|\eps(\tau,|k|,\hat{k}\cdot V_*)|^2} \ud{k}.
      \end{aligned}
   \end{align}
   By rotational symmetry of the potential $\phi$, $\hat \phi$ is real. Thus, by Plemelj's formula, Lemma \ref{App:Plemelj}, for $k \neq 0$,
   \begin{align}
       \lim_{i\tau  \to 0^+} \Im \eps^*(\tau,|k|,\hat{k}\cdot V_*) &= \hat \phi(k) \lim_{i\tau  \to 0^+} \Im \int_{\R^3} \frac{k \cdot \nabla_v \mu(v)}{k\cdot (v-V_\ast) + \tau} \dd v \\
      &=  \hat \phi(k) \lim_{i\tau  \to 0^+} \Im \int_{\{w \cdot k = 0\}} \int_{\R} \frac{k \cdot \nabla_v \mu(V_\ast + \lambda \hat k + w)}{\lambda |k| + \tau } \dd \lambda \dd w \\
      &= \pi \hat \phi(k) \int_{\{ w \cdot k = 0\}}  \hat k \cdot  \nabla_v \mu(V_\ast  + w) \dd w  .
   \end{align}
  By the radial symmetry of the potential $\Phi$, $\hat \Phi$ is real. Since the left hand side of \eqref{eq:force.0} is real, we can simplify the above to
  \begin{equation}\label{eq:sToinfty}
  \begin{aligned}
      \lim_{s \rightarrow \infty}(\rho[h(s,\cdot)]* \nabla \phi)(0)&=\lim_{i\tau  \rightarrow 0^+} \frac{-e_0}{(2\pi)^3} \int_{\Reals^3} \frac{i k  \hat{\Phi}(k)\eps^*(\tau,k)}{|\eps(\tau,|k|,\hat{k}\cdot V_*)|^2} \ud{k}\\
      &= \lim_{i\tau \rightarrow 0^+}\frac{e_0}{8\pi^2} \int_{\Reals^3} \frac{k  \hat{\Phi}(k)\hat{\phi} (k)\int_{\{k\cdot v=k\cdot V_\ast\}} \hat{k}\cdot  \nabla_v \mu(v)}{|\eps(\tau,|k|,\hat{k}\cdot V_*)|^2} \ud{k}.
   \end{aligned}
   \end{equation}
   Recall Assumption~\ref{Ass:monotone}, i.e. we have 
   \begin{align}
       \nabla_v \mu(v) = -v \psi(v),
       \end{align}
 for some non-negative, continuous, exponentially decaying, positive function $\psi$. This finally yields
   \begin{align} \label{eq:sLimit}
       \lim_{s \rightarrow \infty}e_0(\rho[h(s,\cdot)]* \nabla \phi)(0)\cdot V_\ast&= -\frac{e_0^2}{8\pi^2} \int_{\Reals^3} \frac{  \hat{\Phi}(k)|k|\hat{\phi}(k)(\hat{k}V_\ast)^2 }{|\eps(-i0^+,|k|,\hat{k}\cdot V_*)|^2}\varphi(\hat{k}\cdot V_\ast) \ud{k},
   \end{align}
   where $\varphi(u)$ is a non-negative, continuous,  exponentially decaying function given by
   \begin{align}\label{eq:varphi} 
       \varphi(u) = \int_{\{e_1 \cdot v=u\}} \psi(v) \ud{v}.
   \end{align}
   Since $\psi$ is radial, non-negative and not everywhere vanishing, we also have $\varphi(0)>0$. In particular, since $\hat \phi$ and $\hat \Phi$ are both positive (cf. Assumption \ref{ass:phi}) \eqref{linearForce} holds, i.e. 
   \begin{align}
        \lim_{s \rightarrow \infty} e_0(\rho[h(s,\cdot)]* \nabla \phi)(0)\cdot V_\ast < 0.
   \end{align}
   It remains to determine the asymptotics of the integral for $|V_*|\rightarrow \infty$. 
   We rewrite the integral in terms of the variable $u=\hat{k} \cdot \hat{V}_*$. Multiplying with $|V_*|$ we obtain
   \begin{align}
   \lim_{s \rightarrow \infty} e_0|V_*| (\rho[h(s,\cdot)]* \nabla \phi)(0)\cdot V_\ast&= -\frac{e_0^2|V_*|}{4\pi} \int_0^\infty  \int_{-1}^1 \frac{  \hat{\Phi}(r)r^3\hat{\phi}(r)(u |V_*|)^2 }{|\eps(-i0^+,r,u|V_*|)|^2}\varphi(u|V_*|) \ud{r}\ud{u} \\
   &= -\frac{1}{4\pi} \int_0^\infty  \int_{-|V_*|}^{|V_*|} \frac{  \hat{\Phi}(r)r^3\hat{\phi}(r)U^2 }{|\eps(-i0^+,r,U)|^2}\varphi(U) \ud{r}\ud{U}.
   \end{align} 
  The integral converges exponentially fast to a positive limit for $|V_*|\rightarrow \infty$. This establishes \eqref{eq:linearbounds}.
\end{proof}

\begin{remark}
The friction force is related to the Balescu-Lenard correction of the Landau equation. More precisely, consider the case $\phi=\Phi$ in~\eqref{eq:sToinfty}. We obtain 
\begin{align}
    \lim_{s \rightarrow \infty}e_0(\rho[h(s,\cdot)]* \nabla \phi)(0) = \frac{-1}{8\pi^2} \int_{\Reals^3} \frac{k  |\hat{\phi} (k)|^2\int_{\{k\cdot v=k\cdot V_\ast\}} \hat{k}\cdot  \nabla_v \mu(v)}{|\eps(-i0^+,|k|,\hat{k}\cdot V_*)|^2} \ud{k}\\
    = \frac{-1}{8\pi^2} \int_{\Reals^3}\int_{\Reals^3} \frac{ \delta(k\cdot(v-v_*))|\hat{\phi} (k)|^2(k\otimes k)  \cdot  \nabla_v \mu(v)}{|\eps(-i0^+,|k|,\hat{k}\cdot V_*)|^2} \ud{k} \ud{v},
\end{align}
which gives the friction coefficient of the Balescu-Lenard equation
\begin{align}
    \partial_t G &= \operatorname{LB} (G), \\
    \operatorname{LB} (G)(v) &= \nabla_v \cdot \left( \int_{\Reals^3} \int_{\Reals^3} B(v,v-v_*;\nabla G) (\nabla GG_*-G \nabla_* G_*) \ud{v_*}\right),  \\
    B(v,v-v_*;\nabla G) &= \int_{\Reals^3} \frac{ \delta(k\cdot(v-v_*))|\hat{\phi} (k)|^2(k\otimes k) }{|\eps(-i0^+,|k|,\hat{k}\cdot v_*;\nabla G)|^2} \ud{k}\\
    \eps(\tau,r,\hat{k}\cdot V_*;\nabla G)&=1-\hat{\phi} (r)\int_{\Reals^3}\frac{\hat{k} \cdot \nabla_v G(v)}{\tau/r+\hat{k} \cdot (v-V_\ast)} \dd v. 
\end{align}
The equation was formally derived in~\cite{balescu_1960,lenard_1960}, for a recent well-posedness result see~\cite{duerinckx_2021}.
Notice that we recover the Landau equation from the Balescu-Lenard equation when we neglect collective effects, i.e. replace  $\eps \equiv1$. 
\end{remark}

%% file: appendix.tex
\section{Proof of Proposition~\ref{prop:penrose}} \label{App:B}
\begin{proof}[Proof of Proposition~\ref{prop:penrose}]
By Assumptions~\ref{ass:phi} and~\ref{Ass:Radial}, the function $a(z)$ defined in \eqref{eq:a} decays for $|z|\rightarrow \infty$, $\Im(z)\leq 0$. Therefore the infimum in~\eqref{eq:Penrose} can be replaced by a minimum. This allows us to argue by contradiction. For  $\overline{C}>0$ given, assume there exist $\xi^*\in \R^3$, $\Im(z^*)\leq 0$ such that
\begin{align}\label{eq:zstar}
    a(z^*)=(\hat{\phi}(k))^{-1}>1 .
\end{align}

As in the proof of Proposition 2.7 in~\cite{bedrossian_landau_2018}, we use Penrose's argument principle (cf.~\cite{penrose_electrostatic_1960}):
the function $z\mapsto a(z)$ is a holomorphic function on the lower half plane, vanishing for $|z|\rightarrow \infty$. The boundary behavior of the function is given by the curve $\gamma: \Reals \rightarrow \Complex$ given by
\begin{align}
    \vec{\gamma}(x)= a(x-i0):= \lim_{\eps\rightarrow 0} a(x-i\eps). 
\end{align}
By the argument principle, \eqref{eq:zstar} can only hold if the curve $\vec{\gamma}$ intersects the half-line $\{y\in \R: y>1\}$. 

\medskip

Writing $\mu(v)=\mu(|v|)$ by slight abuse of notation, we have the following representation (cf.~\cite{bedrossian_landau_2018} Appendix and~\cite{mouhot_landau_2011} Section 3)
\begin{align}
    \vec{\gamma}(x) =  \operatorname{PV} \int_{\Reals} \frac{-2\pi u \mu(|u|)}{u-x} \ud{u} - i 2  \pi^2  u \mu(|u|) .
\end{align}
By Assumption~\ref{Ass:Radial}, there exists $\overline{C}>0$ such that
\begin{align}
    \left| \operatorname{PV} \int_{\Reals} \frac{-2\pi u \mu(|u|)}{u-x} \ud{u}\right| <\frac12, \quad |x|\geq \overline{C}.
\end{align}
Now it suffices to observe that the imaginary part does not vanish if $\mu(v)>0$ for  $|v|\leq \overline{C}$. This contradicts the assumption for $\overline{C}$ large enough and finishes the proof.
\end{proof}

\section{Two standard auxiliary Lemmas}\label{App:A}

In this section, we recall two standard results which we use to compute the linearized force in Section~\ref{sec:linear.force}.

\begin{lemma} \label{lem:Laplace}
	Assume $f \in C^1_b(\R)$, $f=0$ in $(-\infty,0]$ and let $\tilde f$ be it's Fourier transform. Then, 
	\begin{align}
		\lim_{t \to \infty} f(t) = \lim_{z \downarrow 0} z \tilde f(-iz),
	\end{align}
	whenever the limit on the right-hand side exists.
\end{lemma}
\begin{proof}
	Provided the right-hand side above exists, we have
	\begin{align}
		\lim_{z \downarrow 0} z \tilde f(-iz) &= \lim_{z \downarrow 0} \int_0^\infty f(t) z e^{-zt} \dd t 
		 = \lim_{z \downarrow 0} \int_0^\infty f'(t)  e^{-zt} \dd t - \left[ f e^{- zt} \right]_0^\infty  \\
		 &= \lim_{z \downarrow 0} \int_0^\infty f'(t)  e^{-zt} \dd t - \left[ f e^{- zt} \right]_0^\infty 
		= \int_0^\infty f'(t) \dd t - f(0) 
		= \lim_{t \to \infty} f(t).
	\end{align}
	as claimed.
\end{proof}

\begin{lemma}[Plemelj's formula, e.g.~\cite{muskhelishvili_1958}]\label{App:Plemelj}
	For $f\in L^2(\Reals)\cap C^1(\Reals)$ we have the identity
	\begin{align}
		\lim_{\delta \rightarrow 0^+} \int_\Reals \frac{f(y)}{(x-y)\pm i\delta} \ud{y}=\mp i \pi f(x) + \lim_{\delta \rightarrow 0^+} \int_{\{|x-y|\geq \delta\}} \frac{f(y)}{x-y} \ud{y} .
	\end{align}
\end{lemma}